\documentclass[11pt,leqno,amscd,amssymb,verbatim, url]{amsart}
\usepackage{amsfonts,amssymb,graphicx}
\usepackage{amsrefs}
\usepackage{parskip}
\usepackage{amsmath,amscd,mathtools}
\usepackage{hyperref}
\usepackage[dvipsnames]{xcolor}
\newcommand{\Lie}{\operatorname{Lie}}
\newcommand{\Ext}{\operatorname{Ext}}

\newcommand{\opH}{\operatorname{H}}
\newcommand{\Hom}{\operatorname{Hom}}

\DeclareMathOperator{\Ind}{Ind}
\newcommand{\gl}{\mathfrak{g}}
\newcommand{\B}{\mathrm{B}}
\newcommand{\T}{\mathrm{T}}

\theoremstyle{definition}

\newtheorem{theorem}{Theorem}[subsection]

\makeatletter\let\c@question\c@theorem\makeatother

\makeatletter\let\c@fact\c@theorem\makeatother

\makeatletter\let\c@note\c@theorem\makeatother

\newtheorem{lemma}{Lemma}[subsection]
\makeatletter\let\c@lemma\c@theorem\makeatother

\makeatletter\let\c@alg\c@theorem\makeatother

\newtheorem{remark}{Remark}[subsection]
\makeatletter\let\c@remark\c@theorem\makeatother

\makeatletter\let\c@example\c@theorem\makeatother

\newtheorem{prop}{Proposition}[subsection]
\makeatletter\let\c@prop\c@theorem\makeatother

\makeatletter\let\c@conj\c@theorem\makeatother

\newtheorem{cor}{Corollary}[subsection]
\makeatletter\let\c@cor\c@theorem\makeatother

\makeatletter\let\c@defn\c@theorem\makeatother

\numberwithin{equation}{subsection}

\usepackage[capitalise]{cleveref}
\crefname{theorem}{Theorem}{Theorems}
\crefname{fact}{Fact}{Facts}
\crefname{question}{Question}{Question}
\crefname{note}{Note}{Notes}
\crefname{lemma}{Lemma}{Lemmas}
\crefname{alg}{Algorithm}{Algorithms}
\crefname{remark}{Remark}{Remarks}
\crefname{example}{Example}{Examples}
\crefname{prop}{Proposition}{Propositions}
\crefname{conj}{Conjecture}{Conjectures}
\crefname{cor}{Corollary}{Corollaries}
\crefname{defn}{Definition}{Definitions}
\crefname{equation}{}{}

\title{On the cohomology of the Ree groups and kernels of exceptional isogenies}

\begin{document}
\author{Aura-Cristiana Radu}
\address{School of Mathematics, Statistics and Physics,
Herschel Building,
Newcastle,
NE1 7RU, UK}
\email{a.radu2@newcastle.ac.uk}


\begin{abstract}
Let $G$ be a simple, simply connected algebraic group over an algebraically closed field $k$ of characteristic $p>0$. Let $\sigma : G \rightarrow G$ be a surjective endomorphism of $G$ such that the fixed point set $G(\sigma)$ is a Suzuki or Ree group. Then, let $G_{\sigma}$ denote the scheme-theoretic kernel of $\sigma.$ Using methods of Jantzen and Bendel-Nakano-Pillen, we compute the $1$-cohomology for the Frobenius kernels with coefficients in the induced modules,  $\opH^{1}(G_{\sigma}, \opH^{0}(\lambda))$, and the $1$-cohomology for the Frobenius kernels with coefficients in the simple modules,  $\opH^{1}(G_{\sigma}, L(\lambda))$ for the Suzuki and Ree groups. Moreover, we improve the known bounds for identifying extensions for the Ree groups of type $F_4$ with the ones for the algebraic group.
\end{abstract}

\parindent=0pt

\maketitle

\section{Introduction}

Let $G$ be a simple, simply connected algebraic group over an algebraically closed field $k$ of characteristic $p>0$. Then, for a strict endomorphism $\sigma : G \rightarrow G$, the fixed point set of the points, $G(\sigma):=G(k)^{\sigma}$, is a finite group. Moreover, the scheme-theoretic kernel of $\sigma$ is an infinitesimal subgroup of $G$ and we denote it by $G_{\sigma}$. The study of cohomology of finite groups of Lie type has been of great interest throughout the years, as it encapsulates crucial information regarding the category of $kG^{\sigma}$-modules. In particular, one aspect of this broader topic is the computation of non-split extensions between simple modules. 

The groundbreaking work of Cline, Parshall, Scott and van der Kallen \cites{CPS, CPSvdK} relates rational cohomology to the cohomology of finite groups. Further work by Andersen \cite{And} then provides a general approach for Chevalley groups, with restrictions on the minimal bound on the characteristic $p$. However, since the cases of small values of $p$ could not be tackled using this construction, a mixture of techniques arose, characterised by the fact that they relied on specific information concerning the groups and root systems. (See \cite[Chapter 12]{Hum06} for a literature review.) 

In a series of papers \cites{Sin92, Sin93, Sin94a, Sin94b}, the $1$-cohomology for the Suzuki-Ree groups was considered. In particular, in \cites{Sin94a, Sin94b}, Sin computed the $1$-cohomology for the algebraic group of type $F_4$ in characteristic $2$.

Bendel-Nakano-Pillen take a different approach, building on \cites{CPS, CPSvdK}, in which extensions for the finite group are compared to extensions for the ambient algebraic group using $\Ind_{G(\sigma)}^G k$; moreover, this approach required passage to Frobenius kernels. Initially, the Chevalley groups were considered in \cite{BNP04b}, for which the strict endomorphism $\sigma : G \to G$ is $\sigma = F^r$, the composition of the Frobenius map with itself $r$ times, and the scheme-theoretic kernel is $G_{\sigma}=G_r$. This was followed by \cite{BNP06}, where the authors provide some analogues for the twisted groups; these groups are characterised by the existence of a non-trivial graph automorphism $\theta$ such that $\sigma=F^r \circ \theta$ and in this case the scheme-theoretic kernel $G_{\sigma}=G_r$. We shall henceforth refer to both of these types of kernels as classical Frobenius kernels. Due to the existence of Sin's results concerning the Suzuki and Ree groups, they did not provide analogues using their method. 

This paper aims to fill some gaps in the literature: first, to provide the explicit description of the $1$-cohomology for the scheme-theoretic kernels with coefficients in the induced modules and with coefficients in the simple modules for the Suzuki and Ree groups; second, to improve the known bounds for identifying extensions for the Ree groups of type $F_4$ with the ones for the algebraic group.

We describe the structure of the paper. In Section \ref{Sect2} we fix some notation and remind the reader of certain facts regarding the structure of the Suzuki and Ree groups. In particular, for $G$ with root system $\Phi$, in cases $(\Phi,p)=(C_2,2)$, $(G_2, 3)$ or $(F_4,2)$, there exists a fixed purely inseparable isogeny $\tau : G \to G$ whose square is the Frobenius map. Then the strict endomorphism $\sigma$ is given by $\sigma = \tau^r=F^{r/2}$, for an odd positive integer $r$. In these cases, following \cite[3.3]{BT}, we shall to refer to $\sigma$ as an exceptional isogeny. Thus the fixed point set under $\sigma$ becomes a Suzuki-Ree group and we denote the scheme-theoretic kernel $G_{\sigma}$ by $G_{r/2}$. To differentiate it from the classical case, we call this infinitesimal subgroup of $G$ an exotic or half Frobenius kernel. 

In Section \ref{Sect3} we compute the $1$-cohomology for the exotic Frobenius kernels with coefficients in the induced modules, $\opH^1(G_{r/2},\opH^0(\lambda))$, for the Suzuki groups (Subsection \ref{subsecC2}), the Ree groups of type $G_2$ (Subsection \ref{subsecG2}) and of type $F_4$ (Subsection \ref{subsecF4}). Moreover, we calculate the $1$-cohomology for the classical (Theorem \ref{c2gs-simple}, Theorem \ref{g2gs-simple}, Theorem \ref{f4gs-simple}) and exotic Frobenius kernels with coefficients in simple modules (Theorem \ref{c2gr2-simple}, Theorem \ref{g2gr2-simple}, Theorem \ref{f4gr2-simple}). 

Then, in Section \ref{Sect4}, we focus on the Ree groups of type $F_4$. We consider a certain truncation of $\Ind_{G(\sigma)}^G k$ and relate the finite group cohomology to the algebraic group cohomology.
In Section \ref{subsec41}, we precisely bound the weights in our truncated category (see Lemma \ref{bnp52}), performing many spectral sequence computations involving half Frobenius kernels, instead of the classical ones. Thus, we ensure the sharpness of our bound on the size of the finite group using these methods. We observe, rather surprisingly, that the Ree groups of type $F_4$ exhibit very different behaviour compared to the other finite groups of Lie type (be it Chevalley, twisted or, indeed, Suzuki or Ree groups of type $G_2$). 

In order to see this, first recall some of the terminology used in \cite{BNP04b}, \cite{BNP+12} and \cite{PSS13}. Let $\mathcal{C}_t$ be the full subcategory of all finite-dimensional $G$-modules whose composition factors $L(\nu)$ have highest weights in the set $\pi_t=\lbrace \nu \in X_{+} : \langle \nu, \alpha_0^{\vee} \rangle < t \rbrace.$ The weight $\nu \in \pi_t$ is $(t-1)$-small. 

Now, let $\sigma$ denote the appropriate strict endomorphism, as discussed above. In Remark \ref{different}(a), we observe that in the case $G=F_4$, $p=2$, $\sigma=F^{r/2}$ for $r$ odd, $\Ext^1_{G_{\sigma}}(L(\lambda), L(\mu))^{(-\sigma)}$ is a rational $G$-module whose weights are $(h+4)$-small. This is in contrast to \cite[Theorem 2.3.1]{BNP+12},  which states that for all $(G, p, \sigma)$ aside from the case we consider, one has that $\Ext^1_{G_{\sigma}}(L(\lambda), L(\mu))^{(-\sigma)}$, for $\lambda, \mu \in X_{\sigma}$ is a rational $G$-module whose weights are $(h-1)$-small. This comes as a surprise, given the fact that similar methods were used.  

In Section \ref{subsec42}, we turn our attention to finite group extensions. We find that self-extensions between simple $kG(\sigma)$-modules vanish, provided $r \geq 15$. (Theorem \ref{self}). Finally, in Theorem \ref{ext}, we conclude that, for $r \geq 15$, the $\Ext^1$ group between simple $kG(\sigma)$-modules is isomorphic to the $\Ext^1$ group between a specific pair of $\sigma$-restricted simple $G$-modules (which depends on the pair of $kG(\sigma)$-modules). 

Some of the results in this section are developments of the unpublished note \cite{Ste13}, and we reproduce and improve the proofs for the benefit of the literature.  We underscore the fact that our results allow for computations with exotic Frobenius kernels, as opposed to classical ones, and improve upon the bounds on the size of the finite group given in \cite{Ste13}.

{\bf Acknowledgements}:
The author wishes to thank her PhD supervisor Dr David I. Stewart for the guidance throughout. 
\section{Preliminaries} \label{Sect2}
\subsection{Notation}
We assume the following notation. 

Let $G$ be a simple, simply connected algebraic group over an algebraically closed field $k$ of characteristic $p>0$. Let $\sigma : G \rightarrow G$ be a surjective endomorphism of $G$ such that the fixed point set $G(\sigma)$ is a finite group of Lie type. Then, let $G_{\sigma}$ denote the scheme-theoretic kernel of $\sigma.$ 

We denote by $\T$ a maximal split torus in $G$ and let $\Phi$ be the corresponding root system; let $\Pi = \{ \alpha_1,..., \alpha_n \}$ be the set of simple roots in the Bourbaki ordering \cite[Planches]{Bou82} and $\alpha_0$ the maximal short root. Let $\B$ denote a Borel subgroup containing $\T$, corresponding to the negative roots, and let $U$ denote its unipotent radical. For our choice of $\sigma$, all these subgroups can be chosen to be $\sigma$-invariant.

Let $\langle \cdot, \cdot \rangle$ be the standard inner product on the Euclidean space $\mathbb{E}:=\mathbb{Z}\Phi \otimes_{\mathbb{Z}} \mathbb{R}$. Then, let $\alpha^{\vee}=\frac{2 \alpha}{\langle \alpha, \alpha \rangle}$ be the coroot of $\alpha \in \Phi$ and let $h$ be the Coxeter number of the root system. 

We have the weight lattice  $X(T)=X= \bigoplus \mathbb{Z} \omega_i$, for $\omega_i$ the fundamental dominant weights satisfying $\langle \omega_i, \alpha_j^{\vee} \rangle = \delta_{ij}$, for $\alpha_j$ a simple root. Then $X_{+}$ is the cone of dominant weights. 

Let $W$ be the Weyl group of $\Phi$, generated by the set of simple reflections $\{s_{\beta} : \beta \in \Pi \}$. For $\alpha \in \Phi,$ $s_{\alpha}: \mathbb{E} \to \mathbb{E}$ is the orthogonal reflection in the hyperplane $H_{\alpha} \subset \mathbb{E}$ of vectors orthogonal to $\alpha$. Write $\ell:W \to \mathbb{N}$ for the standard length function on $W$: for $w \in W$, $\ell(w)$ is the minimum number of simple reflections required to write $w$ as a product of simple reflections. Moreover, note that $W$ acts naturally on $X(T)$ via the dot action. (cf. \cite[II.1.5]{Jan03}) 

Let $\lambda^{\ast}= - w_{0} \lambda$, with $w_{0}$ the longest word in the Weyl group $W$ and $\lambda \in X(T).$ For the remainder of this paper, we have 
$\lambda^{\ast}= \lambda$, as we only consider root systems of type $C_2$, $G_2$ or $F_4$, for which $w_0=-1$ (cf. \cite[Planches III, VIII, IX]{Bou82}). The irreducible $G$-modules are indexed by the dominant weights, so that $L(\lambda)$ is the finite-dimensional irreducible module of highest weight $\lambda \in X_{+}$; moreover, the irreducible modules are self-dual in this case. Consider $\lambda \in X_{+}$ as a one-dimensional $\B$-module and let $\opH^0(\lambda) = \Ind_{\B}^{G} \lambda$ be the induced module; note that since $G/\B$ is a projective variety, this module is finite-dimensional. We also have that the Weyl module $\mathrm{V}(\lambda) \cong \opH^0(\lambda^{\ast})^{\ast}.$

\subsection{The Suzuki and Ree groups} \label{prelsuzree}

Let $G$, $\sigma$ and $G(\sigma)$ be defined as above. Now and for the remainder of the paper, $G$ is of type $C_2$ ($p=2$), $G_2$ ($p=3$) or $F_4$ ($p=2$). Let $\tau : G \rightarrow G$ be the fixed purely inseparable isogeny defined such that $\tau^2 =F$, where $F$ denotes the Frobenius endomorphism of $G$. For a positive integer $s$, if we set $r=2s$, then $\sigma = F^s$ and $G(\sigma)$ is the split Chevalley group. If, however, we set $r=2s+1$, $\sigma = \tau \circ F^s = \tau^{r}$ and $G(\sigma)$ is one of $\prescript{2}{}C_2(2^{2s+1})$, $\prescript{2}{}G_2(3^{2s+1})$ or $\prescript{2}{}F_4(2^{2s+1})$. 
From this point onwards, we have $r=2s+1$ and we shall use $\sigma$ and $\tau^{r}$ interchangeably. 

Given $\sigma$ and a rational $G$-module $M$, let $M^{(\sigma)}$ denote the twist of the module, obtained by precomposing the action map with $\sigma$. We may also define the untwist, $M^{(-\sigma)}$, if $G_{\sigma}$ acts trivially on $M$ \cite[I.9.10]{Jan03}. 

Let $X_1 = \{ \lambda \in X_{+} :  \langle \lambda, \alpha_{i}^{\vee} \rangle < p, \alpha_i \in \Pi \}$ be the set of $p$-restricted weights and we may define the $\tau$-restricted ones, $X_{\tau} \subset X_1$, as the subset of ones orthogonal to the long simple roots. As a result, we have the following condition for $X_{r/2}$-restricted weights: if $r$ even, then $X_{\sigma}=X_s$; if $r$ is odd, then we require that $\langle \lambda, \alpha_{i}^{\vee} \rangle < p^{s+1}$, for $\alpha_i \in \Pi $ short, and $\langle \lambda, \alpha_{i}^{\vee} \rangle < p^{s}$, for $\alpha_i \in \Pi $ long. 

We then have that any dominant weight $\lambda$ may be uniquely expressed as $\lambda = \lambda_0 + \tau^r \lambda_1$, for $\lambda_0 \in X_{r/2}$ and $\lambda_1 \in X_{+}.$ In fact, we have an analogue of Steinberg's Tensor Product Theorem, with $\tau$ in place of $F$; namely, $L(\lambda) \cong L(\lambda_0) \otimes L(\lambda_1)^{(r/2)}$ \cite[Theorem 12.2]{Ste63}.

\subsubsection*{Frobenius kernels}

We turn our attention to the scheme-theoretic kernel $G_{\sigma}=G_{r/2}$ of $\sigma$; many of the results concerning classical Frobenius kernels also hold in this case, and we refer the interested reader to \cite[Remark 2.2.1]{BNP+12} for a more detailed discussion. For our purposes, we have that $G_{r/2}$ is a normal subgroup scheme and $G/G_{r/2} \cong G^{(r/2)}$. We also note that when $r$ is odd, $G_{r/2}/G_{1/2} \cong G_{\frac{r-1}{2}}$ is a classical Frobenius kernel. Then, observe that by \cite[2.4]{BNP04a}, the Steinberg module $\mathrm{St}_{r/2}$, of highest weight $(\tau^r-1)\rho$, is injective as a $G_{r/2}$-module. 

Moreover, since $B,T$ and $U$ are subgroups of $G$, by \cite[I.9.5]{Jan03}, it follows that $B_{r/2}= B \cap G_{r/2}$, $U_{r/2}= U \cap G_{r/2}$ and $T_{r/2}= T \cap G_{r/2}$. Thus, we may consider the Frobenius kernels $B_{r/2}$, $U_{r/2}$ and $T_{r/2}$, which are also normal subgroup schemes of the groups $B, U, T$ respectively. 

These various normal subgroups give rise to Lyndon--Hochschild--Serre (LHS) spectral sequences of which we make significant use throughout the paper.

\subsubsection*{Spectral sequences} 

We recall some of the key facts about spectral sequences, for the unfamiliar reader. (See \cite{McC} or \cite{Jan03} for an exhaustive treatment.) Let $\mathcal{C}$ be an abelian category; then, a spectral sequence (of cohomological type) consists of a family of bigraded objects $E_{n} = \bigoplus_{i,j \in \mathbb{Z}} E_n^{i,j}$ of $\mathcal{C}$ and differentials of bidegree $(n,-n+1)$, $d_n:E_n^{i,j} \rightarrow E_n^{i+n,j-n+1}$ and $d_{n}:E_n^{i-n,j+n-1} \rightarrow E_n^{i,j} $, which satisfy $d_n \circ d_n =0.$ We require
\begin{displaymath}
E_{n+1}^{i,j} \cong \opH(E_n^{i,j}) \cong \frac{\mathrm{ker}(d_n:E_n^{i,j} \rightarrow E_n^{i+n,j-n+1})}{\mathrm{im}(d_{n}:E_n^{i-n,j+n-1} \rightarrow E_n^{i,j})}.
\end{displaymath}
The collections $(E_n^{i,j})_{i,j}$ for fixed $n$ are known as the sheets of the spectral sequence, and we move to the next one by taking cohomology, using the isomorphism above. 

We say that the spectral sequence converges if, for every pair $(i,j)$, $E_n^{i,j}$ eventually stabilises as $n \rightarrow \infty$, and we denote the stable value by $E_{\infty}^{i,j}$. Furthermore, $\lbrace E_n, d_n \rbrace$ is a first quadrant spectral sequence if $E_n^{i,j}=0$ if $i<0$ or $j<0$ and we know that such sequences converge. 

By  \cite[6.5 and 6.6.(3)]{Jan03},  for $H_1, H_2$ algebraic $k$-groups, such that $H_2 \lhd H_1$, we have a first quadrant Lyndon-Hochschild-Serre spectral sequence
for each $H_1$-module $M$
\begin{displaymath}
	E_{2}^{i,j} = \opH^i(H_1/H_2, \opH^j(H_2, M)) \Rightarrow \opH^{i+j}(H_1, M).
	\end{displaymath}
Note that since $H_2 \lhd H_1$, $H_2$ is exact in $H_1$ and the category of $H_1/H_2$-modules is abelian. 

In particular, in this paper, we use the LHS spectral sequences corresponding to $B_{1/2} \lhd B_{r/2}$ and $G_{r/2} \lhd G$.


\section{Cohomology for the Frobenius Kernels} \label{Sect3}

In this section we compute the 1-cohomology for the Frobenius kernels of the induced modules for the Suzuki and Ree groups. Thus in this section $G$ is a simply-connected algebraic group of type $C_2$ (\ref{subsecC2}), $G_2$ (\ref{subsecG2}) and $F_4$ (\ref{subsecF4}).

We fix now and for the remainder of the paper a positive odd integer $r=2s+1$, with a view to calculating invariants of $G_{(r/2)}=\ker:\tau^r=\sigma:G\to G$.

\subsection{Preliminaries} \label{ssecFK}

We adapt methods of Jantzen \cite{Jan91} in order to compute the $B_{\tau}$-cohomology. Then, based on an argument in \cite{BNP04b}, we extend the results from $B_{\tau}$ to $B_{r/2}$; using an analogue of \cite[II.12.2(2)]{Jan03} for exotic Frobenius kernels, we obtain $\opH^1(G_{r/2}, \opH^0(\lambda))$. Moreover, we extend the $G_1$-cohomology results computed in \cite{Sin94b} to calculate $\opH^1(G_{s}, L(\lambda))$, for a positive integer $s$ and $\lambda \in X_{s}(T)$ and $\opH^1(G_{r/2}, L(\lambda))$, for $\lambda \in X_{r/2}(T)$. 

Since $G$ is simply connected, there exists a Chevalley basis for the Lie algebra $\gl_{\mathbb{Z}}$, which may be reduced modulo $p$ to obtain the restricted Lie algebra $\gl = \Lie(G).$ We write $\gl = \gl_{\mathbb{Z}} \otimes_{\mathbb{Z}} k$, where $\gl_{\mathbb{Z}} = \{ X_{\alpha}, \alpha \in \Phi, H_{\alpha}= [X_{\alpha}, X_{-\alpha}], \alpha \in \Pi \}$. Hence, suppose that $\alpha, \beta$ are roots with $\alpha+\beta$ also a root, with the associated root vectors $X_{\alpha}$, $X_{\beta}$ and $X_{\alpha + \beta}$, respectively, in $\gl_{\mathbb{Z}}$. It follows that the commutator $[X_{\alpha}, X_{\beta}] = N_{\alpha\beta} X_{\alpha + \beta}$, for some integer $N_{\alpha\beta}$ (with possible values $0, \pm 1, \pm 2, \pm 3$). 

Abusing notation, we shall also denote the element $X_{\alpha} \otimes 1$ of $\gl$ by $X_{\alpha}$. Moreover, upon reduction modulo $p$, whenever we have $\alpha, \beta$ two short roots whose sum is a long root, the structure constant $N_{\alpha,\beta}$ will vanish. 

Recall $(\Phi,p)$ is special, and therefore there exists a special isogeny $\tau$, satisfying $\tau^2=F$, the Frobenius map. This interacts with the root system in the following way. There is a subsystem of short roots denoted $\Phi_{s}$. In case $(G,p)=(C_2,2)$, $(G_2,3)$ and $(F_4,2)$ respectively, $\Phi_s$ is of type $A_1A_1$, $A_2$ and $D_4$, respectively. Degeneracies in the commutator relations in our specific characteristics guarantee Lie subalgebras $\mathfrak{g}_s$ wih root system $\Phi_s$ which are generated by the root vectors corresponding to the elements of $\Phi_s$, and maximal rank subgroups of $G$ whose root system is $\Phi_s$. The kernel of $d\tau$ is $\mathfrak{g}_s$, hence we write $\mathfrak{g}_\tau$ for this ideal. The kernel $G_\tau$ of $\tau$ is an infinitesimal group scheme of height one, whose representation theory is equivalent to the one of the restricted Lie algebra $\mathfrak{g}_{\tau}$. Since $U$ is $\tau$-stable, we get also a kernel $U_{\tau}$, whose Lie algebra $\mathfrak{u}_{\tau}$ is the ideal in $\mathfrak{u}$ generated by negative short roots. We obtain an analogue of \cite[Lemma 2.1]{Jan91}, whose proof is identical:
\begin{lemma}
	We have an isomorphism of $B$-modules
	\begin{displaymath}
	\opH^1(U_{\tau}, k ) \cong \opH^1(\mathfrak{u}_{\tau}, k) \cong (\mathfrak{u}_{\tau}/\left[ \mathfrak{u}_{\tau}, \mathfrak{u}_{\tau} \right])^{*}, 
	\end{displaymath}
	where $\mathfrak{u}_{\tau} = \Lie(U_{\tau}) = \left\langle X_{\beta} :  \beta \in \Phi_{s}^{-}\right\rangle $.
\end{lemma}
Here, $\Phi_{s}^{-}$ denotes the set of the negative roots of $\Phi_s$, the subsystem generated by the short roots.  

Analogously to \cite[Prop 2.2]{Jan91} we have:
\begin{lemma} \label{directsum}
Let $\beta_i$ be a set of simple roots of $\Phi_s$. Then,
\begin{displaymath}
\opH^1(U_{\tau}, k ) \cong \bigoplus_{i} k_{\beta_i}.
\end{displaymath}
\end{lemma}
\begin{proof}
The subalgebra
	$\left[ \mathfrak{u}_{\tau}, \mathfrak{u}_{\tau} \right]$ is spanned by all commutators $\left[X_{\alpha}, X_{\beta} \right] =N_{\alpha,\beta} X_{\alpha,\beta}$, for $\alpha,\beta$ negative short roots. Moreover, $N_{\alpha,\beta}\neq 0$ if and only if $\alpha+\beta$ is a short root. Using this, one checks $\left[ \mathfrak{u}_{\tau}, \mathfrak{u}_{\tau} \right]$ is spanned by root vectors corresponding to non-simple roots. Thus $\mathfrak{u}_\tau/\left[ \mathfrak{u}_{\tau}, \mathfrak{u}_{\tau} \right]$ has a basis with elements the classes of $X_{-\beta_i}$, being the weight vectors for $T_{\tau}$ for weights $-\beta_i$. The result follows by dualising.
\end{proof}	
As discussed in Subsection \ref{prelsuzree}, $B_{\tau}$ acts trivially on the weight module $k_{\tau(\lambda)} \cong k_\lambda ^{\tau}$.
Then, we obtain:
\begin{lemma} \label{Ttau}
For $\lambda \in X_{r/2}$ and $\beta_i$ simple roots of $\Phi_s$, there exist the following isomorphisms
\begin{equation*}
\begin{aligned}
\opH^1(B_{\tau}, \lambda) & \cong \left[ \opH^1(U_{\tau}, k ) \otimes k_{\lambda} \right]^{T_{\tau}} \\
& \cong \left[ \bigoplus_{i} k_{\beta_i + \lambda } \right]^{T_{\tau}}.
\end{aligned}
\end{equation*}
\end{lemma}

Since any weight $\lambda$ can be uniquely written as $\lambda = \lambda_{0} + \tau(\lambda_{1})$, for $\lambda_0 \in X_{\tau}(T)$ and $\lambda_1 \in X(T)$, we have $\opH^1(B_{\tau}, \lambda) \cong \opH^1(B_{\tau}, \lambda_{0}) \otimes  \tau(\lambda_{1})$. In particular, when $\lambda$ is $r/2$-restricted, we have $\lambda = \lambda_{0} + \tau(\lambda_{1})$, for $\lambda_0 \in X_{\tau}(T)$ and $\lambda_1 \in X_s(T)$.  Thus, it suffices to compute $\opH^1(B_{\tau}, \lambda_{0})$, for $\lambda_{0} \in X_{\tau}(T)$. 

Considered as a $T$-module, $\opH^1(U_{\tau}, k) \otimes \lambda_{0}$ is the direct sum of certain $k_{\beta_i + \lambda_{0}}$, for $\beta_i$, as previously defined. Such a summand yields a non-zero contribution to  $\opH^1(B_{\tau}, \lambda_{0})$ if and only if $\beta_i + \lambda_{0} \in \tau X(T)$. Hence, the problem boils down to checking which of these weights belong to $\tau X(T)$.

Once we have established the appropriate $B_{r/2}$-cohomology, the next result yields the $G_{r/2}$-cohomology with coefficients in the induced modules.
\begin{lemma}
Let $\lambda \in X(T)_{+}$. Then 
    \begin{equation} \label{indBG}
    \opH^{1}(G_{r/2}, \opH^{0}(\lambda))^{(-r/2)} \cong \Ind_{B}^{G}(\opH^{1}(B_{r/2}, \lambda)^{(-r/2)}).
\end{equation}
\end{lemma}

\begin{proof}
By \cite[Remark 2.2.1, (2.2.3)]{BNP+12}, there exists a spectral sequence 
\begin{displaymath}
E^{i,j}_2= R^i \Ind^{G}_{B} \opH^j(B_{r/2}, \lambda)^{(-r/2)} \Rightarrow \opH^{i+j=n}(G_{r/2}, \opH^0(\lambda))^{(-r/2)},
\end{displaymath}
for $\lambda \in X_{+}$ viewed as a one-dimensional $B$-module, giving rise to the corresponding five-term exact sequence
\begin{displaymath}
0 \to E^{1,0}_2 \to E^1_2 \to E^{0,1}_2 \to E^{2,0}_2 \to E^2_2.
\end{displaymath}
Following the programme in \cite[II.12.2]{Jan03}, suppose first $\lambda \notin \tau^{r}X(T)$. Then $\opH^0(B_{r/2},\lambda)=0$, forcing $E^{n,0}_2=0$. Otherwise, $\lambda \in \tau^{r}X(T)$ and so we may write $\lambda=\tau^r\lambda'$ for some $\lambda' \in X(T)_{+}$. Then
\begin{displaymath}
E^{n,0}_2=R^n \Ind_{B}^G \opH^0(B_{r/2},\lambda)^{(-r/2)}\cong R^n \Ind_{B}^G \lambda'=0,
\end{displaymath}
for $n>0$, by Kempf's vanishing theorem (cf. \cite[II.4.5]{Jan03}). Therefore, $E^1_2 \cong E_2^{0,1}$, as required.
\end{proof}
By Kempf's vanishing theorem, $\opH^{0}(\lambda)=\Ind_{B}^{G} \lambda$ is zero unless $\lambda$ is dominant. For $\lambda \in X(T)_{+}$, one may use the preceding computations of $B_{r/2}$-cohomology to compute \linebreak $\opH^{1}\left(G_{r/2}, \opH^{0}(\lambda)\right)$ thanks to the isomorphism in (\ref{indBG}).

Moreover, one can make use of the  $G_1$-cohomology with coefficients in simple modules, computed in \cite[Proposition 2.3, 3.5, 4.11]{Sin94b}, to calculate $\opH^1(G_{s}, L(\lambda))$, for a positive integer $s$ and $\lambda \in X_{s}(T)$. Having established the $G_s$-cohomology, applying the LHS spectral sequence corresponding to $G_s \lhd G_{r/2}$ to compute $\opH^1(G_{r/2}, L(\lambda))$, for $\lambda \in X_{r/2}(T)$, completes the objectives set out for this section.

The remaining sections consider each case of $(\Phi,p)$ separately, computing the \linebreak $B_{r/2}$-cohomology and $G_{r/2}$-cohomology explicitly.


\subsection{$C_2$ in characteristic 2} \label{subsecC2}

Let $G$ be simply connected of type $C_2$ over $k$ of characteristic 2.
Following \cite[Planche III]{Bou82}, let $\Phi=\{ \pm 2 \epsilon_1, \pm 2\epsilon_2, \pm \epsilon_{1} \pm \epsilon_{2} \}$ be the roots of a system of type $C_2$. Writing $\epsilon_1=(1,0)$ and $\epsilon_2=(0,1)$, a base of simple roots is $\Pi:=\left\lbrace \alpha_{1}, \alpha_2 \right\rbrace$, with $\alpha_1=(1,-1)$ short, and  $\alpha_2=(0,2)$  long; furthermore, the corresponding fundamental dominant weights are $\omega_1=(1,0)$, $\omega_2=(1,1)$. One checks that a set of simple roots of $\Phi_s$ is $\Pi_s:=\left\lbrace \alpha_{1}, \alpha_1+\alpha_2 \right\rbrace $. We shall denote these simple roots by $\beta_1=\alpha_1=(1,-1)$, $\beta_2=\alpha_1+\alpha_2=(1,1)$. The special isogeny induces a $\mathbb{Z}$-linear map $\tau^{\ast} : X(T) \rightarrow X(T)$, under which $\omega_1 \mapsto \omega_2 \mapsto 2\omega_1.$ From now on, we abuse notation, writing $\tau$ instead of $\tau^{\ast}$. Thus, the $\tau$-restricted weights are $0$ and $\omega_1$.

\subsubsection*{$B_{\tau}$-cohomology}

Let $\lambda \in X_{r/2}$ be written as $\lambda = \lambda_{0} + \tau(\lambda_{1})$, for some $\lambda_1 \in X_s(T)$, such that $\opH^1(B_{\tau}, \lambda) \cong \opH^1(B_{\tau}, \lambda_{0}) \otimes  \tau(\lambda_{1})$. Thus, it suffices to compute $\opH^1(B_{\tau}, \lambda_{0})$, for $\lambda_{0} \in X_{\tau}(T)$. 

\begin{theorem} \label{btau-c2}
Let $\lambda_0 \in X_{\tau}(T)$. Then
\begin{equation*}\opH^{1}\left(B_{\tau}, \lambda_{0} \right) \cong\left\{\begin{array}{ll}
k_{\omega_2-\omega_1}^{(\tau)} \oplus k_{\omega_1}^{(\tau)} & \text { if } \lambda_{0}= k \\
0 & \text { else. }
\end{array}\right.\end{equation*}
\end{theorem}

\begin{proof}
By Lemma \ref{directsum}, considered as a $T$-module, $\opH^1(U_{\tau}, k) \otimes \lambda_{0}$ is the direct sum of certain $k_{\beta_i + \lambda_{0}}$, for $\beta_i \in \Pi_s$. By Lemma \ref{Ttau}, such a summand yields a non-zero contribution to  $\opH^1(B_{\tau}, \lambda_{0})$ if and only if $\beta_i + \lambda_{0} \in \tau X(T)$. We now directly verify which of these weights belong to $\tau X(T)$.

First, suppose $\lambda_{0}=0$. Then, we have
	\begin{displaymath}
	\beta_1+0=\alpha_1=2\omega_1 -\omega_2=\tau(\omega_2-\omega_1).
	\end{displaymath}
	\begin{displaymath}
	\beta_2+0=\alpha_1+\alpha_2=\omega_2 =\tau(\omega_1).
	\end{displaymath}
	 Hence, $\opH^1(B_{\tau}, k) \cong \left[ \bigoplus_{i} k_{\beta_i + 0} \right]^{T_{\tau}} \cong   \left[ k_{\tau(\omega_2-\omega_1)} \oplus k_{\tau\omega_1} \right]^{T_{\tau}} \cong k_{\omega_2-\omega_1}^{(\tau)} \oplus k_{\omega_1}^{(\tau)}$. 
	
Now, suppose $\lambda_{0}=\omega_1$ and we obtain
	\begin{displaymath}
	\beta_1+\omega_1=3\omega_1-\omega_2 \notin \tau X(T).
	\end{displaymath}
	\begin{displaymath}
	\beta_2+\omega_1=\omega_2+\omega_1 \notin \tau X(T).
	\end{displaymath}
	 Then, $\opH^1(B_{\tau}, \omega_1) \cong \left[ \bigoplus_{i} k_{\beta_i + \omega_1} \right]^{T_{\tau}} = 0$, as neither belongs to $\tau X(T).$ 
\end{proof}
\subsubsection*{$B_{r/2}$-cohomology} In this subsection, we extend the calculations of the previous section to compute $\opH^1(B_{r/2}, \lambda)$, for $\lambda \in X_{r/2}(T)$.

First, we note that in this case, the calculation of $\opH^1(B_{r/2}, \lambda)$ requires, among other things, knowledge of the second $B_s$-cohomology with coefficients in a $p^s$-restricted weight; this was computed in \cite[Theorem Appendix C.2.6]{W}. For the reader's convenience, we list these cohomology groups here, with data extracted specifically for the underlying root system of $G$ of type $C_2$.
\begin{lemma} \label{h2bs}
    Assume the underlying root system of $G$ is of type $C_2$.
    Let $s$ be a positive integer, $p=2$ with ${\lambda'} \in X_s(T)$ and $w \in W$.  Then 
    \begin{equation*}\opH^{2}(B_{s}, {\lambda'})  \cong\left\{\begin{array}{ll}
	 \opH^2(B_1,w \cdot0 +2\nu)^{(s)}  & \text { if } {\lambda'}=2^{s-1}(w \cdot0 +2\nu), \ell(w)=0,2 \\
	 \nu^{(s)}  &\text { if }  {\lambda'}=2^s\nu+2^l w \cdot 0, \ell(w)=0,2   \text { and } 0 \leq l <s-1 \\
	 \nu^{(s)}  & \text { if } {\lambda'}=2^{s}\nu-2^l \alpha, \alpha \in \Pi, 0 \leq l \leq s-1; \\
	 & \text{ and } l \neq s-1 \text { if } \alpha=\alpha_2 \\
	 \nu^{(s)}  & \text { if } {\lambda'}=2^{s}\nu-2^t \beta-2^l \alpha, \alpha, \beta \in \Pi, 0 \leq l <t<s \\
	 \nu^{(s)}  & \text { if } {\lambda'}=2^{s}\nu-2^l( \alpha_1+\alpha_2), 0 \leq l <s-1 \\
	 M^{(s)} \otimes \nu^{(s)}  & \text { if } {\lambda'}=2^{s}\nu-2^{s-1} \alpha_2-2^l \alpha, \alpha \in \Pi, 0 \leq l <s-1 \\
	 M^{(s)} \otimes \nu^{(s)}  & \text { if } {\lambda'}=2^{s}\nu-2^{s-1}\alpha, \alpha \in \Pi \\
	 \opH^1(B_{s-1},M^{(-1)} \otimes \lambda_1)  & \text { if } {\lambda'}=2\lambda_1, \text { for some } \lambda_1 \in X(T), s>1 \\
	 \oplus \opH^2(B_{s-1}, \lambda_1) & \\
	 0 & \text { else. }
	 \end{array}\right.\end{equation*}
	Here $M$ denotes an indecomposable $B$-module with head $k_{\alpha_1}$ and socle $k$ (cf. \cite[Theorem Appendix C.2.5]{W}). Note that it is implicit in the statement of the lemma that $s \geq 1$ or $s \geq 2$, depending on the case. 
\end{lemma}
If $r=1$, we refer the reader to Theorem \ref{btau-c2}.
\begin{theorem} \label{BrC2} Suppose $r=2s+1>1$ and let $\lambda \in X_{r/2}(T)$. Then, for $0 \leq i \leq s-2$, we have
	\begin{equation*}\opH^{1}\left(B_{r/2}, \lambda \right) \cong\left\{\begin{array}{ll}
	k_{\omega_{1}}^{(r/2)}  & \text { if } \lambda=(2^{s}-1) \omega_2=\tau^r \omega_1-\beta_2 \\
	k_{\omega_{2}}^{(r/2)}  & \text { if } \lambda=(2^{s+1}-2) \omega_1 + \omega_2=\tau^r \omega_2-\beta_1 \\
	k_{\omega_{1}}^{(r/2)}  & \text { if } \lambda=2^{s} \omega_1= \tau^r \omega_1-\tau^{2s-1} \alpha_1\\
	M^{(r/2)}_{C_2} & \text { if }  \lambda=0=\tau^r(\omega_2-\omega_1)-\tau^{2s-1}\alpha_2\\
	k_{\omega_{1}}^{(r/2)}  & \text { if } \lambda=(2^{s}-2^{i+1})\omega_2 +2^{i+1}\omega_1=\tau^r \omega_1-\tau^{2i+1} \alpha_1 \\
	k_{\omega_2}^{(r/2)} & \text { if } \lambda=2^{i+1}\omega_2 +(2^{s+1}-2^{i+2})\omega_1=\tau^r \omega_2-\tau^{2i+1} \alpha_2  \\
	0 & \text { else. }
	\end{array}\right.\end{equation*}
	Here $M_{C_2}$ denotes the 2-dimensional indecomposable $B$-module with head $k_{\omega_{1}}$ and socle $ k_{\omega_{2}-\omega_{1}}$ (cf. \cite[2.2]{BNP04b}).
\end{theorem} 

We underline that the last two non-zero instances only occur when $s \geq 2$ (or $r \geq 5$).

\begin{proof}
The second equality in each case identifying two forms of $\lambda$ is straightforward, recalling $\tau(\omega_1)=\omega_2$. Hence we now prove that $\lambda$ must be equal to one of the weights given by the first equality in each case. We consider the LHS spectral sequence
	\begin{displaymath}
	E_{2}^{i,j} = \opH^i(B_{r/2}/B_{\tau}, \opH^j(B_{\tau}, \lambda)) \Rightarrow \opH^{i+j}(B_{r/2}, \lambda)
	\end{displaymath}
	and the corresponding five-term exact sequence
	\begin{displaymath}
	0 \to E^{1,0} \to E^1 \to E^{0,1} \to E^{2,0} \to E^2.
	\end{displaymath}
	We will identify $E^1$ with either $E^{0,1}$ and $E^{1,0}$ and compute all of the non-zero cases in this way. 
	First, we fix some notation.  
	Since $\lambda \in X_{r/2}(T)$, it may be uniquely expressed as $\lambda=  \sum_{i=0}^{r-1} \tau^{i} \lambda_{i}$, where $\lambda_i$ are $\tau$-restricted weights. Then, we write $\lambda=\lambda_{0}+ \tau(\lambda')$, for $\lambda'= \sum_{j=1}^{r-1} \tau^{j-1} \lambda_{j}.$ Suppose $E^{0,1}\neq 0$ and consider the $E^{0,1}$-term.\\
	We have
	\begin{equation*}
	\begin{aligned}
	E^{0,1} & = \Hom_{B_{r/2}/B_{\tau}}(k, \opH^1(B_{\tau}, \lambda) )  \\
	        & \cong \Hom_{B_{r/2}/B_{\tau}}(k, \opH^1(B_{\tau}, \lambda_{0}) \otimes \tau(\lambda')).
	\end{aligned}
	\end{equation*}
There is only one $\tau$-restricted weight for which $\opH^1(B_{\tau}, \lambda_0) \neq 0$, namely $\lambda_0=0$. In this case
\begin{displaymath}
\opH^{1}(B_{\tau}, k) \cong k_{\omega_2-\omega_1}^{(\tau)} \oplus k_{\omega_1}^{(\tau)}.
\end{displaymath}
Hence
\begin{equation*}
	\begin{aligned}
	E^{0,1} & = \Hom_{B_{r/2}/B_{\tau}}( k, (k_{\omega_1}^{(\tau)} \oplus k_{\omega_2-\omega_1}^{(\tau)}) \otimes \tau(\lambda')) \\
	& \cong \Hom_{B_{(r-1)/2}}( k, (k_{\omega_1}^{(\tau)} \oplus k_{\omega_2-\omega_1}^{(\tau)}) \otimes k_{\lambda'}^{(\tau)})  \\
	& \cong \Hom_{B_{(r-1)/2}}( k, k_{\omega_1+\lambda'}^{(\tau)} \oplus k_{\omega_2-\omega_1+\lambda'}^{(\tau)}).
	\end{aligned}
	\end{equation*}
	
	Now $\Hom_{B_{(r-1)/2}}( k, k_{\omega_1+\lambda'}^{(\tau)} \oplus k_{\omega_2-\omega_1+\lambda'}^{(\tau)})$ is non-zero if at least one of $\omega_1+\lambda'$ or $\omega_2-\omega_1+\lambda' \in \tau^{r-1}X(T)$. In fact, $\Hom_{B_{(r-1)/2}}( k, k_{\omega_1+\lambda'}^{(\tau)} \oplus k_{\omega_2-\omega_1+\lambda'}^{(\tau)}) $ is at most one-dimensional: since $\omega_2-2\omega_1\not\in\tau^{r-1}X(T)$, at most one of $\omega_1+\lambda'$ and $\omega_2-\omega_1+\lambda'$ is in $\tau^{r-1}X(T)$. We take each case in turn.
	
	First, suppose $\omega_1+\lambda'\in \tau^{r-1} X(T)$. As $p=2$, we have $\lambda' = (a2^s-1)\omega_1+b2^s\omega_2\in X_s$. It immediately follows $b=0$ and $a=1$, giving $\lambda=\lambda_0+\tau(\lambda')=(2^s-1) \omega_2$ and 
	\[E^{0,1}=\Hom_{B_{r/2}/B_{\tau}}(k, k_{\tau(\omega_1+\lambda')}\oplus k_{\tau(\omega_2-\omega_1+\lambda')}).\]
	The first term in the target of the $\Hom$ is $k_{\omega_2+(2^s-1) \omega_2}=k_{2^s\omega_2}$. Thus $E^{0,1}\cong k_{2^s\omega_2}=(k_{\omega_1})^{(r/2)}$.
	
	In the case $\omega_2-\omega_{1}+\lambda' \in \tau^{r-1} X(T)$, a similar argument leads us to conclude that $E^{0,1}=k_{\omega_2}^{(r/2)}$ and $\lambda=\omega_2 + (2^{s+1}-2) \omega_1$. 
To conclude, for $\lambda \in X_{r/2}(T)$,
\begin{equation*}E^{0,1} \cong \left\{\begin{array}{ll}
	 k_{\omega_{1}}^{(r/2)}  & \text { if } \lambda =(2^{s}-1) \omega_2  \\
	 k_{\omega_{2}}^{(r/2)}  & \text { if } \lambda =(2^{s+1}-2)\omega_{1} + \omega_2\\
	 0 & \text { else. }
	 \end{array}\right.\end{equation*}
	
	 Now suppose $E^{1,0}\neq 0$. We have 
	 \begin{align*}E^{1,0}&=\opH^{1}(B_{r/2} / B_{\tau}, \Hom_{B_{\tau}}(k, \lambda)), \\
	 &=\opH^{1}(B_{r/2} / B_{\tau}, \Hom_{B_{\tau}}(k, \lambda_0)\otimes \tau(\lambda'))
	 \end{align*}
	 so $\lambda_0=0$ and $\lambda=\tau(\lambda')$. Thus $E^{1,0}\cong \opH^{1}( B_{s},  {\lambda'} ^{(\tau)}) \cong  \opH^{1}( B_{s},  {\lambda'})^{(\tau)}$
	 for $\lambda=\tau {\lambda'}$. Notice that since $r-1=2s>0$, $B_{(r-1)/2}=B_{s}$ is a classical Frobenius kernel and $\opH^{1}( B_{s},  {\lambda'})$ is the $B_s$-cohomology for ${\lambda'} \in X_{s}(T)$ computed in \cite[Theorem 2.7]{BNP04b}. We have
	 \begin{equation*}\opH^{1}(B_{s}, {\lambda'} ) \cong\left\{\begin{array}{ll}
	 k_{\omega_{1}}^{(s)}  & \text { if } {\lambda'}=2^{s} \omega_1 - 2^{s-1}\alpha_1 \\
	 M^{(s)}_{C_2} & \text { if }  {\lambda'}=0=2^s(\omega_2-\omega_1)-2^{s-1}\alpha_2 \\
	 k_{\omega_{j}}^{(s)}  & \text { if } {\lambda'}=2^{s}\omega_{\alpha} -2^{i}\alpha, \alpha \in \Pi, 0 \leq i \leq s-2\\
	 0 & \text { else. }
	 \end{array}\right.\end{equation*}
	 with $M_{C_2}$ having the structure as claimed in the statement of the theorem. We note the implicit constraints on $s$ in the different cases. Hence,
	 \begin{equation*}E^{1,0} \cong \opH^{1}\left(B_{s}, {\lambda'} \right)^{(\tau)} \cong \left\{\begin{array}{ll}
	 k_{\omega_{1}}^{(r/2)}  & \text { if } {\lambda'}=2^{s} \omega_1 - 2^{s-1}\alpha_1 \\
	 M^{(r/2)}_{C_2} & \text { if }  {\lambda'}= 0=2^s(\omega_2-\omega_1)-2^{s-1}\alpha_2 \\
	 k_{\omega_{j}}^{(r/2)}  & \text { if } {\lambda'}=2^{s}\omega_{\alpha} -2^{i}\alpha, \alpha \in \Pi, 0 \leq i \leq s-2\\
	 0 & \text { else. }
	 \end{array}\right.\end{equation*}
	 One can recover $\lambda$ from ${\lambda'}$, recalling $\alpha_1=2\omega_{1}-\omega_{2}$ and $\alpha_2=2\omega_{2}-2\omega_{1}$.	 For example if ${\lambda'} =2^{s} \omega_1 - 2^{s-1}\alpha_1=2^{s-1} \omega_2$, then $\lambda=\tau {\lambda'}=2^s \omega_1.$ The other cases are similar.
	 	  
In light of the above, observe that there is no choice of $\lambda$ for which $E^{1,0}$ and $E^{0,1}$ are both non-zero. Hence if $E^{1,0} \neq 0$, then $E^{0,1}=0$, implying that $E^{1} \cong E^{1,0}$. Alternatively, suppose that $E^{0,1} \neq 0$, so $E^{1,0}=0$. It remains to check whether the differential $d_2: E^{0,1} \to E^{2,0}$ is the zero map. Assume $E^{2,0} \neq 0$ and we have 
\begin{equation*}
	 \begin{aligned}
	 E^{2,0} &=\opH^{2}(B_{r/2} / B_{\tau}, \Hom_{B_{\tau}}(k, \lambda)) \\
	 &\cong \opH^{2}(B_{(r-1)/2}, \Hom_{B_{\tau}}(k, \lambda)^{(-\tau)})^{(\tau)} \\
	 & \cong \opH^{2}( B_{s},  {\lambda'}^{(\tau)}) \cong  \opH^{2}( B_{s},  {\lambda'})^{(\tau)}
	 \end{aligned}
	 \end{equation*}
	 for $\lambda=\tau {\lambda'}$. As before, $B_{(r-1)/2}=B_{s}$ is a classical Frobenius kernel and $\opH^{2}( B_{s}, {\lambda'})$ is the second $B_s$-cohomology for ${\lambda'} \in X_{s}(T)$ computed in Lemma \ref{h2bs}. 

Since $E^{0,1} \neq 0$, ${\lambda'}=2^s \omega_1-\omega_1$ or $2^s \omega_2+\omega_1-\omega_2$. In each case the coefficient of $\omega_1$ in ${\lambda'}$ is odd, so ${\lambda'}$ is not in the root lattice; however, since $E^{2,0} \neq 0$, we see from Lemma \ref{h2bs} that $\lambda'$ is in the root lattice---a contradiction. It follows that the differential $d_2: E^{0,1} \to E^{2,0}$ is the zero map. Therefore, if $E^{0,1} \neq 0$, we have $E^{1} \cong E^{0,1}$.
\end{proof}	

For a general $\lambda \in X(T)$, not necessarily lying in $X_{r/2}$, we proceed as in \cite[2.8]{BNP04b}.
\begin{cor} \label{C2weightform}
    Let $\lambda \in X(T)$ and $r=2s+1>1$. Then $\opH^1(B_{r/2}, \lambda) \neq 0$ if and only if $\lambda= \tau^r \omega - \tau^i \alpha$, for some weight $\omega \in X(T)$, and $\alpha \in \Pi$ with $0 \leq i \leq 2s-1$ or $\lambda= \tau^r \omega - \beta$, for some weight $\omega \in X(T)$, and $\beta \in \Pi_s$.
\end{cor}
\begin{proof}
Suppose $\opH^1(B_{r/2}, \lambda) \neq 0$. Then we may uniquely write $\lambda=\lambda_0+\tau^r \lambda_1$, for $\lambda_0 \in X_{r/2}(T)$ and $\lambda_1 \in X(T)$. It follows that $\opH^1(B_{r/2},\lambda) \cong \opH^1(B_{r/2},\lambda_0) \otimes \tau^r \lambda_1$. Thus, by Theorem \ref{BrC2}, $\opH^1(B_{r/2},\lambda_0) \neq 0$ if and only if $\lambda_0=\tau^r \omega'-\tau^i \alpha$ for $\alpha \in \Pi$ and $0 \leq i \leq 2s-1$, or $\lambda_0=\tau^r \omega'-\beta$ for $\beta \in \Pi_s$, with $\omega'$ the specific weight in the theorem. In the first case, we may then write $\lambda=\lambda_0+\tau^r \lambda_1=\tau^r \omega'-\tau^i \alpha + \tau^r \lambda_1=\tau^r(\omega'+\lambda_1)-\tau^i \alpha=\tau^r\omega-\tau^i \alpha$. Secondly, we have $\lambda=\lambda_0+\tau^r \lambda_1=\tau^r \omega'-\beta + \tau^r \lambda_1=\tau^r(\omega'+\lambda_1)-\beta=\tau^r\omega-\beta$. In both cases, we obtain the required form.

Conversely, suppose we are given any weight $\lambda=\tau^r\omega-\tau^i \alpha$, with $\alpha \in \Pi$ and $0 \leq i \leq 2s-1$ or $\lambda= \tau^r \omega - \beta$, for $\beta \in \Pi_s$. In either case, one can always express $\omega$ as $\omega=\omega'+\lambda_1$, for the required weight $\omega'$ in Theorem \ref{BrC2} and some weight $\lambda_1 \in X(T).$ Hence, $\opH^1(B_{r/2},\lambda) \neq 0$ for all such $\lambda$, as non-vanishing is independent of the choice of $\lambda_1.$ 
\end{proof}

Suppose $\opH^1(B_{r/2},\lambda)\neq 0$ and let $(\zeta, j)$ denote the appropriate pair, $(\alpha, i)$ or $(\beta, 1)$, as defined in the previous corollary. 
Now, given $\lambda=\tau^r \omega-\tau^j \zeta$, we may write $\lambda=\tau^r \omega' - \tau^j \zeta + \tau^r \lambda_1$, 
where $\omega'$ is chosen as per the list in Theorem \ref{BrC2}, and so it follows that $\lambda_1$ is $\omega-\omega'$. Hence
\begin{equation*}
\begin{aligned}
    \opH^1(B_{r/2}, \lambda) & \cong \opH^1(B_{r/2}, \lambda_0) \otimes k_{\lambda_1}^{(r/2)} \\
    & \cong \opH^1(B_{r/2}, \tau^r \omega' - \tau^j \zeta) \otimes k_{\omega-\omega'}^{(r/2)}.
\end{aligned}
\end{equation*}
Direct verification, substituting the answers from Theorem \ref{BrC2}, yields the following result
\begin{theorem} \label{Br2C2gen}
Let $\lambda \in X(T).$ Then, for $ 0 \leq i \leq s-2$, we have
	\begin{equation*}\opH^{1}\left(B_{r/2}, \lambda \right) \cong\left\{\begin{array}{ll}
	k_{\omega}^{(r/2)}  & \text { if } \lambda =\tau^r \omega - \beta_i, \omega \in X(T), \beta_i \in \Pi_s \\
	k_{\omega}^{(r/2)}  & \text { if } \lambda =\tau^r \omega - \tau^{2s-1}\alpha_1, \omega \in X(T) \\
	M^{(r/2)}_{C_2} \otimes k_{\omega+\omega_1-\omega_2}^{(r/2)} & \text { if }   \lambda =\tau^r \omega - \tau^{2s-1}\alpha_2, \omega \in X(T)\\
	k_{\omega}^{(r/2)}  & \text { if } \lambda =\tau^r \omega - \tau^{2i+1}\alpha_j, \omega \in X(T), \alpha_j \in \Pi \\
	0 & \text { else. }
	\end{array}\right.\end{equation*}
\end{theorem} 

\subsubsection*{$G_{r/2}$-cohomology of induced modules}

By Kempf's vanishing theorem, $\opH^{0}(\lambda)=\Ind_{B}^{G} \lambda$ is zero unless $\lambda$ is dominant. For $\lambda \in X(T)_{+}$, one may use Theorem \ref{btau-c2} and Theorem \ref{BrC2}, respectively, to compute $\opH^{1}(G_{r/2}, H^{0}(\lambda))$ with the aid of the isomorphism (\ref{indBG}).
Finally, we note that, by \cite[3.1. Theorem (C)]{BNP04b}, $\Ind_B^G(M_{C_{2}})= \opH^0(\omega_1)$.

In case $r=1$, we obtain:
\begin{theorem} \label{gtau-c2}
Let $\lambda \in X_{\tau}(T)$. Then
\begin{equation*}\opH^{1}\left(G_{\tau}, \opH^0(\lambda) \right)^{(-\tau)}  \cong\left\{\begin{array}{ll}
 \opH^0(\omega_1) \cong L(\omega_1) & \text { if } \lambda= 0 \\
0 & \text { else. }
\end{array}\right.\end{equation*}
\end{theorem}
Otherwise, we have:
\begin{theorem} \label{c2gr2}
	Let $r=2s+1>1$, $\lambda \in X_{r/2}(T)$ and $0 \leq i \leq s-2$. Then
	\begin{equation*}\opH^{1}(G_{r/2}, \opH^{0}(\lambda))^{(-r/2)} \cong\left\{\begin{array}{ll}
	\opH^{0}(\omega_1)  & \text { if } \lambda=(2^{s}-1) \omega_2=\tau^r \omega_1-\beta_2 \\
	\opH^{0}(\omega_2)  & \text { if } \lambda=(2^{s+1}-2) \omega_1 + \omega_2=\tau^r \omega_2-\beta_1 \\
    \opH^{0}(\omega_1)  & \text { if } \lambda=2^{s} \omega_1=\tau^r \omega_1-\tau^{2s-1}\alpha_1  \\
	\opH^{0}(\omega_1)  & \text { if } \lambda= 0=\tau^r( \omega_2-\omega_1)-\tau^{2s-1}\alpha_2 \\
	\opH^{0}(\omega_1)  & \text { if } \lambda=(2^{s}-2^{i+1}) \omega_2 +2^{i+1}\omega_1=\tau^r \omega_1-\tau^{2i+1}\alpha_1 \\
	\opH^{0}(\omega_2)  & \text { if }  \lambda=2^{i+1}\omega_2 +(2^{s+1}-2^{i+2})\omega_1=\tau^r \omega_2-\tau^{2i+1}\alpha_2  \\
	0 & \text { else. }
	\end{array}\right.\end{equation*} 
\end{theorem}

Next, one can make use of Theorem \ref{Br2C2gen} to compute $\opH^1(G_{r/2}, \opH^0(\lambda))$ in terms of induced modules for all dominant weights $\lambda$, by applying the induction functor $\Ind_B^G$. The only non-obvious case is dealt with in the remark below. 

\begin{remark}
    Let $\tau^r \omega - \tau^{2s-1} \alpha_2 \in X(T)_+$.
    Then $\langle \omega, \alpha_1^{\vee} \rangle \geq -1$ and
    $\langle \omega, \alpha_2^{\vee} \rangle \geq 1$. In this case, by \cite[Proposition 3.4 (B)(c)]{BNP04b},  $\Ind_{B}^G(M_{C_2} \otimes k_{\omega+\omega_1-\omega_2})$ has a filtration with factors satisfying the following short exact sequence
    \begin{displaymath}
    0 \to \opH^0(\omega) \to \Ind_{B}^G(M_{C_2} \otimes k_{\omega+\omega_1-\omega_2}) \to \opH^0(\omega+2\omega_1-\omega_2) \to 0.
    \end{displaymath}
    (Observe that $\opH^0(\omega+2\omega_1-\omega_2)$ is always present, but $\opH^0(\omega)$ appears as a factor if $\langle \omega, \alpha_1^{\vee} \rangle \geq 0$.)
    \end{remark}
\subsubsection*{$G_{r/2}$-cohomology with coefficients in simple modules}  In this subsection, we make use of the  $G_1$-cohomology with coefficients in simple modules, computed in \cite[Proposition 2.3]{Sin94b}, to calculate $\opH^1(G_{s}, L(\lambda))$, for a positive integer $s$ and $\lambda \in X_{s}(T)$.

First, we underline that in this case, the calculation of $\opH^1(G_{s}, L(\lambda))$ requires knowledge of the following cohomology group.
\begin{lemma} \label{h2g1}
    Let $G$ be of type $C_2$ and $p=2$. Then $\opH^2(G_1,L(\omega_1))=0$.
\end{lemma}
\begin{proof}
We run the LHS spectral sequence corresponding to $G_{\tau} \lhd G_1$. The $E_2$-page is given by
\begin{displaymath}
E^{i,j}_2=\opH^i(G_{\tau}, \opH^j(G_{\tau}, L(\omega_1))^{(-\tau)})^{(\tau)}.
\end{displaymath}
By Theorem \ref{gtau-c2}, $\opH^1(G_{\tau}, L(\omega_1))=0$, so $E^{i,1}=0.$ It follows that we obtain the following five-term exact sequence
\begin{displaymath}
0 \to E^{2,0} \to E^{2} \to E^{0,2} \to E^{3,0}.
\end{displaymath}
First, note that the $E^{2,0}$-term vanishes, since $\Hom_{G_{\tau}}(k, L(\omega_1))=0$. Moreover, $L(\omega_1)$ is an injective module for $G_{\tau}$, so $\opH^2(G_{\tau},L(\omega_1))=0$. Therefore, $E^{0,2}$ also vanishes, so we conclude that $E^2=\opH^2(G_1,L(\omega_1))=0$.
\end{proof}
\begin{theorem} \label{c2gs-simple}
	Let $s$ be a positive integer, $\lambda \in X_s(T)$ and $1 \leq i \leq s-1$. Then
	\begin{equation*}\opH^{1}(G_s, L(\lambda))^{(-s)} \cong\left\{\begin{array}{ll}
	L(\omega_1)  & \text { if } \lambda=0 \\
	k  & \text { if } \lambda=\omega_2 \\
    k  & \text { if } \lambda=2^{i} \omega, \omega \in \{\omega_1, \omega_2\}  \\
	0 & \text { else. }
	\end{array}\right.\end{equation*} 
	Note that it is implicit in the statement of the theorem that $s \geq 1$ or $s \geq 2$, depending on the case. 
\end{theorem}
\begin{proof}
We proceed inductively. When $s=1$, we refer the reader to \cite[Proposition 2.3]{Sin94b}.
We write $\lambda=\lambda_0+2^{s-1}\lambda_1$, for $\lambda_0 \in X_{s-1}$ and $\lambda_1 \in X_1.$ Suppose $s>1$ and consider the LHS spectral sequence corresponding to $G_{s-1} \lhd G_s$. The $E_2$-page is given by
\begin{displaymath}
E^{i,j}_2:=\opH^i(G_1, \opH^j(G_{s-1}, L(\lambda_0))^{(-s+1)} \otimes L(\lambda_1))^{(s-1)}.
\end{displaymath}
First, consider the $E^{1,0}$-term. We have
\begin{displaymath}
E^{1,0}_2=\opH^1(G_1, \Hom_{G_{s-1}}(k, L(\lambda_0))^{(-s+1)} \otimes L(\lambda_1))^{(s-1)}.
\end{displaymath}
Note that $E^{1,0} \neq 0$ if and only if $\lambda_0=0$, in which case we obtain 
\begin{equation*}E^{1,0}=\opH^{1}(G_1, L(\lambda_1))^{(s-1)} \cong\left\{\begin{array}{ll}
	L(\omega_1)^{(s)} & \text { if } \lambda_1=0 \\
	k  & \text { if } \lambda_1=\omega_2 \\
	0 & \text { else. }
	\end{array}\right.\end{equation*} 
(cf. \cite[Proposition 2.3]{Sin94b}). Therefore, recalling that $\lambda=\lambda_0+2^{s-1}\lambda_1$, we may conclude that for $\lambda \in X_s$,
\begin{equation*}E^{1,0}\cong\left\{\begin{array}{ll}
	L(\omega_1)^{(s)} & \text { if } \lambda=0 \\
	k  & \text { if } \lambda=2^{s-1}\omega_2 \\
	0 & \text { else. }
	\end{array}\right.\end{equation*}
Now consider the $E^{0,1}$-term. We have
\begin{displaymath}
E^{0,1}=\Hom_{G_1}( L(\lambda_1), \opH^1(G_{s-1}, L(\lambda_0))^{(-s+1)})^{(s-1)}.
\end{displaymath}
We take each non-zero instance of $ \opH^1(G_{s-1}, L(\lambda_0))$ in turn. By the induction hypothesis, if $\lambda_0=0$, then $E^{0,1}=\Hom_{G_1}( L(\lambda_1), L(\omega_1))^{(s-1)}$. Thus $E^{0,1} \neq 0$ if and only if $\lambda_1=\omega_1$; it follows that $E^{0,1} \cong k$ for $\lambda=2^{s-1}\omega_1$. The other cases follow similarly and we obtain, for $1 \leq i \leq s-2$
\begin{equation*}E^{0,1}\cong\left\{\begin{array}{ll}
	k & \text { if } \lambda=2^{s-1}\omega_1 \\
	k  & \text { if } \lambda=\omega_2 \\
	k  & \text { if } \lambda=2^{i}\omega, \omega \in \{\omega_1,\omega_2\} \\
	0 & \text { else. }
	\end{array}\right.\end{equation*}
Notice that there is no choice of $\lambda$ for which $E^{1,0}$ and $E^{0,1}$ are both non-zero. Hence if $E^{1,0} \neq 0$, then $E^{0,1}=0$, implying that $E^{1} \cong E^{1,0}$. Alternatively, suppose that $E^{0,1} \neq 0$, so $E^{1,0}=0$. It remains to check whether the differential $d_2: E^{0,1} \to E^{2,0}$ is the zero map. The $E^{2,0}$-term is
\begin{displaymath}
E^{2,0}=\opH^2(G_1, \Hom_{G_{s-1}}(k, L(\lambda_0))^{(-s+1)} \otimes L(\lambda_1))^{(s-1)}.
\end{displaymath}
We consider each choice of $\lambda$ for which $E^{0,1} \neq 0$. If $\lambda=\omega_2$ or $\lambda=2^{i}\omega$, we obtain $\Hom_{G_{s-1}}(k, L(\lambda_0))=0$, so $E^{2,0}=0$.

It remains to verify the case $\lambda=2^{s-1}\omega_1$. Then $E^{2,0}=\opH^2(G_1,L(\omega_1))^{(s-1)}$, which vanishes by Lemma \ref{h2g1}. It follows that $d_2: E^{0,1} \to E^{2,0}$ is the zero map and we reach our conclusion.
\end{proof}
Next, making use of the previous theorem concerning the cohomology for classical Frobenius kernels, we compute $\opH^{1}(G_{r/2}, L(\lambda))$ for $r$ an odd positive integer and $\lambda \in X_{r/2}$.

If $r=1$, we refer the reader to Theorem \ref{gtau-c2}. Otherwise we obtain
\begin{theorem} \label{c2gr2-simple}
	Suppose $r=2s+1>1$ and let $\lambda \in X_{r/2}(T)$ with $1 \leq i \leq s-1$. Then
    \begin{equation*}\opH^{1}(G_{r/2}, L(\lambda))^{(-r/2)} \cong\left\{\begin{array}{ll}
	L(\omega_1)  & \text { if } \lambda=0 \\
	k  & \text { if } \lambda=2^s \omega_1 \\
	k  & \text { if } \lambda=\omega_2 \\
    k  & \text { if } \lambda=2^{i} \omega, \omega \in \{\omega_1, \omega_2\}  \\
	0 & \text { else. }
	\end{array}\right.\end{equation*} 
\end{theorem}
\begin{proof}
For $\lambda \in X_{r/2}$, write $\lambda=\lambda_0+2^s\lambda_1$, for $\lambda_0 \in X_s$ and $\lambda_1 \in X_{\tau}$. Consider the LHS spectral sequence corresponding to $G_{s} \lhd G_{r/2}$. The $E_2$-page is given by
\begin{displaymath}
E^{i,j}_2:=\opH^i(G_{\tau}, \opH^j(G_{s}, L(\lambda_0))^{(-s)} \otimes L(\lambda_1))^{(s)}.
\end{displaymath}
First, consider the $E^{1,0}$-term. We have
\begin{displaymath}
E^{1,0}=\opH^1(G_{\tau}, \Hom_{G_{s}}(k, L(\lambda_0))^{(-s)} \otimes L(\lambda_1))^{(s)}.
\end{displaymath}
Note that $E^{1,0} \neq 0$ if and only if $\lambda_0=0$, in which case we obtain 
\begin{equation*}E^{1,0}=\opH^{1}(G_{\tau}, L(\lambda_1))^{(s)} \cong\left\{\begin{array}{ll}
	L(\omega_1)^{(r/2)} & \text { if } \lambda_1=0 \\
	0  & \text { if } \lambda_1=\omega_1.
	\end{array}\right.\end{equation*} 
(cf. Theorem \ref{gtau-c2} and \cite[Lemma 2.1]{Sin94b}). Next, consider the $E^{0,1}$-term:
\begin{displaymath}
E^{0,1}=\Hom_{G_{\tau}}(L(\lambda_1), \opH^1(G_{s}, L(\lambda_0))^{(-s)})^{(s)}.
\end{displaymath}
We take each non-zero instance of $\opH^1(G_{s}, L(\lambda_0))^{(-s)}$ from Theorem \ref{c2gs-simple} in turn. If $\lambda_0=\omega_2$, then 
$E^{0,1}=\Hom_{G_{\tau}}(L(\lambda_1), k)^{(s)}$. Thus $E^{0,1} \neq 0$ if and only if $\lambda_1=0$; we obtain $E^{0,1} \cong k$ for $\lambda=\omega_2$. The other cases are similar. Moreover, we can recover $\lambda$, recalling $\lambda=\lambda_0+2^s\lambda_1$. We get, for $1 \leq i \leq s-1$
\begin{equation*}E^{0,1}\cong\left\{\begin{array}{ll}
	k  & \text { if } \lambda=\omega_2 \\
	k  & \text { if } \lambda=2^{i}\omega, \omega \in \{\omega_1,\omega_2\} \\
	k & \text { if } \lambda=2^{s}\omega_1 \\
	0 & \text { else. }
	\end{array}\right.\end{equation*}
Observe that there is no choice of $\lambda$ for which $E^{1,0}$ and $E^{0,1}$ are both non-zero. Hence if $E^{1,0} \neq 0$, then $E^{0,1}=0$, implying that $E^{1} \cong E^{1,0}$. Alternatively, suppose that $E^{0,1} \neq 0$, so $E^{1,0}=0$. We must also investigate whether the differential $d_2: E^{0,1} \to E^{2,0}$ is the zero map. The $E^{2,0}$-term is
\begin{displaymath}
E^{2,0}=\opH^2(G_{\tau}, \Hom_{G_{s}}(k, L(\lambda_0))^{(-s)} \otimes L(\lambda_1))^{(s)}.
\end{displaymath}
We consider each choice of $\lambda$ for which $E^{0,1} \neq 0$ in turn. If $\lambda=\omega_2$ or $\lambda=2^{i}\omega$, for $\omega \in \{\omega_1,\omega_2\}$, it follows that $\Hom_{G_{s}}(k, L(\lambda_0))=0$, which forces $E^{2,0}=0$.

Lastly, suppose $\lambda=2^{s}\omega_1$. In this case, $E^{2,0}=\opH^2(G_{\tau},L(\omega_1))^{(s)}=0$, since $L(\omega_1)$ is injective for $G_{\tau}$. We conclude that $d_2$ is the zero map. Therefore, $E^{0,1} \neq 0$ implies $E^1 \cong E^{0,1}$.
\end{proof}
\subsection{$G_2$ in characteristic 3} \label{subsecG2}

Let $G$ be simply connected of type $G_2$ over $k$ of characteristic 3.
Following \cite[Planche IX]{Bou82}, let $\Phi=\{ \pm (\epsilon_{1}-\epsilon_{2}), \pm(\epsilon_{1} - \epsilon_{3}), \pm(\epsilon_{2} - \epsilon_{3}),  \pm(2\epsilon_{1} - \epsilon_{2} - \epsilon_{3}), \pm(2\epsilon_{2} - \epsilon_{1} - \epsilon_{3}), \pm(2\epsilon_{3} - \epsilon_{1} - \epsilon_{2})   \}$ be the roots of a system of type $G_2$. Writing  $\epsilon_{1}=(1,0,0)$, $\epsilon_{2}=(0,1,0)$ and $\epsilon_{3}=(0,0,1)$, we may take a base of simple roots to be $\Pi:=\{\alpha_1,\alpha_2\}$, with
$\alpha_1=(1,-1,0)$ short, and  $\alpha_2=(-2,1,1)$ long; moreover, the corresponding fundamental dominant weights are $\omega_1=(0,-1,1)$ and $\omega_2=(-1,-1,2)$. We may check that a set of simple roots of $\Phi_s$ is $\Pi_s:=\{\alpha_{1}, \alpha_1+\alpha_2 \}$. We shall denote these simple roots by $\beta_1=\alpha_1=(1,-1,0)$, $\beta_2=\alpha_1+\alpha_2=(-1,0,1)$. The special isogeny induces a $\mathbb{Z}$-linear map $\tau^{\ast}:X(T) \to X(T)$, under which $\omega_1 \mapsto \omega_2 \mapsto 3\omega_1.$ From this point onwards, we abuse notation, writing $\tau$ instead of $\tau^{\ast}$. Thus, the $\tau$-restricted weights are $0$, $\omega_1$ and $2 \omega_1$.

\subsubsection*{$B_{\tau}$-cohomology}

Let $\lambda \in X_{r/2}$ be expressed as $\lambda = \lambda_{0} + \tau(\lambda_{1})$, for $\lambda_0 \in X_{\tau}(T)$ and $\lambda_1 \in X_s(T)$, such that $\opH^1(B_{\tau}, \lambda) \cong \opH^1(B_{\tau}, \lambda_{0}) \otimes  \tau(\lambda_{1})$. Thus, it suffices to compute $\opH^1(B_{\tau}, \lambda_{0})$, for $\lambda_{0} \in X_{\tau}(T)$. 

\begin{theorem} \label{btau-g2}
Let $\lambda_0 \in X_{\tau}(T)$. Then
\begin{equation*}\opH^{1}\left(B_{\tau}, \lambda_{0} \right) \cong\left\{\begin{array}{ll}
k_{\omega_2-\omega_1}^{(\tau)} \oplus k_{\omega_1}^{(\tau)} & \text { if } \lambda_{0}= \omega_1 \\
0 & \text { else. }
\end{array}\right.\end{equation*}
\end{theorem}

\begin{proof}
Once again, Lemma \ref{directsum} tells us that, regarded as a $T$-module, $\opH^1(U_{\tau}, k) \otimes \lambda_{0}$ is the direct sum of certain $k_{\beta_i + \lambda_{0}}$, for $\beta_i \in \Pi_s$, as previously defined. Such a summand yields a non-zero contribution to  $\opH^1(B_{\tau}, \lambda_{0})$ if and only if $\beta_i + \lambda_{0} \in \tau X(T)$, by Lemma \ref{Ttau}. Hence, we need only check which of these weights belong to $\tau X(T)$. 

First, suppose $\lambda_{0}=0$. It is readily checked that we have no non-zero contribution. We conclude that  $\opH^1(B_{\tau}, 0)=0$.
	
Then, let $\lambda_{0}=\omega_1$ and we have
	\begin{displaymath}
	\beta_1+\omega_1=2\omega_1-\omega_2+\omega_1=3\omega_1 -\omega_2=\tau(\omega_2-\omega_1).
	\end{displaymath}
	\begin{displaymath}
	\beta_2+\omega_1=\omega_2-\omega_1 +\omega_1 = \omega_2 =\tau(\omega_1).
	\end{displaymath}
	 Then, $\opH^1(B_{\tau}, \omega_1) \cong \left[ \bigoplus_{i} k_{\beta_i + \omega_1} \right]^{T_{\tau}} \cong   \left[ k_{\tau(\omega_2-\omega_1)} \oplus k_{\tau\omega_1} \right]^{T_{\tau}} \cong k_{\omega_2-\omega_1}^{(\tau)} \oplus k_{\omega_1}^{(\tau)}$.  
	 
Lastly, suppose $\lambda_{0}=2\omega_1$. We obtain
	\begin{displaymath}
	\beta_1+2\omega_1=2\omega_1-\omega_2+2\omega_1=4\omega_1 -\omega_2 \notin \tau X(T).
	\end{displaymath}
	\begin{displaymath}
	\beta_2+2\omega_1=\omega_2-\omega_1 +2\omega_1 = \omega_1 + \omega_2 \notin \tau X(T).
	\end{displaymath}
	 Then, $\opH^1(B_{\tau}, 2\omega_1) \cong \left[ \bigoplus_{i} k_{\beta_i + \omega_1} \right]^{T_{\tau}} = 0$, since none of them lie in $\tau X(T).$ 
\end{proof}
\subsubsection*{$B_{r/2}$-cohomology} In this subsection, we extend the results of the previous section to calculate $\opH^1(B_{r/2}, \lambda)$, for $\lambda \in X_{r/2}(T)$.

First, when $r=1$, we direct the reader to Theorem \ref{btau-g2}. Otherwise, we obtain
\begin{theorem} \label{Br2G2}
	Suppose $r=2s+1>1$ and let $\lambda \in X_{r/2}(T)$. Then, for $0 \leq i \leq s-2$, we have
	\begin{equation*}\opH^{1}\left(B_{r/2}, \lambda \right) \cong\left\{\begin{array}{ll}
	k_{\omega_{1}}^{(r/2)}  & \text { if } \lambda=\omega_1 + (3^{s}-1) \omega_2=\tau^r \omega_1-\beta_2  \\
	 k_{\omega_{2}}^{(r/2)}  & \text { if } \lambda=(3^{s+1}-2)\omega_1+\omega_2=\tau^r \omega_2-\beta_1\\
	k_{\omega_{1}}^{(r/2)}  & \text { if } \lambda=3^{s} \omega_1 +3^{s-1} \omega_2=\tau^r \omega_1-\tau^{2s-1}\alpha_1 \\
	M^{(r/2)}_{G_2} & \text { if }  \lambda=3^{s} \omega_1 =\tau^r(\omega_2-\omega_1)-\tau^{2s-1}\alpha_2\\
	k_{\omega_{1}}^{(r/2)}  & \text { if } \lambda=(3^{s}-3^{i} \cdot 2)\omega_2 +3^{i+1}\omega_12=\tau^r \omega_1-\tau^{2i+1}\alpha_1\\
	k_{\omega_2}^{(r/2)} & \text { if } \lambda=3^{i+1}\omega_2 +(3^{s+1}-3^{i+1} \cdot 2)\omega_1=\tau^r \omega_2-\tau^{2i+1}\alpha_2  \\
	0 & \text { else. }
	\end{array}\right.\end{equation*}
	Here $M_{G_2}$ denotes the 2-dimensional indecomposable $B$-module with head $k_{\omega_{1}}$ and socle $k_{\omega_{2}-\omega_{1}}$ (cf. \cite[2.2]{BNP04b}). Moreover, the last two non-zero instances only occur for $s \geq 2$ (or $r \geq 5$).
\end{theorem} 
\begin{proof}
    The second equality in each case identifying two forms of $\lambda$ is readily verifiable, recalling $\tau(\omega_1)=\omega_2$. Thus, we focus on proving that $\lambda$ must be equal to one of the weights given by the first equality in each case.
	We consider the LHS spectral sequence
	\begin{displaymath}
	E_{2}^{i,j} = \opH^i(B_{r/2}/B_{\tau}, \opH^j(B_{\tau}, \lambda)) \Rightarrow \opH^{i+j}(B_{r/2}, \lambda)
	\end{displaymath}
	and the corresponding five-term exact sequence
	\begin{displaymath}
	0 \to E^{1,0} \to E^1 \to E^{0,1} \to E^{2,0} \to E^2.
	\end{displaymath}
	As before, we will identify $E^1$ with either $E^{0,1}$ or $E^{1,0}$ and we calculate all of the non-zero cases in this way. We begin by fixing some notation. Since $\lambda \in X_{r/2}(T)$, it has a unique $\tau$-adic expansion and we write $\lambda=  \sum_{i=0}^{r-1} \tau^{i} \lambda_{i}$, with $\lambda_i$ $\tau$-restricted weights. Then, $\lambda=\lambda_{0}+ \tau(\lambda')$, for $\lambda'= \sum_{j=1}^{r-1} \tau^{j-1} \lambda_{j}.$ Suppose $E^{0,1}\neq 0$ and consider the $E^{0,1}$-term.
	We have
	\begin{equation*}
	\begin{aligned}
	E^{0,1} & = \Hom_{B_{r/2}/B_{\tau}}(k, \opH^1(B_{\tau}, \lambda)) \\
	        & \cong \Hom_{B_{r/2}/B_{\tau}}(k, \opH^1(B_{\tau}, \lambda_{0}) \otimes \tau(\lambda')).
	\end{aligned}
	\end{equation*}
There is only one $\tau$-restricted weight for which $\opH^1(B_{\tau}, \lambda_0) \neq 0$, namely $\lambda_0=\omega_1$. In this case, we obtain
\begin{displaymath}
\opH^{1}(B_{\tau}, \omega_1) \cong k_{\omega_2-\omega_1}^{(\tau)} \oplus k_{\omega_1}^{(\tau)}.
\end{displaymath}
Hence
\begin{equation*}
	\begin{aligned}
	E^{0,1} & = \Hom_{B_{r/2}/B_{\tau}}( k, (k_{\omega_1}^{(\tau)} \oplus k_{\omega_2-\omega_1}^{(\tau)}) \otimes \tau(\lambda')) \\
	& \cong \Hom_{B_{(r-1)/2}}( k, (k_{\omega_1}^{(\tau)} \oplus k_{\omega_2-\omega_1}^{(\tau)}) \otimes k_{\lambda'}^{(\tau)})  \\
	& \cong \Hom_{B_{(r-1)/2}}( k, k_{\omega_1+\lambda'}^{(\tau)} \oplus k_{\omega_2-\omega_1+\lambda'}^{(\tau)}).
	\end{aligned}
	\end{equation*}
	
	Similarly to the proof of Theorem \ref{BrC2}, $\Hom_{B_{(r-1)/2}}( k, k_{\omega_1+\lambda'}^{(\tau)} \oplus k_{\omega_2-\omega_1+\lambda'}^{(\tau)})$ is non-zero if at least one of $\omega_1+\lambda'$ and $\omega_2-\omega_1+\lambda'$ belongs to $\tau^{r-1}X(T)$. Moreover, $\Hom_{B_{(r-1)/2}}( k, k_{\omega_1+\lambda'}^{(\tau)} \oplus k_{\omega_2-\omega_1+\lambda'}^{(\tau)})$ is at most  one-dimensional: since $\omega_2-2\omega_1 \notin \tau^{r-1}X(T)$, at most one of  $\omega_1+\lambda'$ and $\omega_2-\omega_1+\lambda'$ lies in $\tau^{r-1}X(T)$.  Thus, we consider both cases in turn to determine the possible values of $\lambda$ and $E^{0,1}$.
First, suppose $\omega_2-\omega_{1}+\lambda' \in \tau^{r-1} X(T)$. Since $p=3$, we have $\lambda'=(a3^s+1)\omega_1+(b3^s-1)\omega_2 \in X_s(T)$. It immediately follows that we must have $a=0$, $b=1$, in which case $\lambda'=\omega_1+(3^s-1)\omega_2$, giving $\lambda=(3^{s+1}-2)\omega_1+\omega_2$ and
\[E^{0,1}=\Hom_{B_{r/2}/B_{\tau}}(k, k_{\tau(\omega_1+\lambda')}\oplus k_{\tau(\omega_2-\omega_1+\lambda')}).\]
	The second term in the target of the $\Hom$ is $k_{\tau(\omega_2-\omega_1+\omega_1+(3^s-1) \omega_2)}=k_{3^s\tau(\omega_2)}$. Thus $E^{0,1}\cong k_{3^s\tau(\omega_2)}=(k_{\omega_2})^{(r/2)}$.
	
	In the case $\omega_{1}+\lambda' \in \tau^{r-1} X(T)$, a similar argument leads us to conclude that $E^{0,1}=k_{\omega_1}^{(r/2)}$ for $\lambda=\omega_1 + (3^{s}-1) \omega_2$. 

To conclude, for $\lambda \in X_{r/2}(T)$,
\begin{equation*}E^{0,1} \cong \left\{\begin{array}{ll}
	 k_{\omega_{1}}^{(r/2)}  & \text { if } \lambda=\omega_1 + (3^{s}-1) \omega_2  \\
	 k_{\omega_{2}}^{(r/2)}  & \text { if } \lambda=(3^{s+1}-2)\omega_1+\omega_2\\
	 0 & \text { else. }
	 \end{array}\right.\end{equation*}
 Now suppose $E^{1,0}\neq 0$. We have 
	 \begin{align*}E^{1,0}&=\opH^{1}(B_{r/2} / B_{\tau}, \Hom_{B_{\tau}}(k, \lambda)), \\
	 &=\opH^{1}(B_{r/2} / B_{\tau}, \Hom_{B_{\tau}}(k, \lambda_0)\otimes \tau(\lambda'))
	 \end{align*}
	 so $\lambda_0=0$ and $\lambda=\tau(\lambda')$. Thus $E^{1,0}\cong \opH^{1}( B_{s},  {\lambda'} ^{(\tau)}) \cong  \opH^{1}( B_{s},  {\lambda'})^{(\tau)}$
	 for $\lambda=\tau {\lambda'}$. Notice that since $r-1=2s>0$, $B_{(r-1)/2}=B_{s}$ is a classical Frobenius kernel and $\opH^{1}( B_{s},  {\lambda'})$ is the $B_s$-cohomology for ${\lambda'} \in X_{s}(T)$ computed in \cite[Theorem 2.7]{BNP04b}. We have
	 \begin{equation*}\opH^{1}(B_{s}, {\lambda'} ) \cong\left\{\begin{array}{ll}
	 k_{\omega_{1}}^{(s)}  & \text { if } {\lambda'}=3^{s-1} (\omega_1 + \omega_2) \\
	 M^{(s)}_{G_2} & \text { if }  {\lambda'}=3^{s-1} \omega_2 \\
	 k_{\omega_{j}}^{(s)}  & \text { if } {\lambda'}=3^{s}\omega_{j} -3^{i}\alpha_{j}, j \in \left\lbrace 1,2 \right\rbrace, 0 \leq i \leq s-2\\
	 0 & \text { else. }
	 \end{array}\right.\end{equation*}
	 where $M_{G_2}$ has the structure as claimed in the statement of the theorem. We note the implicit constraints on $s$ in the different cases. Thus,
	 \begin{equation*}E^{1,0} \cong \opH^{1}(B_{s}, {\lambda'} )^{(\tau)} \cong \left\{\begin{array}{ll}
	 k_{\omega_{1}}^{(r/2)}  & \text { if } {\lambda'}=3^{s-1} (\omega_1 + \omega_2) \\
	 M^{(r/2)}_{G_2} & \text { if }  {\lambda'}=3^{s-1} \omega_2 \\
	 k_{\omega_{j}}^{(r/2)}  & \text { if } {\lambda'}=3^{s}\omega_{j} -3^{i}\alpha_{j}, j \in \left\lbrace 1,2 \right\rbrace, 0 \leq i \leq s-2\\
	 0 & \text { else. }
	 \end{array}\right.\end{equation*}
	 We can recover $\lambda$ from ${\lambda'}$, recalling $\alpha_1=2\omega_1-\omega_2$ and $\alpha_2=-3\omega_1+2\omega_2$. For instance, if ${\lambda'} =3^{s-1} (\omega_1 + \omega_2)$, then $\lambda=\tau {\lambda'}=3^s \omega_1 + 3^{s-1} \omega_2$. Note that the other cases follow similarly. 
	 
Finally, note that there is no choice of $\lambda$ for which $E^{0,1}$ and $E^{1,0}$ are simultaneously non-zero. Hence, if $E^{1,0} \neq 0$, then $E^{0,1}=0$ so $E^1 \cong E^{1,0}$. Alternatively, if $E^{0,1} \neq 0$, then $\lambda=\omega_1+(3^s-1)\omega_2$ or $\lambda=(3^{s+1}-2)\omega_1+\omega_2$, according to the earlier discussion. Note that in either case, $\lambda \notin \tau X(T)$, pushing $\Hom_{B_{\tau}}(k, \lambda)=0.$ Hence $E^{1,0}=E^{2,0}=0$, meaning that $E^1 \cong E^{0,1}$.
\end{proof}	
Now, for completeness, for a general $\lambda \in X(T)$, not necessarily lying in $X_{r/2}$, we proceed as in \cite[2.8]{BNP04b}. First, we make the following observation and we note that the proof is identical to the proof of Corollary \ref{C2weightform}.
\begin{cor} \label{G2weightform}
    Let $\lambda \in X(T)$. Then $\opH^1(B_{r/2}, \lambda) \neq 0$ if and only if $\lambda= \tau^r \omega - \tau^i \alpha$, for some weight $\omega \in X(T)$, and $\alpha \in \Pi$ with $0 \leq i \leq 2s-1$ or $\lambda= \tau^r \omega - \beta$, for some weight $\omega \in X(T)$, and $\beta \in \Pi_s$.
\end{cor}
Now, we denote by $(\zeta, j)$ the pair $(\alpha, i)$ or $(\beta, 1)$, respectively, as defined in the previous corollary. 
Now, we write $\lambda=\tau^r \omega' - \tau^j \zeta + \tau^r \lambda_1$, for a given $\lambda=\tau^r \omega-\tau^j \zeta$. Supposing the $B_{r/2}$-cohomology does not vanish on $\lambda$, then $\omega'$ is as given in Theorem \ref{Br2G2} and $\lambda_1 \in X(T)$. Then, set $\lambda_1=\omega-\omega'$ and we obtain
\begin{equation*}
\begin{aligned}
    \opH^1(B_{r/2}, \lambda) & \cong \opH^1(B_{r/2}, \lambda_0) \otimes k_{\lambda_1}^{(r/2)} \\
    & \cong \opH^1(B_{r/2}, \tau^r \omega' - \tau^j \zeta) \otimes k_{\omega-\omega'}^{(r/2)}.
\end{aligned}
\end{equation*}
One then substitutes the results from Theorem \ref{Br2G2}. We omit the details for brevity and obtain
\begin{theorem} \label{Br2G2gen}
Let $\lambda \in X(T)$ and $0 \leq i \leq s-2$. Then
	\begin{equation*}\opH^{1}\left(B_{r/2}, \lambda \right) \cong\left\{\begin{array}{ll}
	k_{\omega}^{(r/2)}  & \text { if } \lambda =\tau^r \omega - \tau^{2s-1}\alpha_1, \omega \in X(T) \\
	M^{(r/2)}_{G_2} \otimes k_{\omega+\omega_1-\omega_2}^{(r/2)} & \text { if }   \lambda =\tau^r \omega - \tau^{2s-1}\alpha_2, \omega \in X(T)\\
	k_{\omega}^{(r/2)}  & \text { if } \lambda =\tau^r \omega - \tau^{2i+1}\alpha_j, \omega \in X(T), \alpha_j \in \Pi\\
	0 & \text { else. }
	\end{array}\right.\end{equation*}
\end{theorem} 
\subsubsection*{$G_{r/2}$-cohomology of induced modules}

Using Kempf's vanishing theorem, Theorem \ref{btau-g2}, Theorem \ref{Br2G2} and (\ref{indBG}), we compute $\opH^{1}(G_{r/2}, \opH^{0}(\lambda))$ for $\lambda \in X_{r/2}.$ 
Furthermore, we note that, by \cite[3.1, Theorem (B)]{BNP04b},  $\Ind_B^G(M_{G_{2}})= \opH^0(\omega_1)$.

In the case $r=1$ we obtain
\begin{theorem} \label{gtau-g2}
Let $\lambda \in X_{\tau}(T)$. Then
\begin{equation*}\opH^{1}\left(G_{\tau}, \opH^0(\lambda) \right)^{(-\tau)} \cong\left\{\begin{array}{ll}
\opH^0(\omega_1) & \text { if } \lambda= \omega_1 \\
0 & \text { else. }
\end{array}\right.\end{equation*}
\end{theorem}

Now, assume $r>1$.
\begin{theorem} \label{g2gr2}
	Let $\lambda \in X_{r/2}(T)$ and $0 \leq i \leq s-2$. Then
	\begin{equation*}\opH^{1}(G_{r/2}, \opH^{0}(\lambda))^{(-r/2)} \cong\left\{\begin{array}{ll}
	\opH^{0}(\omega_1)  & \text { if } \lambda=\omega_1 + (3^{s}-1) \omega_2=\tau^r \omega_1-\beta_2  \\
	 \opH^{0}(\omega_2)  & \text { if } \lambda=(3^{s+1}-2)\omega_1+\omega_2=\tau^r \omega_2-\beta_1\\
    \opH^{0}(\omega_1)  & \text { if } \lambda=3^{s} \omega_1 +3^{s-1} \omega_2=\tau^r \omega_1-\tau^{2s-1} \alpha_1 \\
	\opH^{0}(\omega_1)  & \text { if } \lambda= 3^{s}\omega_1=\tau^r (\omega_2-\omega_1)-\tau^{2s-1} \alpha_2  \\
	\opH^{0}(\omega_1)  & \text { if } \lambda=(3^{s+1}-3^{i} \cdot 2) \omega_2 +3^{i+1}\omega_1=\tau^r \omega_1-\tau^{2i+1} \alpha_1  \\
	\opH^{0}(\omega_2)  & \text { if }  \lambda=3^{i+1}\omega_2 +(3^{s+1}-3^{i+1} \cdot 2)\omega_1=\tau^r \omega_2-\tau^{2i+1} \alpha_2  \\
	0 & \text { else. }
	\end{array}\right.\end{equation*} 
\end{theorem}

Lastly, based on Theorem \ref{Br2G2gen}, one may calculate $\opH^1(G_{r/2}, \opH^0(\lambda))$ in terms of induced modules for all dominant weights $\lambda$, by applying the induction functor $\Ind_B^G$. We handle the only non-obvious case in the following remark.
\begin{remark}
Let $\tau^r \omega - \tau^{2s-1} \alpha_2 \in X(T)_+$. Then $\langle \omega, \alpha_1^{\vee} \rangle \geq -1$ and 
$\langle \omega, \alpha_2^{\vee} \rangle \geq 1$. In this case, by \cite[Proposition 3.4 (A)]{BNP04b}, we note that
\begin{itemize}
    \item[(i)] if $\langle \omega, \alpha_1^{\vee} \rangle \geq 0$, then $\Ind_{B}^G(M_{G_2} \otimes k_{\omega+\omega_1-\omega_2})$
    has a filtration with factors satisfying the following short exact sequence
    \begin{displaymath}
    0 \to \opH^0(\omega) \to \Ind_{B}^G(M_{C_2} \otimes k_{\omega+\omega_1-\omega_2}) \to \opH^0(\omega+2\omega_1-\omega_2) \to 0.
    \end{displaymath}
    \item[(ii)] if $\langle \omega, \alpha_1^{\vee} \rangle =-1$, then $\Ind_{B}^G(M_{G_2} \otimes k_{\omega+\omega_1-\omega_2})\cong \opH^0(\omega+2\omega_1-\omega_2)$.
\end{itemize}
\end{remark}
\subsubsection*{$G_{r/2}$-cohomology with coefficients in simple modules}  In this subsection, we make use of the  $G_1$-cohomology with coefficients in simple modules, computed in \cite[Proposition 3.5]{Sin94b}, to calculate $\opH^1(G_{s}, L(\lambda))$, for a positive integer $s$ and $\lambda \in X_{s}(T)$.
\begin{theorem} \label{g2gs-simple}
	Let $s$ be a positive integer, $\lambda \in X_s(T)$, $0 \leq i \leq s-1$ and $0 \leq j \leq s-2$. Then
	\begin{equation*}\opH^{1}(G_s, L(\lambda))^{(-s)} \cong\left\{\begin{array}{ll}
	L(\omega_1)  & \text { if } \lambda=3^{s-1} \omega_2 \\
	k  & \text { if } \lambda=3^i(\omega_1+\omega_2) \\
    k  & \text { if } \lambda=3^{j}(\omega_2+3\omega_1)  \\
	0 & \text { else. }
	\end{array}\right.\end{equation*} 
	Note that it is implicit in the statement of the theorem that $s \geq 1$ or $s \geq 2$, depending on the case. 
\end{theorem}
\begin{proof}
We proceed inductively. When $s=1$, we refer the reader to \cite[Proposition 3.5]{Sin94b}.
We write $\lambda=\lambda_0+3^{s-1}\lambda_1$, for $\lambda_0 \in X_{s-1}$ and $\lambda_1 \in X_1.$ Suppose $s>1$ and consider the LHS spectral sequence corresponding to $G_{s-1} \lhd G_s$. The $E_2$-page is given by
\begin{displaymath}
E^{i,j}_2:=\opH^i(G_1, \opH^j(G_{s-1}, L(\lambda_0))^{(-s+1)} \otimes L(\lambda_1))^{(s-1)}.
\end{displaymath}
First, consider the $E^{1,0}$-term. We have
\begin{displaymath}
E^{1,0}_2=\opH^1(G_1, \Hom_{G_{s-1}}(k, L(\lambda_0))^{(-s+1)} \otimes L(\lambda_1))^{(s-1)}.
\end{displaymath}
Note that $E^{1,0} \neq 0$ if and only if $\lambda_0=0$, in which case we obtain 
\begin{equation*}E^{1,0}=\opH^{1}(G_1, L(\lambda_1))^{(s-1)} \cong\left\{\begin{array}{ll}
	L(\omega_1)^{(s)} & \text { if } \lambda_1=\omega_2 \\
	k  & \text { if } \lambda_1=\omega_1+\omega_2 \\
	0 & \text { else. }
	\end{array}\right.\end{equation*} 
(cf. \cite[Proposition 3.5]{Sin94b}). Therefore, recalling that $\lambda=\lambda_0+2^{s-1}\lambda_1$, we may conclude that for $\lambda \in X_s$,
\begin{equation*}E^{1,0}\cong\left\{\begin{array}{ll}
	L(\omega_1)^{(s)} & \text { if } \lambda=3^{s-1}\omega_2 \\
	k  & \text { if } \lambda=3^{s-1}(\omega_1+\omega_2) \\
	0 & \text { else. }
	\end{array}\right.\end{equation*}
Now consider the $E^{0,1}$-term. We have
\begin{displaymath}
E^{0,1}=\Hom_{G_1}( L(\lambda_1), \opH^1(G_{s-1}, L(\lambda_0))^{(-s+1)})^{(s-1)}.
\end{displaymath}
We consider each non-zero instance of $\opH^1(G_{s-1}, L(\lambda_0))$ in turn. For example, by the induction hypothesis, if $\lambda_0=3^{s-2}\omega_2$, then $E^{0,1}=\Hom_{G_1}( L(\lambda_1), L(\omega_1))^{(s-1)}$. Thus $E^{0,1} \neq 0$ if and only if $\lambda_1=\omega_1$; we conclude that $E^{0,1} \cong k$ for $\lambda=3^{s-2}(\omega_2+3\omega_1)$. The other cases follow similarly and we obtain, for $1 \leq i \leq s-2$ and $0 \leq j \leq s-3$
\begin{equation*}E^{0,1}\cong\left\{\begin{array}{ll}
	k & \text { if } \lambda=3^{s-2}(\omega_2+3\omega_1) \\
	k  & \text { if } \lambda=3^{i}(\omega_1+\omega_2)\\
	k  & \text { if } \lambda=3^{j}(\omega_2+3\omega_1) \\
	0 & \text { else. }
	\end{array}\right.\end{equation*}
Notice that there is no choice of $\lambda$ for which $E^{1,0}$ and $E^{0,1}$ are both non-zero. Hence if $E^{1,0} \neq 0$, then $E^{0,1}=0$, implying that $E^{1} \cong E^{1,0}$. Alternatively, suppose that $E^{0,1} \neq 0$. Then either $\lambda=3^{s-2}(\omega_2+3\omega_1)$, $\lambda=3^{i}(\omega_1+\omega_2)$ or $\lambda=3^{j}(\omega_2+3\omega_1)$. Observe that in either case, we obtain $\Hom_{G_{s-1}}(k, L(\lambda_0))=0$. Hence $E^{0,1} \neq 0$ implies $E^{1,0}=E^{2,0}=0$, and we have $E^1 \cong E^{0,1}$. 
\end{proof}
Next, we compute $\opH^{1}(G_{r/2}, L(\lambda))$ for $r$ an odd positive integer and $\lambda \in X_{r/2}$, making use of the previous theorem concerning the cohomology for classical Frobenius kernels.

If $r=1$, we refer the reader to Theorem \ref{gtau-g2}. Otherwise we obtain
\begin{theorem} \label{g2gr2-simple}
	Suppose $r=2s+1>1$ and let $\lambda \in X_{r/2}(T)$, for $0 \leq i \leq s-1$. Then
    \begin{equation*}\opH^{1}(G_{r/2}, L(\lambda))^{(-r/2)} \cong\left\{\begin{array}{ll}
	L(\omega_1)  & \text { if } \lambda=3^s \omega_1 \\
	k  & \text { if } \lambda=3^i(\omega_1+\omega_2) \\
	k  & \text { if } \lambda=3^i(\omega_2+3\omega_1) \\
	0 & \text { else. }
	\end{array}\right.\end{equation*} 
\end{theorem}
\begin{proof}
For $\lambda \in X_{r/2}$, write $\lambda=\lambda_0+2^s\lambda_1$, for $\lambda_0 \in X_s$ and $\lambda_1 \in X_{\tau}$. Consider the LHS spectral sequence corresponding to $G_{s} \lhd G_{r/2}$. The $E_2$-page is given by
\begin{displaymath}
E^{i,j}_2:=\opH^i(G_{\tau}, \opH^j(G_{s}, L(\lambda_0))^{(-s)} \otimes L(\lambda_1))^{(s)}.
\end{displaymath}
First, consider the $E^{1,0}$-term. We have
\begin{displaymath}
E^{1,0}=\opH^1(G_{\tau}, \Hom_{G_{s}}(k, L(\lambda_0))^{(-s)} \otimes L(\lambda_1))^{(s)}.
\end{displaymath}
Note that $E^{1,0} \neq 0$ if and only if $\lambda_0=0$, in which case we obtain 
\begin{equation*}E^{1,0}=\opH^{1}(G_{\tau}, L(\lambda_1))^{(s)} \cong\left\{\begin{array}{ll}
	L(\omega_1)^{(r/2)} & \text { if } \lambda_1=\omega_1\\
	0  & \text { else. } 
	\end{array}\right.\end{equation*} 
(cf. Theorem \ref{gtau-g2} and \cite[Lemma 3.2]{Sin94b}). Next, consider the $E^{0,1}$-term:
\begin{displaymath}
E^{0,1}=\Hom_{G_{\tau}}(L(\lambda_1), \opH^1(G_{s}, L(\lambda_0))^{(-s)})^{(s)}.
\end{displaymath}
We take each non-zero instance of $\opH^1(G_{s}, L(\lambda_0))^{(-s)}$ from Theorem \ref{g2gs-simple} in turn. If $\lambda_0=3^{s-1}\omega_2$, then 
$E^{0,1}=\Hom_{G_{\tau}}(L(\lambda_1), L(\omega_1))^{(s)}$. Thus $E^{0,1} \neq 0$ if and only if $\lambda_1=\omega_1$. We obtain $E^{0,1} \cong k$ for $\lambda=3^{s-1}(\omega_2+3\omega_1)$. The other cases are similar. We get, for $0 \leq i \leq s-1$ and $0 \leq j \leq s-2$,
\begin{equation*}E^{0,1}\cong\left\{\begin{array}{ll}
	k  & \text { if } \lambda=3^{s-1}(\omega_2+3\omega_1) \\
	k  & \text { if } \lambda=3^{i}(\omega_1+\omega_2) \\
	k & \text { if } \lambda=3^{j}(\omega_2+3\omega_1) \\
	0 & \text { else. }
	\end{array}\right.\end{equation*}
Note that there is no choice of $\lambda$ for which $E^{1,0}$ and $E^{0,1}$ are both non-zero. Hence if $E^{1,0} \neq 0$, then $E^{0,1}=0$, implying that $E^{1} \cong E^{1,0}$. Alternatively, suppose that $E^{0,1} \neq 0$ and notice that for each choice of $\lambda$ above that $\lambda_0 \neq 0$, which forces $\Hom_{G_s}(k,L(\lambda_0))=0$. Hence $E^{1,0}=E^{2,0}=0$, meaning that $E^1 \cong E^{0,1}$. 
\end{proof}
\subsection{$F_4$ in characteristic 2} \label{subsecF4}
Let $G$ be simply connected of type $F_{4}$ over $k$ of characteristic $2$.
Following \cite[Planche VIII]{Bou82}, let $\Phi=\{ \pm \epsilon_{i}, \pm \epsilon_{i} \pm \epsilon_{j}, \frac{1}{2}(\pm \epsilon_{1} \pm \epsilon_{2} \pm \epsilon_{3} \pm \epsilon_{4}) \}$ be the roots of a system of type $F_4$. Writing $\epsilon_{1}=(1,0,0,0)$, $\epsilon_{2}=(0,1,0,0)$,  $\epsilon_{3}=(0,0,1,0)$ and $\epsilon_{4}=(0,0,0,1)$, a base of simple roots is $\Pi:=\{\alpha_1, \alpha_2,\alpha_3, \alpha_4\}$ with
$\alpha_1=(0,1,-1,0)$, $\alpha_2=(0,0,1,-1)$, $\alpha_3=(0,0,0,1)$ and $\alpha_4=\frac{1}{2}(1,-1,-1,-1)$; furthermore, the corresponding fundamental dominant weights are $\omega_1=(1,1,0,0)$, $\omega_2=(2,1,1,0)$, $\omega_3=\frac{1}{2}(3,1,1,1)$ and $\omega_4=(1,0,0,0)$. Then one can check that a set of simple roots of $\Phi_s$ is $\Pi_s:=\left\lbrace \alpha_{3}, \alpha_4, \alpha_2+\alpha_3, \alpha_1+\alpha_2+\alpha_3\right\rbrace $, with $\alpha_4$ being the central node in the Dynkin Diagram. We shall denote these simple roots by $\beta_1=\alpha_3=(0,0,0,1)$, $\beta_2=\alpha_4=\frac{1}{2}(1,-1,-1,-1)$, $\beta_3=\alpha_2+\alpha_3=(0,0,1,0)$ and $\beta_4=\alpha_1+\alpha_2+\alpha_3=(0,1,0,0)$. The special isogeny induces a $\mathbb{Z}$-linear map $\tau^{\ast}$ as before, under which $\omega_4 \mapsto \omega_1 \mapsto 2\omega_4$ and $\omega_3 \mapsto \omega_2 \mapsto 2\omega_3$. We henceforth abuse notation, writing $\tau$ instead of $\tau^{\ast}$.  Consequently, the $\tau$-restricted weights are $0$, $\omega_3$, $\omega_4$ and $\omega_3 + \omega_4.$

\subsubsection*{$B_{\tau}$-cohomology}

For a given $\lambda \in X_{r/2}$, we write $\lambda = \lambda_{0} + \tau(\lambda_{1})$, for $\lambda_0 \in X_{\tau}(T)$ and $\lambda_1 \in X_s(T)$, such that $\opH^1(B_{\tau}, \lambda) \cong \opH^1(B_{\tau}, \lambda_{0}) \otimes  \tau(\lambda_{1})$. Thus, it suffices to compute $\opH^1(B_{\tau}, \lambda_{0})$, for $\lambda_{0} \in X_{\tau}(T)$. 

\begin{theorem} \label{btau-f4}
Let $\lambda_0 \in X_{\tau}(T)$. Then
\begin{equation*}\opH^{1}(B_{\tau}, \lambda_{0}) \cong\left\{\begin{array}{ll}
k_{\omega_{4}}^{(\tau)} \oplus k_{\omega_2-\omega_3}^{(\tau)} \oplus k_{\omega_3-\omega_4}^{(\tau)} & \text { if } \lambda_{0}= \omega_4 \\
k_{\omega_1}^{(\tau)} & \text { if } \lambda_{0}=\omega_3 \\
0 & \text { else. }
\end{array}\right.\end{equation*}
\end{theorem}

\begin{proof}
Much like in the other cases, regarded as a $T$-module, $\opH^1(U_{\tau}, k) \otimes \lambda_{0}$ is the direct sum of certain $k_{\beta_i + \lambda_{0}}$, for $\beta_i \in \Pi_s$. Given the fact that such a summand yields a non-zero contribution to  $\opH^1(B_{\tau}, \lambda_{0})$ if and only if $\beta_i + \lambda_{0} \in \tau X(T)$, we now inspect which of these weights belong to $\tau X(T).$

To begin with, suppose $\lambda_{0}=0$. It is readily verified that we have no non-zero contribution. Therefore, $\opH^1(B_{\tau}, k)=0$.
	
Then, suppose $\lambda_{0}=\omega_4$. We have
	\begin{displaymath}
	\beta_1+\omega_4=(1,0,0,1)=-\omega_2+2\omega_3=\tau(\omega_2-\omega_3).
	\end{displaymath}
	\begin{displaymath}
	\beta_2+\omega_4=\frac{1}{2}(3,-1,-1,-1)=-\omega_3+3\omega_4 \notin \tau X(T).
	\end{displaymath}
	\begin{displaymath}
	\beta_3+\omega_4=(1,0,1,0)=-\omega_1+\omega_2=\tau(\omega_3-\omega_4).
	\end{displaymath}
	\begin{displaymath}
	\beta_4+\omega_4=(1,1,0,0)=\omega_1=\tau(\omega_4).
	\end{displaymath}
	 Hence,  
	 \begin{equation*}
	 \begin{aligned}
	 \opH^1(B_{\tau}, \omega_4) & \cong \left[ \bigoplus_{i} k_{\beta_i + \omega_1} \right]^{T_{\tau}} \cong \left[ k_{\tau(\omega_4)} \oplus k_{\tau(\omega_2-\omega_3)} \oplus k_{\tau(\omega_3-\omega_4)} \right]^{T_{\tau}} \\
	 & \cong  k_{\omega_4}^{(\tau)} \oplus k_{\omega_2-\omega_3}^{(\tau)} \oplus k_{\omega_3-\omega_4}^{(\tau)}.
	 \end{aligned}
	 \end{equation*}
	
Now let $\lambda_{0}=\omega_3$ and we obtain
	\begin{displaymath}
	\beta_1+\omega_3=\frac{1}{2}(3,1,1,3)=-\omega_2+3\omega_3-\omega_4 \notin \tau X(T).
	\end{displaymath}
	\begin{displaymath}
	\beta_2+\omega_3=(2,0,0,0)=2\omega_4=\tau(\omega_1).
	\end{displaymath}
	\begin{displaymath}
	\beta_3+\omega_3=\frac{1}{2}(3,1,3,1)=-\omega_1+\omega_2+\omega_3-\omega_4 \notin \tau X(T).
	\end{displaymath}
	\begin{displaymath}
	\beta_4+\omega_3=\frac{1}{2}(3,3,1,1)=\omega_1+\omega_3-\omega_4 \notin \tau X(T).
	\end{displaymath}
	Then,  $\opH^1(B_{\tau}, \omega_3) \cong k_{\omega_1}^{(\tau)} $.
	
Finally, for $\lambda_{0}=\omega_3+\omega_4$, we get
	\begin{displaymath}
	\beta_1+\omega_3+\omega_4=\frac{1}{2}(5,1,1,3)=-\omega_2+3\omega_3 \notin \tau X(T).
	\end{displaymath}
	\begin{displaymath}
	\beta_2+\omega_3+\omega_4=(3,0,0,0)=3\omega_4 \notin \tau X(T).
	\end{displaymath}
	\begin{displaymath}
	\beta_3+\omega_3+\omega_4=\frac{1}{2}(5,1,3,1)=-\omega_1+\omega_2+\omega_3 \notin \tau X(T).
	\end{displaymath}
	\begin{displaymath}
	\beta_4+\omega_3+\omega_4=\frac{1}{2}(5,3,1,1)=\omega_1+\omega_3 \notin \tau X(T).
	\end{displaymath}
	Then, $\opH^1(B_{\tau}, \omega_3+\omega_4)=0 $.
\end{proof}
\subsubsection*{$B_{r/2}$-cohomology}In this subsection, we extend the results of the previous section to calculate $\opH^1(B_{r/2}, \lambda)$, for $\lambda \in X_{r/2}(T)$.

If $r=1$, we direct the reader to Theorem \ref{btau-f4}.

\begin{theorem} \label{Br2F4}
	Suppose $r=2s+1>1$ and let $\lambda \in X_{r/2}(T)$. Then, for $ 0 \leq i \leq s-2$, we have
	\begin{equation*}\opH^{1}\left(B_{r/2}, \lambda \right) \cong\left\{\begin{array}{ll}
	k_{\omega_{1}}^{(r/2)}  & \text { if } \lambda =\omega_3+2(2^s-1)\omega_4=\tau^r \omega_1-\beta_2  \\
	 k_{\omega_{2}}^{(r/2)}  & \text { if } \lambda =\omega_4+\omega_2+2(2^s-1)\omega_3=\tau^r \omega_2-\beta_1\\
	 k_{\omega_{3}}^{(r/2)}  & \text { if } \lambda =\omega_4+\omega_1+(2^s-1)\omega_2=\tau^r \omega_3-\beta_3\\
	 k_{\omega_{4}}^{(r/2)}  & \text { if } \lambda =\omega_4+(2^s-1) \omega_1=\tau^r \omega_4-\beta_4\\
	k_{\omega_{1}}^{(r/2)}  & \text { if } \lambda=2^{s} \omega_3=\tau^r \omega_1-\tau^{2s-1}\alpha_1 \\
	k_{\omega_3}^{(r/2)} & \text { if } \lambda=2^{s} \omega_3+ 2^{s-1}\omega_1=\tau^r \omega_3-\tau^{2s-1}\alpha_3 \\
	k_{\omega_4}^{(r/2)} & \text { if } \lambda= 2^{s-1}\omega_2=\tau^r \omega_4-\tau^{2s-1}\alpha_4 \\
	M^{(r/2)}_{F_4} & \text { if }  \lambda=2^{s} \omega_4=\tau^r(\omega_2-\omega_3)-\tau^{2s-1}\alpha_2 \\
	k_{\omega_{1}}^{(r/2)}  & \text { if } \lambda=(2^{s+1}-2^{i+2})\omega_4 +2^{i+1}\omega_3=\tau^r \omega_1-\tau^{2i+1}\alpha_1\\
	k_{\omega_2}^{(r/2)} & \text { if } \lambda=2^{i+1}\omega_2 +(2^{s+1}-2^{i+2})\omega_3 + 2^{i+1}\omega_4=\tau^r \omega_2-\tau^{2i+1}\alpha_2 \\
	k_{\omega_3}^{(r/2)} & \text { if } \lambda=2^{i}\omega_1 +(2^{s}-2^{i+1})\omega_2 + 2^{i+1}\omega_3=\tau^r \omega_2-\tau^{2i+1}\alpha_2 \\
	k_{\omega_4}^{(r/2)} & \text { if } \lambda=(2^{s}-2^{i+1})\omega_1 +2^{i}\omega_2=\tau^r \omega_4-\tau^{2i+1}\alpha_4\\
	0 & \text { else. }
	\end{array}\right.\end{equation*}
	Here $M_{F_4}$ denotes the 3-dimensional indecomposable $B$-module  with the following factors: head $k_{\omega_{4}}$, $k_{\omega_{3}-\omega_{4}}$ and socle $k_{\omega_{2}-\omega_{3}}$ (cf. \cite[2.2]{BNP04b}). We underline that the last four non-zero instances only occur when $s \geq 2$ (or $r \geq 5$). 
\end{theorem} 

\begin{proof}
    The second equality in each case identifying two forms of $\lambda$ follows immediately, recalling $\tau(\omega_4)=\omega_1$ and $\tau(\omega_3)=\omega_2$. We thus show that $\lambda$ must be equal to one of the weights given by the first equality in each case. We consider the LHS spectral sequence
	\begin{displaymath}
	E_{2}^{i,j} = \opH^i(B_{r/2}/B_{\tau}, \opH^j(B_{\tau}, \lambda)) \Rightarrow \opH^{i+j}(B_{r/2}, \lambda)
	\end{displaymath}
	and the corresponding five-term exact sequence
	\begin{displaymath}
	0 \to E^{1,0} \to E^1 \to E^{0,1} \to E^{2,0} \to E^2.
	\end{displaymath}
	Much like in the previous subsections, we shall identify $E^1$ with either $E^{0,1}$ or $E^{1,0}$, in order to determine all of the non-zero cases. We must first fix some notation.
	Since $\lambda \in X_{r/2}(T)$, we may uniquely write $\lambda= \sum_{i=0}^{r-1} \tau^{i} \lambda_{i}$, where $\lambda_i$ are $\tau$-restricted. Then,  $\lambda=\lambda_{0}+ \tau(\lambda')$, for $\lambda'= \sum_{j=1}^{r-1} \tau^{j-1} \lambda_{j}$. Suppose $E^{0,1}\neq 0$ and we have
	We have
	\begin{equation*}
	\begin{aligned}
	E^{0,1} & = \Hom_{B_{r/2}/B_{\tau}}(k, \opH^1(B_{\tau}, \lambda)) \\
	        & \cong \Hom_{B_{r/2}/B_{\tau}}(k, \opH^1(B_{\tau}, \lambda_{0}) \otimes \tau(\lambda')).
	\end{aligned}
	\end{equation*}
	There are two $\tau$-restricted weights for which  $\opH^{1}(B_{\tau}, \lambda_{0}) \neq 0$, namely $\omega_4$ and $\omega_3$, and we consider each case in turn.
	
	First, suppose $\lambda_{0}=\omega_{4}$ and we have  $\opH^{1}(B_{\tau}, \omega_4) = k_{\omega_4}^{(\tau)} \oplus k_{\omega_2-\omega_3}^{(\tau)} \oplus k_{\omega_3-\omega_4}^{(\tau)}$. Hence
	\begin{equation*}
	\begin{aligned}
	E^{0,1} & = \Hom_{B_{r/2}/B_{\tau}}( k, (k_{\omega_4}^{(\tau)} \oplus k_{\omega_2-\omega_3}^{(\tau)} \oplus k_{\omega_3-\omega_4}^{(\tau)}) \otimes \tau(\lambda')) \\
	& \cong \Hom_{B_{(r-1)/2}}( k, (k_{\omega_4}^{(\tau)} \oplus k_{\omega_2-\omega_3}^{(\tau)} \oplus k_{\omega_3-\omega_4}^{(\tau)}) \otimes k_{\lambda'}^{(\tau)})  \\
	& \cong \Hom_{B_{(r-1)/2}}( k, k_{\omega_4+\lambda'}^{(\tau)} \oplus k_{\omega_2-\omega_3+\lambda'}^{(\tau)} \oplus k_{\omega_3-\omega_4+\lambda'}^{(\tau)}).
	\end{aligned}
	\end{equation*}
    Notice that $\Hom_{B_{(r-1)/2}} (k, k_{\omega_4+\lambda'}^{(\tau)} \oplus k_{\omega_2-\omega_3+\lambda'}^{(\tau)} \oplus k_{\omega_3-\omega_4+\lambda'}^{(\tau)}) $ is either zero or one-dimensional: at most one of $\omega_4+\lambda'$, $\omega_2-\omega_3+\lambda'$ or $ \omega_3-\omega_4+\lambda' \in \tau^{r-1}X(T)$.
	
	First, suppose $\omega_4+\lambda' \in \tau^{r-1}X(T)$. As $p=2$, we obtain $\lambda'=a2^{s}  \omega_1+ b2^{s}  \omega_2+c2^{s} \omega_3+(d2^{s}-1) \omega_4 \in X_s(T)$. It follows that we must have $a=b=c=0$ and $d=1$, pushing $\lambda=\lambda_0+\tau(\lambda')=\omega_4+(2^s-1)\omega_1$ and 
	\[E^{0,1}=\Hom_{B_{r/2}/B_{\tau}}(k, k_{\tau(\omega_4+\lambda')}\oplus k_{\tau(\omega_2-\omega_3+\lambda')}\oplus k_{\tau(\omega_3-\omega_4+\lambda')}).\]
	The first term in the target of the $\Hom$ is $k_{\omega_1+(2^s-1) \omega_1}=k_{2^s\omega_1}$. Thus $E^{0,1}\cong k_{2^s\omega_1}=k_{\omega_4}^{(r/2)}$.

	Now assume $\omega_2-\omega_{3}+\lambda' \in \tau^{r-1} X(T)$ and a similar argument leads us to conclude that $E^{0,1}=k_{\omega_2}^{(r/2)}$ for $\lambda=\omega_2+2(2^s-1)\omega_3+\omega_4$.

	Lastly, suppose $\omega_3-\omega_{4}+\lambda' \in \tau^{r-1} X(T)$, and we obtain $E^{0,1}=k_{\omega_3}^{(r/2)}$ for $\lambda=\omega_1+(2^s-1)\omega_2+\omega_4$.
	
Analogously, the case where $\lambda_{0} =\omega_3$ leads to $E^{0,1} \cong  k_{\omega_1}^{(r/2)}$, when $\lambda=\omega_3+2(2^s-1)\omega_4$. 

Overall, we conclude that for $\lambda \in X_{r/2}(T)$,
\begin{equation*}E^{0,1} \cong \left\{\begin{array}{ll}
	 k_{\omega_{1}}^{(r/2)}  & \text { if } \lambda =\omega_3+2(2^s-1)\omega_4  \\
	 k_{\omega_{2}}^{(r/2)}  & \text { if } \lambda =\omega_4+\omega_2+2(2^s-1)\omega_3\\
	 k_{\omega_{3}}^{(r/2)}  & \text { if } \lambda =\omega_4+\omega_1+(2^s-1)\omega_2\\
	 k_{\omega_{4}}^{(r/2)}  & \text { if } \lambda =\omega_4+(2^s-1) \omega_1\\
	 0 & \text { else. }
	 \end{array}\right.\end{equation*}
	 Now suppose $E^{1,0}\neq 0$. We have 
	 \begin{align*}E^{1,0}&=\opH^{1}(B_{r/2} / B_{\tau}, \Hom_{B_{\tau}}(k, \lambda)), \\
	 &=\opH^{1}(B_{r/2} / B_{\tau}, \Hom_{B_{\tau}}(k, \lambda_0)\otimes \tau(\lambda'))
	 \end{align*}
	 so $\lambda_0=0$ and $\lambda=\tau(\lambda')$. Thus $E^{1,0}\cong \opH^{1}( B_{s},  {\lambda'} ^{(\tau)}) \cong  \opH^{1}( B_{s},  {\lambda'})^{(\tau)}$
	 for $\lambda=\tau {\lambda'}$. Notice that since $r-1=2s>0$, $B_{(r-1)/2}=B_{s}$ is a classical Frobenius kernel and $\opH^{1}( B_{s},  {\lambda'})$ is the $B_s$-cohomology for ${\lambda'} \in X_{s}(T)$ computed in \cite[Theorem 2.7]{BNP04b}. We have
	 \begin{equation*}\opH^{1}(B_{s}, {\lambda'} ) \cong\left\{\begin{array}{ll}
	 k_{\omega_{j}}^{(s)}  & \text { if } {\lambda'}=2^{s} \omega_j-2^{s-1} \alpha_{j}, j \in \left\lbrace 1,3,4 \right\rbrace  \\
	 M^{(s)}_{F_4} & \text { if }  {\lambda'}=2^{s-1} \omega_1 \\
	 k_{\omega_{\alpha}}^{(s)}  & \text { if } {\lambda'}=2^{s}\omega_{\alpha} -2^{i}\alpha, \alpha \in \Pi, 0 \leq i \leq s-2\\
	 0 & \text { else. }
	 \end{array}\right.\end{equation*}
	 with $M_{F_4}$ having the structure as claimed in the statement of the theorem. We note the implicit constraints on $s$ in the different cases. Thus,
	 \begin{equation*}E^{1,0} \cong \opH^{1}(B_{s}, {\lambda'} )^{(\tau)} \cong\left\{\begin{array}{ll}
	 k_{\omega_{j}}^{(r/2)}  & \text { if } {\lambda'}=2^{s} \omega_j-2^{s-1} \alpha_{j}, j \in \left\lbrace 1,3,4 \right\rbrace  \\
	 M^{(r/2)}_{F_4} & \text { if }  {\lambda'}=2^{s-1} \omega_1 \\
	 k_{\omega_{\alpha}}^{(r/2)}  & \text { if } {\lambda'}=2^{s}\omega_{\alpha} -2^{i}\alpha, \alpha \in \Pi, 0 \leq i \leq s-2\\
	 0 & \text { else. }
	 \end{array}\right.\end{equation*}
	 
	 Lastly, one may recover $\lambda$ from ${\lambda'}$, recalling $\alpha_1=2\omega_{1}-\omega_{2}, \alpha_2=-\omega_{1}+2\omega_{2}-2\omega_{3}$, $\alpha_3=-\omega_{2}+2\omega_{3}-\omega_{4}$ and $\alpha_4=-\omega_{3}+2\omega_{4}$. 
	 
For example, when ${\lambda'} =2^{s}\omega_{1}-2^{s-1}\alpha_1$, then
$\lambda=\tau {\lambda'}=2^s \omega_{4}-2^{s-1}\tau(2\omega_{1}-\omega_{2})=2^s \omega_{3}$. The other cases follow similarly.

By the discussion above, notice that there is no $\lambda$ for which $E^{0,1}$ and $E^{1,0}$ are both non-zero. Thus, if $E^{0,1}=0$, then $E^1 \cong E^{1,0}$. Alternatively, if $E^{0,1} \neq 0$, then $\lambda$ must be one of the following: either $\lambda =\omega_3+2(2^s-1)\omega_4$, $\lambda =\omega_4+\omega_2+2(2^s-1)\omega_3$, $\lambda =\omega_4+\omega_1+(2^s-1)\omega_2$ or $\lambda =\omega_4+(2^s-1) \omega_1$. Clearly, in all of these cases, $\lambda \notin \tau X(T)$, thus forcing $\Hom_{B_{\tau}}(k, \lambda)=0.$ Hence $E^{1,0}=E^{2,0}=0$, implying that $E^1 \cong E^{0,1}$.
\end{proof}	
For a general $\lambda \in X(T)$, not necessarily lying in $X_{r/2}$, we proceed as in \cite[2.8]{BNP04b}. First, we make the following observation, whose proof is identical to the proof of Corollary \ref{C2weightform}:
\begin{cor} \label{F4weightform}
    Let $\lambda \in X(T)$. Then $\opH^1(B_{r/2}, \lambda) \neq 0$ if and only if $\lambda= \tau^r \omega - \tau^i \alpha$, for some weight $\omega \in X(T)$, and $\alpha \in \Pi$ with $0 \leq i \leq 2s-1$ or $\lambda= \tau^r \omega - \beta$, for some weight $\omega \in X(T)$, and $\beta \in \Pi_s$.
\end{cor}
Like in the previous cases, let $(\zeta, j)$ denote the appropriate pair, $(\alpha, i)$ or $(\beta, 1)$, defined in the previous corollary.
Given $\lambda=\tau^r \omega-\tau^j \zeta$, we may write $\lambda=\tau^r \omega' - \tau^j \zeta + \tau^r \lambda_1$. The non-vanishing of $\opH^1(B_{r/2},\lambda)$ is solely dependent on the choice of $\lambda_0$, so $\omega'$ is as given in Theorem \ref{Br2F4} for some $\lambda_1 \in X(T)$. Then, set $\lambda_1=\omega-\omega'$ and we get $\opH^1(B_{r/2}, \lambda)\cong \opH^1(B_{r/2}, \tau^r \omega' - \tau^j \zeta) \otimes k_{\omega-\omega'}^{(r/2)}.$

Substituting the results from Theorem \ref{Br2F4} leads to the the following result
\begin{theorem} \label{Br2F4gen}
Let $\lambda \in X(T).$ and $0 \leq i \leq s-2$. Then
	\begin{equation*}\opH^{1}\left(B_{r/2}, \lambda \right) \cong\left\{\begin{array}{ll}
	k_{\omega}^{(r/2)}  & \text { if } \lambda =\tau^r \omega - \beta, \omega \in X(T), \beta \in \Pi_s \\
	k_{\omega}^{(r/2)}  & \text { if } \lambda =\tau^r \omega - \tau^{2s-1}\alpha_j, \omega \in X(T),\\
	& \quad\alpha_j \in \Pi, j \in \lbrace 1,3,4 \rbrace \\
	M^{(r/2)}_{F_4} \otimes k_{\omega+\omega_3-\omega_2}^{(r/2)} & \text { if }   \lambda =\tau^r \omega - \tau^{2s-1}\alpha_2, \omega \in X(T)\\
	k_{\omega}^{(r/2)}  & \text { if } \lambda =\tau^r \omega - \tau^{2i+1}\alpha_j, \omega \in X(T), \alpha_j \in \Pi\\
	0 & \text { else. }
	\end{array}\right.\end{equation*}
\end{theorem} 
\subsubsection*{$G_{r/2}$-cohomology of induced modules}

Using Kempf's vanishing theorem, Theorem \ref{btau-f4}, Theorem \ref{Br2F4} and (\ref{indBG}), we compute $\opH^{1}(G_{r/2}, \opH^{0}(\lambda))$ for $\lambda \in X_{r/2}.$ 
Finally, we note that, by \cite[3.1, Theorem (C)]{BNP04b}, $\Ind_B^G(M_{F_{4}})= \opH^0(\omega_4)$.

First, suppose $r=1$. 
\begin{theorem} \label{gtau-f4}
Let $\lambda \in X_{\tau}(T)$. Then
\begin{equation*}\opH^{1}(G_{\tau}, \opH^0(\lambda))^{(-\tau)} \cong\left\{\begin{array}{ll}
\opH^0(\omega_{4}) & \text { if } \lambda= \omega_4 \\
\opH^0(\omega_{1}) & \text { if } \lambda=\omega_3 \\
0 & \text { else. }
\end{array}\right.\end{equation*}
\end{theorem}

Now, let $r>1$.
\begin{theorem}\label{reegr2}
	Let $\lambda \in X_{r/2}(T)$ and $0 \leq i \leq s-2$. Then
	\begin{equation*}\opH^{1}(G_{r/2}, \opH^{0}(\lambda))^{(-r/2)} \cong\left\{\begin{array}{ll}
	\opH^0(\omega_1)  & \text { if } \lambda =\omega_3+2(2^s-1)\omega_4=\tau^r \omega_1-\beta_2  \\
	 \opH^0(\omega_2)  & \text { if } \lambda =\omega_4+\omega_2+2(2^s-1)\omega_3=\tau^r \omega_2-\beta_1\\
	 \opH^0(\omega_3)  & \text { if } \lambda =\omega_4+\omega_1+(2^s-1)\omega_2=\tau^r \omega_3-\beta_3\\
	 \opH^0(\omega_4)  & \text { if } \lambda =\omega_4+(2^s-1) \omega_1=\tau^r \omega_4-\beta_4\\
    \opH^{0}(\omega_1)  & \text { if } \lambda=2^{s} \omega_3=\tau^r \omega_1-\tau^{2s-1}\alpha_1 \\
	\opH^{0}(\omega_3)  & \text { if } \lambda=2^{s}\omega_3 +2^{s-1}\omega_{1}=\tau^r \omega_3-\tau^{2s-1}\alpha_3\\
	\opH^{0}(\omega_4)  & \text { if } \lambda= 2^{s-1}\omega_2=\tau^r \omega_4-\tau^{2s-1}\alpha_4\\
	\opH^{0}(\omega_4)  & \text { if }  \lambda=2^{s} \omega_4=\tau^r(\omega_2-\omega_3)-\tau^{2s-1}\alpha_2 \\
	\opH^{0}(\omega_1)  & \text { if } \lambda=(2^{s+1}-2^{i+2}) \omega_4 +2^{i+1}\omega_3=\tau^r \omega_1-\tau^{2i+1}\alpha_1 \\
	\opH^{0}(\omega_2)  & \text { if }  \lambda=2^{i+1}\omega_2 +(2^{s+1}-2^{i+2})\omega_3 + 2^{i+1}\omega_4\\
	&\quad\quad =\tau^r \omega_2-\tau^{2i+1}\alpha_2\\
	\opH^{0}(\omega_3)  & \text { if } \lambda=2^{i}\omega_1 +(2^{s}-2^{i+1}) \omega_2 + 2^{i+1}\omega_3 \\
	&\quad\quad =\tau^r \omega_3-\tau^{2i+1}\alpha_3\\
	\opH^{0}(\omega_4)  & \text { if } \lambda=(2^{s}-2^{i+1})\omega_1 +2^{i}\omega_2=\tau^r \omega_4-\tau^{2i+1}\alpha_4, \\
	0 & \text { else. }
	\end{array}\right.\end{equation*} 
\end{theorem}
One can use Theorem \ref{Br2F4gen} to determine $\opH^1(G_{r/2}, \opH^0(\lambda))$ in terms of induced modules for all dominant weights $\lambda$, by applying the induction functor $\Ind_B^G$. The remark below deals with the only non-obvious case. 
\begin{remark}
    Let $\tau^r \omega - \tau^{2s-1} \alpha_2 \in X(T)_+$. Then $\langle \omega, \alpha_1^{\vee} \rangle \geq 0$, $\langle \omega, \alpha_2^{\vee} \rangle \geq 1$, $\langle \omega, \alpha_3^{\vee} \rangle \geq -1$ and $\langle \omega, \alpha_4^{\vee} \rangle \geq 0$.
    
    By  \cite[Proposition 3.4 (B)(d)]{BNP04b}, $\Ind_{B}^G(M_{F_4} \otimes k_{\omega+\omega_3-\omega_2})$ has a filtration with the following factors, from top to bottom: $\opH^0(\omega+\omega_3+\omega_4-\omega_2)$, $\opH^0(\omega+2\omega_3-\omega_4-\omega_2)$ and  $\opH^0(\omega)$. Furthermore, observe that:
    \begin{itemize}
        \item[(i)] $\opH^0(\omega+\omega_3+\omega_4-\omega_2)$ is always present.
        \item[(ii)] $\opH^0(\omega+2\omega_3-\omega_4-\omega_2)$ appears as a factor if $\langle \omega, \alpha_4^{\vee} \rangle \geq 1$ and does not if $\langle \omega, \alpha_4^{\vee} \rangle = 0$.
        \item[(iii)] $\opH^0(\omega)$ is present if $\langle \omega, \alpha_3^{\vee} \rangle \geq 0$ and is not present if $\langle \omega, \alpha_3^{\vee} \rangle = -1$.
    \end{itemize}
\end{remark}
\subsubsection*{$G_{r/2}$-cohomology with coefficients in simple modules} In this subsection, we make use of the  $G_1$-cohomology with coefficients in simple modules, computed in \cite[Proposition 4.11]{Sin94b}, to calculate $\opH^1(G_{s}, L(\lambda))$, for a positive integer $s$ and $\lambda \in X_{s}(T)$.
\begin{theorem} \label{f4gs-simple}
	Let $s$ be a positive integer and $\lambda \in X_s(T)$, for $0 \leq i \leq s-2$. Then
	\begin{equation*}\opH^{1}(G_s, L(\lambda))^{(-s)} \cong\left\{\begin{array}{ll}
	L(\omega_4)  & \text { if } \lambda=2^{s-1} \omega_1 \\
	k \oplus L(\omega_1)  & \text { if } \lambda=2^{s-1} \omega_2 \\
	k  & \text { if } \lambda=2^{s-1}(\omega_1+\omega_4) \\
	k \oplus L(\omega_4)  & \text { if } \lambda=2^{s-1} (\omega_2+\omega_3) \\
	k  & \text { if } \lambda=2^{i}(\omega_1+2\omega_4) \\
	k  & \text { if } \lambda=2^{i}\omega_2 \\
	k  & \text { if } \lambda=2^{i}(\omega_2+2\omega_1) \\
	k  & \text { if } \lambda=2^{i}(\omega_1+\omega_4) \\
	k  & \text { if } \lambda=2^{i}(\omega_2+\omega_3) \\
    k  & \text { if } \lambda=2^{i}(\omega_2+\omega_3+2\omega_4)  \\
	0 & \text { else. }
	\end{array}\right.\end{equation*} 
	Note that it is implicit in the statement of the theorem that $s \geq 1$ or $s \geq 2$, depending on the case. 
\end{theorem}
\begin{proof}
We proceed inductively. When $s=1$, we refer the reader to \cite[Proposition 4.11]{Sin94b}.
We write $\lambda=\lambda_0+2^{s-1}\lambda_1$, for $\lambda_0 \in X_{s-1}$ and $\lambda_1 \in X_1.$ Suppose $s>1$ and consider the LHS spectral sequence corresponding to $G_{s-1} \lhd G_s$. The $E_2$-page is given by
\begin{displaymath}
E^{i,j}_2:=\opH^i(G_1, \opH^j(G_{s-1}, L(\lambda_0))^{(-s+1)} \otimes L(\lambda_1))^{(s-1)}.
\end{displaymath}
First, consider the $E^{1,0}$-term. We have
\begin{displaymath}
E^{1,0}_2=\opH^1(G_1, \Hom_{G_{s-1}}(k, L(\lambda_0))^{(-s+1)} \otimes L(\lambda_1))^{(s-1)}.
\end{displaymath}
Note that $E^{1,0} \neq 0$ if and only if $\lambda_0=0$, in which case we obtain 
\begin{equation*}E^{1,0}=\opH^{1}(G_1, L(\lambda_1))^{(s-1)} \cong\left\{\begin{array}{ll}
	L(\omega_4)^{(s)} & \text { if } \lambda_1=\omega_1 \\
	k \oplus L(\omega_1)^{(s)} & \text { if } \lambda_1=\omega_2 \\
	k  & \text { if } \lambda_1=\omega_1+\omega_4 \\
	k \oplus L(\omega_4)^{(s)} & \text { if } \lambda_1=\omega_2+\omega_3 \\
	0 & \text { else. }
	\end{array}\right.\end{equation*} 
(cf. \cite[Proposition 4.11]{Sin94b}). Therefore, recalling that $\lambda=\lambda_0+2^{s-1}\lambda_1$, we may conclude that for $\lambda \in X_s$,
\begin{equation*}E^{1,0}\cong\left\{\begin{array}{ll}
	L(\omega_4)^{(s)} & \text { if } \lambda=2^{s-1}\omega_1 \\
	k \oplus L(\omega_1)^{(s)} & \text { if } \lambda=2^{s-1}\omega_2 \\
	k  & \text { if } \lambda=2^{s-1}(\omega_1+\omega_4) \\
	k \oplus L(\omega_4)^{(s)} & \text { if } \lambda=2^{s-1}(\omega_2+\omega_3) \\
	0 & \text { else. }
	\end{array}\right.\end{equation*}
Now consider the $E^{0,1}$-term. We have
\begin{displaymath}
E^{0,1}=\Hom_{G_1}(L(\lambda_1), \opH^1(G_{s-1}, L(\lambda_0))^{(-s+1)})^{(s-1)}.
\end{displaymath}
We take each non-zero instance of $\opH^1(G_{s-1}, L(\lambda_0))$ in turn. For instance, by the induction hypothesis, if $\lambda_0=2^{s-2}\omega_2$, then $E^{0,1}=\Hom_{G_1}( L(\lambda_1), k \oplus L(\omega_1))^{(s-1)}$. Thus $E^{0,1} \neq 0$ if $\lambda_1=\omega_0$ or $\lambda_1=\omega_4$; hence $E^{0,1} \cong k$ for $\lambda=2^{s-2}(\omega_2+\omega_3)$ or $\lambda=2^{s-2}(\omega_2+\omega_3+2\omega_4)$, respectively. The other cases follow similarly and we obtain, for $0 \leq i \leq s-2$
\begin{equation*}E^{0,1}\cong\left\{\begin{array}{ll}
	k & \text { if } \lambda=2^{i}(\omega_1+2\omega_4) \\
	k & \text { if } \lambda=2^{i}\omega_2 \\
	k & \text { if } \lambda=2^{i}(\omega_1+2\omega_1) \\
	k & \text { if } \lambda=2^{i}(\omega_1+\omega_4) \\
	k & \text { if } \lambda=2^{i}(\omega_2+\omega_3) \\
	k  & \text { if } \lambda=2^{i}(\omega_2+\omega_3+2\omega_4)\\
	0 & \text { else. }
	\end{array}\right.\end{equation*}
Based on the discussion above, notice that there is no choice of $\lambda$ for which $E^{1,0}$ and $E^{0,1}$ are both non-zero. Hence if $E^{1,0} \neq 0$, then $E^{0,1}=0$, implying that $E^{1} \cong E^{1,0}$. Alternatively, suppose that $E^{0,1} \neq 0$. Then observe that for all of the choices of $\lambda$ above, $\lambda_0 \neq 0$, pushing $\Hom_{G_{s-1}}(k, L(\lambda_0))=0$. Therefore $E^{0,1} \neq 0$ implies $E^{1,0}=E^{2,0}=0$, meaning that $E^1 \cong E^{0,1}$. 
\end{proof}
Next, with the aid of the previous theorem concerning the cohomology for classical Frobenius kernels, we compute $\opH^{1}(G_{r/2}, L(\lambda))$ for $r$ an odd positive integer and $\lambda \in X_{r/2}$.

If $r=1$, we refer the reader to \cite[Lemma 4.5(a) and 4.6]{Sin94b}. Otherwise, we get:
\begin{theorem} \label{f4gr2-simple}
	Suppose $r=2s+1>1$ and let $\lambda \in X_{r/2}(T)$, $0 \leq i \leq s-1$ and $0 \leq j \leq s-2$. Then
    \begin{equation*}\opH^{1}(G_{r/2}, L(\lambda))^{(-r/2)} \cong\left\{\begin{array}{ll}
	L(\omega_4)  & \text { if } \lambda=2^{s} \omega_4 \\
	k \oplus L(\omega_1)  & \text { if } \lambda=2^{s} \omega_3 \\
	k  & \text { if } \lambda=2^{i}(\omega_1+2\omega_4) \\
	k  & \text { if } \lambda=2^{i} \omega_2\\
	k  & \text { if } \lambda=2^{i}(\omega_1+\omega_4) \\
	k  & \text { if } \lambda=2^{i}(\omega_2+\omega_3) \\
    k  & \text { if } \lambda=2^{i}(\omega_2+\omega_3+2\omega_4)  \\
    k  & \text { if } \lambda=2^{j}(\omega_2+2\omega_1) \\
	0 & \text { else. }
	\end{array}\right.\end{equation*}
\end{theorem}
\begin{proof}
For $\lambda \in X_{r/2}$, write $\lambda=\lambda_0+2^s\lambda_1$, for $\lambda_0 \in X_s$ and $\lambda_1 \in X_{\tau}$. Consider the LHS spectral sequence corresponding to $G_{s} \lhd G_{r/2}$. The $E_2$-page is given by
\begin{displaymath}
E^{i,j}_2:=\opH^i(G_{\tau}, \opH^j(G_{s}, L(\lambda_0))^{(-s)} \otimes L(\lambda_1))^{(s)}.
\end{displaymath}
First, consider the $E^{1,0}$-term. We have
\begin{displaymath}
E^{1,0}=\opH^1(G_{\tau}, \Hom_{G_{s}}(k, L(\lambda_0))^{(-s)} \otimes L(\lambda_1))^{(s)}.
\end{displaymath}
Note that $E^{1,0} \neq 0$ if and only if $\lambda_0=0$, in which case we obtain 
\begin{equation*}E^{1,0}=\opH^{1}(G_{\tau}, L(\lambda_1))^{(s)} \cong\left\{\begin{array}{ll}
	L(\omega_4)^{(r/2)} & \text { if } \lambda_1=\omega_4\\
	k \oplus L(\omega_1)^{(r/2)} & \text { if } \lambda_1=\omega_3\\
	0  & \text { else. } 
	\end{array}\right.\end{equation*} 
(cf. Theorem \ref{gtau-f4} and \cite[Lemma 4.5(a) and 4.6]{Sin94b}). Next, consider the $E^{0,1}$-term:
\begin{displaymath}
E^{0,1}=\Hom_{G_{\tau}}(L(\lambda_1), \opH^1(G_{s}, L(\lambda_0))^{(-s)})^{(s)}.
\end{displaymath}
We take each non-zero instance of $\opH^1(G_{s}, L(\lambda_0))^{(-s)}$ from Theorem \ref{g2gs-simple} in turn. For example, if $\lambda_0=2^{s-1}\omega_2$, then 
$E^{0,1}=\Hom_{G_{\tau}}(L(\lambda_1), k \oplus    L(\omega_1))^{(s)}$. Thus $E^{0,1} \neq 0$ if and only if $\lambda_1=0$, since $\omega_1 \notin X_{\tau}$. We obtain $E^{0,1} \cong k$ for $\lambda=2^{s-1}\omega_2$. The other cases are similar. We get, for $0 \leq i \leq s-1$ and $0 \leq j \leq s-2$,
\begin{equation*}E^{0,1}\cong\left\{\begin{array}{ll}
    k  & \text { if } \lambda=2^{i}(\omega_1+2\omega_4) \\
	k  & \text { if } \lambda=2^{i} \omega_2\\
	k  & \text { if } \lambda=2^{i}(\omega_1+\omega_4) \\
	k  & \text { if } \lambda=2^{i}(\omega_2+\omega_3) \\
    k  & \text { if } \lambda=2^{i}(\omega_2+\omega_3+2\omega_4)  \\
    k  & \text { if } \lambda=2^{j}(\omega_2+2\omega_1) \\
	0 & \text { else. }
	\end{array}\right.\end{equation*}
In light of the discussion above, notice that there is no choice of $\lambda$ for which $E^{1,0}$ and $E^{0,1}$ are both non-zero. Hence if $E^{1,0} \neq 0$, then $E^{0,1}=0$, implying that $E^{1} \cong E^{1,0}$. Alternatively, suppose that $E^{0,1} \neq 0$ and notice that for each choice of $\lambda$ above, we obtain $\Hom_{G_s}(k,L(\lambda_0))=0$. Hence $E^{1,0}=E^{2,0}=0$, meaning that $E^1 \cong E^{0,1}$. 
\end{proof}

\section{Bounding cohomology for the Ree groups of type $F_4$} \label{Sect4}

In this section we turn our attention to the extensions between simple modules for the Ree groups of type $F_4$, for which we aim to prove results using the \cite{BNP06} approach. 

To begin with, we briefly discuss the motivation behind the \cite{BNP06} framework. Their method relies on the use of a certain truncated category of $G$-modules. In such a category, the weights of the $G$-modules have a suitable upper bound, and it is highest weight category (see \cite[Definition 3.1]{CPS3} for a definition). Moreover, the key fact is that this truncated category contains enough projective modules (we refer the reader to \cite[4.2 and 4.5]{BNP01}, or \cite[\S 1]{Don86}, for a more general treatment of truncated categories). This enables us to use various LHS spectral sequences in order to link it to related module categories, for the Frobenius kernels and finite group. 

In Subsection \ref{subsec41}, we provide precise definitions and results on which we base our construction.

\subsection{Filtering $\Ind_{G(\sigma)}^G k$} \label{subsec41}
We begin by fixing some notation and introducing some terminology. For the trivial module $k$, set $\mathcal{G}(k):= \Ind_{G(\sigma)}^{G} k$; it is an infinite-dimensional module since the coset space $G/G(\sigma)$ is affine. Then for any finite set of dominant weights $\pi \subseteq X(T)_{+}$, we define $\mathcal{G}_{\pi}(k)$ to be the maximal $G$-submodule of $\mathcal{G}(k)$ having composition factors with weights in $\pi$.

Now, observe the following result from \cite{BNP+12} concerning the structure of $\mathcal{G}(k)$.
\begin{theorem}{(\cite[Prop 3.1.2]{BNP+12})}\label{sections}
	The $G$-module $\mathcal{G}(k)$ has a filtration with factors of the form $\opH^0(\nu) \otimes \opH^0(\nu^{*})^{(\sigma)}$, one for each
	$\nu \in X(T)_{+}$ and occurring in an order compatible with the dominance order on $X_{+}$.
\end{theorem}
Since $G/G(\sigma)$ is affine, the induction functor is exact (cf. \cite[I.5.13]{Jan03}). Then, by generalised Frobenius reciprocity (cf. \cite[I.4.6]{Jan03}), there exists an isomorphism for each $n \geq 0$ and any two $G$-modules $V$, $W$:
\begin{equation} \label{functor}
\Ext^n_{G(\sigma)}(V,W) \cong \Ext^n_G(V,W \otimes \mathcal{G}(k)).
\end{equation}
In view of Theorem \ref{sections}, in order to apply (\ref{functor}) and study $ \Ext_{G(\sigma)}^{1}(L(\lambda), L(\mu))$ for $\lambda, \mu \in X_{\sigma}$, we must investigate the $\Ext$-groups 
\begin{displaymath}
\Ext^{1}_G(L(\lambda), L(\mu) \otimes \opH^{0}(\nu) \otimes \opH^0(\nu^{\ast})^{(r/2)}) \cong \Ext^{1}_G(L(\lambda) \otimes V(\nu)^{(r/2)}, L(\mu) \otimes \opH^{0}(\nu)),
\end{displaymath}
for all $\nu \neq 0$, $\nu \in X_{+}$.
First, we provide a way to identify homomorphisms over $G_{r/2}$ with homomorphisms over $G$, under a certain condition. This holds for the Suzuki groups and the Ree groups.
\begin{lemma} \label{l33}
    Let $r \in \mathbb{N}$ and set $s= \left \lfloor{r/2}\right \rfloor $.  Let $\lambda, \mu \in X_{r/2}$ and $\nu \in X_{+}$. We have:
    \begin{itemize}
        \item[(a)] If $\left\langle \nu, \alpha_{0}^{\vee} \right\rangle < p^s$, then the $G$-module $\Hom_{G_{r/2}}(L(\lambda), L(\mu) \otimes \opH^0(\nu))$ has trivial $G$-structure, meaning that it is isomorphic to $\Hom_G(L(\lambda), L(\mu) \otimes \opH^0(\nu))$. 
        \item[(b)]  If $\tau^r\theta$ is a weight of $\Hom_{G_{r/2}}(L(\lambda), L(\mu) \otimes \opH^0(\nu))$, then $\left\langle \tau^r\theta, \alpha_{0}^{\vee} \right\rangle \leq \left\langle \nu, \alpha_{0}^{\vee} \right\rangle$.
    \end{itemize}
\end{lemma}
\begin{proof}
(a) This is  \cite[Proposition 3.1]{BNP06} when $r$ is even. When $r$ is odd, we use the same argument. Without loss of generality, we may assume  $\left\langle \mu, \alpha_{0}^{\vee} \right\rangle \leq \left\langle \lambda, \alpha_{0}^{\vee} \right\rangle$. Since all $G$-composition factors of $\Hom_{G_{r/2}}(L(\lambda), L(\mu) \otimes \opH^0(\nu))$ are $G_{r/2}$-trivial, they must be of the form $L(\theta)^{(r/2)}$, for some $\theta \in X(T)$. Let $L(\theta)^{(r/2)}$ be such a factor and then a weight of $L(\mu) \otimes \opH^0(\nu)$ will be $\lambda + \tau^r \theta$; we obtain
\begin{displaymath}
\left\langle \lambda + \tau^r \theta, \alpha_{0}^{\vee} \right\rangle \leq \left\langle \mu + \nu, \alpha_{0}^{\vee} \right\rangle \leq \left\langle \lambda + \nu, \alpha_{0}^{\vee} \right\rangle
\end{displaymath}
(with the last inequality following from the assumption). Thus
\begin{displaymath}
p^s\left\langle \theta, \alpha_{0}^{\vee} \right\rangle \leq p^s\left\langle \tau \theta, \alpha_{0}^{\vee} \right\rangle \leq \left\langle  \nu, \alpha_{0}^{\vee} \right\rangle < p^s,
\end{displaymath}
(with the last inequality following from the hypothesis), pushing $\theta=0$, and thus proving the claim.

Part (b) follows immediately from the proof of part (a).
\end{proof} 
From this point onwards, unless stated otherwise, we let $G$ be of type $F_4$ and $p=2.$ 
Next we prove a result in flavour of \cite[Lemma 5.2]{BNP06}.
\begin{lemma}\label{bnp52}
	Let $\lambda,\mu \in X_{r/2}(T)$ and $\nu \in X(T)_{+}$. Assume further that $2^{s}>4$. If $\Ext^{1}_G(L(\lambda) \otimes V(\nu)^{(r/2)}, L(\mu) \otimes \opH^{0}(\nu)) \neq 0$, then $\left\langle \nu, \alpha_{0}^{\vee} \right\rangle < 17 = h+5$. Furthermore, except for possibly one dominant weight, namely $\nu = 8\omega_{4}$, the non-vanishing implies $\left\langle  \nu,  \alpha_{0}^{\vee}\right\rangle <16.$ 
\end{lemma}

\begin{proof}
	Consider the LHS spectral sequence 
	\begin{displaymath}
	\begin{aligned}
	E_{2}^{i, j} &=\Ext_{G / G_{r/2}}^{i}(V(\nu)^{(r/2)}, \Ext_{G_{r/2}}^{j}(L(\lambda), L(\mu) \otimes \opH^{0}(\nu))) \\
	& \Rightarrow \Ext_{G}^{i+j}(L(\lambda) \otimes V(\nu)^{(r/2)}, L(\mu) \otimes \opH^{0}(\nu)).
	\end{aligned}
	\end{displaymath}
	Consider the $E_{2}^{i,0}$-term:
	\[
	E_{2}^{i,0}=\Ext_{G / G_{r/2}}^{i}(V(\nu)^{(r/2)},  \Hom_{G_{r/2}}(L(\lambda), L(\mu) \otimes \opH^{0}(\nu)).
	\]
	It follows from Lemma \ref{l33} (b) that any weight $\theta$ of $\Hom_{G_{r/2}}(L(\lambda), L(\mu) \otimes \opH^{0}(\nu))^{(-r/2)}$ satisfies $\left\langle \theta, \alpha_{0}^{\vee} \right\rangle \leq \frac{1}{p^s}\left\langle \nu, \alpha_{0}^{\vee} \right\rangle < \left\langle \nu, \alpha_{0}^{\vee} \right\rangle$. Since $V(\nu)$ is projective in the category of modules with weights $\beta$ so that $\left\langle \beta, \alpha_{0}^{\vee} \right\rangle < \left\langle \nu, \alpha_{0}^{\vee} \right\rangle$, we may conclude that the $E_{2}^{i,0}$ terms vanish. Therefore
	\[
	E_{2} \cong E_{2}^{0,1} \cong \Hom_{G / G_{r/2}}(V(\nu)^{(r/2)}, \Ext^1_{G_{r/2}}(L(\lambda), L(\mu) \otimes \opH^{0}(\nu))).
	\]
	Let $\tau^{r} \gamma$ be a weight of a composition factor of $\Ext_{G_{r/2}}^{1}(L(\lambda), L(\mu) \otimes \opH^{0}(\nu))$. We claim that 
	\begin{equation}\label{ineq1}
	    \left\langle \tau^{r}\gamma , \alpha_{0}^{\vee}\right\rangle  \leq \left\langle \lambda + \mu + \nu, \alpha_{0}^{\vee}\right\rangle + 2^s.
	\end{equation}
	In order to show this, first consider $\opH^1(G_{r/2}, L(\lambda) \otimes L(\mu) \otimes \opH^0(\nu))$. Let $L(\sigma_0) \otimes L(\sigma_1)^{(r/2)}$ be a composition factor of $L(\lambda) \otimes L(\mu) \otimes \opH^0(\nu)$, for some $\sigma_0 \in X_{r/2}$ and $\sigma_1 \in X_{+}$. Hence, in order to bound the weights of $\opH^1(G_{r/2}, L(\lambda) \otimes L(\mu) \otimes \opH^0(\nu))$, we must evaluate the weights of $\opH^1(G_{r/2}, L(\sigma_0)) \otimes L(\sigma_1)^{(r/2)}.$ 
	
	Observe that $\opH^1(G_{r/2}, L(\sigma_0))^{(-r/2)}$ for $\sigma_0 \in X_{r/2}$ was computed in Theorem \ref{f4gr2-simple}.  Let $\tau^r \theta$ denote a weight of $\opH^1(G_{r/2}, L(\sigma_0))$. We claim that it must satisfy
	\begin{equation} \label{sigma0firstbit}
	\left\langle \tau^{r}\theta, \alpha_{0}^{\vee} \right\rangle \leq \left\langle \sigma_0, \alpha_{0}^{\vee} \right\rangle +2^s.
	\end{equation}
	We consider each non-zero instance in the theorem in turn. We present the explicit computation of the case $\sigma_0 = 2^s \omega_3$, for which $\opH^1(G_{r/2}, L(\sigma_0))^{(-r/2)} \cong k \oplus L(\omega_1)$. Since $0\leq \omega_1$, we may assume $\theta=\omega_1$. We obtain
	\begin{displaymath}
	\left\langle \tau^{r}\omega_1, \alpha_{0}^{\vee} \right\rangle = 2^s \cdot 4 \leq 2^s \cdot 3 +2^s.
	\end{displaymath}
	Similar calculations for all of the other choices of $(\sigma_0, \theta)$ lead us to conclude that the inequality (\ref{sigma0firstbit}) holds and this proves the claim.

Thus, if $\tau^r \gamma$ is a weight of $\opH^1(G_{r/2}, L(\lambda) \otimes L(\mu) \otimes \opH^0(\nu))$, we have $\left\langle \tau^{r}\gamma, \alpha_{0}^{\vee} \right\rangle \leq \left\langle \tau^r\theta, \alpha_{0}^{\vee} \right\rangle +\left\langle \tau^r\sigma_1, \alpha_{0}^{\vee} \right\rangle$, for $L(\sigma_0) \otimes L(\sigma_1)^{(r/2)}$ a composition factor of $L(\lambda) \otimes L(\mu) \otimes \opH^0(\nu)$ and $\theta$ a weight of $\opH^1(G_{r/2}, L(\sigma_0))^{(-r/2)}$. Using (\ref{sigma0firstbit}), we obtain
	\[
	\begin{aligned}
	\left\langle \tau^{r}\gamma, \alpha_{0}^{\vee}\right\rangle &\leq \left\langle \tau^{r}\theta, \alpha_{0}^{\vee}\right\rangle+\left\langle \tau^{r}\sigma_{1}, \alpha_{0}^{\vee}\right\rangle \leq\left\langle\sigma_{0}, \alpha_{0}^{\vee}\right\rangle+2^{s}+\left\langle \tau^{r}\sigma_{1}, \alpha_{0}^{\vee}\right\rangle \\
	& \leq \left\langle \lambda+\mu+\nu, \alpha_{0}^{\vee}\right\rangle+2^{s}.
	\end{aligned}
	\]
	This verifies (\ref{ineq1}). 
	
	Consider the short exact sequence 
	\begin{displaymath}
	0 \rightarrow L(\mu) \rightarrow \mathrm{St}_{r/2} \otimes L\left((2^s-1)(\omega_{1}+\omega_{2})+(2^{s+1}-1)(\omega_{3}+\omega_{4})+w_{0} \mu\right) \rightarrow R \rightarrow 0.
	\end{displaymath}
	
	Using the long exact sequence of cohomology, along with the fact that $\mathrm{St}_{r/2}$ is injective as a $G_{r/2}$-module, one obtains a surjection
	\[
	\Hom_{G_{r/2}}(L(\lambda), R \otimes \opH^{0}(\nu)) \twoheadrightarrow \Ext_{G_{r/2}}^{1}(L(\lambda), L(\mu) \otimes \opH^{0}(\nu)).
	\]
	Hence, any weight $\tau^{r} \gamma$ of $\Ext_{G_{r/2}}^{1}(L(\lambda), L(\mu) \otimes \opH^{0}(\nu))$ also satisfies
	\begin{equation} \label{ineq2}
	\begin{aligned}
	\left\langle \tau^{r}\gamma, \alpha_{0}^{\vee}\right\rangle \leq 2(2^{s}-1)\left\langle \tau(\omega_{3}+\omega_{4}), \alpha_{0}^{\vee}\right\rangle &+2(2^{s+1}-1)\left\langle \omega_{3}+\omega_{4}, \alpha_{0}^{\vee}\right\rangle  \\
	&-\left\langle\lambda, \alpha_{0}^{\vee}\right\rangle-\left\langle\mu, \alpha_{0}^{\vee}\right\rangle+\left\langle \nu, \alpha_{0}^{\vee}\right\rangle.
	\end{aligned}
	\end{equation}
	
	Adding (\ref{ineq1}) and (\ref{ineq2}) and dividing by two yields
	\begin{equation} \label{weight}
	\begin{aligned}
	\left\langle \tau^{r}\gamma, \alpha_{0}^{\vee}\right\rangle & \leq (2^{s}-1)\left\langle \tau(\omega_{3}+\omega_{4}), \alpha_{0}^{\vee}\right\rangle+(2^{s+1}-1)\left\langle \omega_{3}+\omega_{4}, \alpha_{0}^{\vee}\right\rangle+\left\langle \nu, \alpha_{0}^{\vee}\right\rangle +2^{s-1}.\\
	\left\langle \tau^{r}\gamma, \alpha_{0}^{\vee}\right\rangle & \leq (2^{s}-1)\cdot 6 +(2^{s+1}-1) \cdot 5 + 2^{s-1} +\left\langle \nu, \alpha_{0}^{\vee}\right\rangle.
	\end{aligned}
	\end{equation}
	
	Since we assume $E^{0,1}_2 \neq 0$, we may assume $\tau^r \nu$ is a weight of $\Ext_{G_{r/2}}^{1}(L(\lambda), L(\mu) \otimes \opH^{0}(\nu))$. Therefore, put $\gamma=\nu$ to get
	\begin{equation}
	\left\langle \tau^{r}\nu, \alpha_{0}^{\vee}\right\rangle  \leq (2^{s}-1)\cdot 6 +(2^{s+1}-1) \cdot 5 + 2^{s-1} +\left\langle \nu, \alpha_{0}^{\vee}\right\rangle.
	\end{equation}
	Then, we have 
	\begin{equation} \label{nu-first}
	\langle \tau^{r}\nu, \alpha_{0}^{\vee}\rangle - \langle \nu, \alpha_{0}^{\vee}\rangle  \leq (2^{s}-1)\cdot 6 +(2^{s+1}-1) \cdot 5 + 2^{s-1}.
	\end{equation}
	Therefore, to finish the proof, we must investigate the link between $\left\langle \nu, \alpha_{0}^{\vee}\right\rangle$ and $\left\langle \tau^{r} \nu, \alpha_{0}^{\vee}\right\rangle.$  
	
	Since $\nu \in X(T)_{+}$, we may write $\nu  = a \omega_1+ b \omega_2 + c \omega_3 + d \omega_4$, for some non-negative integers $a,b,c,d$. 	
	Then, $\tau \nu = 2a \omega_4 + 2b \omega_3 + c \omega_2 + d \omega_1.$
	
	Furthermore, recalling $\left\langle  \omega_4,  \alpha_{0}^{\vee}\right\rangle = 2,  \left\langle  \omega_3,  \alpha_{0}^{\vee}\right\rangle = 3, \left\langle  \omega_2,  \alpha_{0}^{\vee}\right\rangle = 4, \left\langle  \omega_1,  \alpha_{0}^{\vee}\right\rangle = 2$, we have $\left\langle  \tau \nu,  \alpha_{0}^{\vee}\right\rangle = \left\langle  \nu,  \alpha_{0}^{\vee}\right\rangle +2a+2b+c.$
	Since $\left\langle  \tau\nu,  \alpha_{0}^{\vee}\right\rangle \geq \left\langle  \nu,  \alpha_{0}^{\vee}\right\rangle$, inequality (\ref{nu-first}) yields
   \begin{displaymath}
    \langle \tau^{r}\nu, \alpha_{0}^{\vee}\rangle - \langle \tau \nu, \alpha_{0}^{\vee}\rangle \leq \langle \tau^{r}\nu, \alpha_{0}^{\vee}\rangle - \langle \nu, \alpha_{0}^{\vee}\rangle  \leq (2^{s}-1)\cdot 6 +(2^{s+1}-1) \cdot 5 + 2^{s-1},
   \end{displaymath}
   giving
   \begin{displaymath}
   \langle \tau \nu, \alpha_{0}^{\vee}\rangle  \leq  6 +\frac{2^{s+1}-1}{2^s-1}\cdot 5 + \frac{2^{s-1}}{2^s-1}.
   \end{displaymath}
   Notice that, if $s \geq 3$, $ \langle \tau \nu, \alpha_{0}^{\vee}\rangle <18$ and if $s \geq 4$, $ \langle \tau \nu, \alpha_{0}^{\vee}\rangle <17$. Recall that $\left\langle  \tau \nu,  \alpha_{0}^{\vee}\right\rangle = \left\langle  \nu,  \alpha_{0}^{\vee}\right\rangle +2a+2b+c$, with $a,b,c \geq 0$. First, they are equal only when $a=b=c=0$ and thus we get $\left\langle   \nu,  \alpha_{0}^{\vee}\right\rangle=2d<18$. Since $d$ is a non-negative integer, we must have $d \leq 8$, in which case $\left\langle \nu,  \alpha_{0}^{\vee}\right\rangle \leq 16$ (with equality only for $d=8$ and $\nu=8\omega_4$). 

   It remains to investigate the case $\left\langle  \tau \nu,  \alpha_{0}^{\vee}\right\rangle \neq \left\langle \nu,  \alpha_{0}^{\vee}\right\rangle$, for which $2a+2b+c>0$. It is readily verifiable that $2a+2b+c \geq 2$ implies $\left\langle  \nu,  \alpha_{0}^{\vee}\right\rangle <16.$ Otherwise, $2a+2b+c=1$ and it immediately follows that $c=1$ and $a=b=0$. Therefore $ \nu=\omega_3+d\omega_{4}$, with $\left\langle  \nu,  \alpha_{0}^{\vee}\right\rangle =3+2d<17$. This inequality forces $d \leq 6$, in which case, $\left\langle  \nu,  \alpha_{0}^{\vee}\right\rangle \leq 15 <16$, as claimed.
	\end{proof}

\begin{remark} \label{different}
\begin{itemize}
    \item[(a)] One may show, using an argument very similar to Lemma \ref{bnp52}, that for $\lambda, \mu \in X_{\sigma}$,  $\Ext^1_{G_{\sigma}}(L(\lambda), L(\mu))^{(-\sigma)}$ is a rational $G$-module whose composition factors have high weights $\nu$ which satisfy  $\langle \nu, \alpha_0^{\vee} \rangle \leq h+4$. 
    
    \item[(b)] Let $\sigma: G \to G$ denote the appropriate strict endomorphism so that $G(\sigma)$ is a finite group of Lie type and $G_{\sigma}$ the associated scheme-theoretic kernel.
    
    By \cite[Theorem 2.3.1]{BNP+12}, for all $(G, p, \sigma)$ aside from the case where $G=F_4$, $p=2$ and $\sigma$ is an exceptional isogeny, there exists the following result concerning $G_{\sigma}$-extensions:  $\Ext^1_{G_{\sigma}}(L(\lambda), L(\mu))^{(-\sigma)}$ for $\lambda, \mu \in X_{\sigma}$ is a rational $G$-module whose composition factors have high weights $\nu$ which are $(h-1)$-small. Part $(a)$ fills a gap in their result.
\end{itemize}
\end{remark}

 By Lemma \ref{bnp52}, we know that $\Ext^{1}_G(L(\lambda) \otimes V(\nu)^{(r/2)}, L(\mu) \otimes \opH^{0}(\nu)) \neq 0$ implies $\langle \nu, \alpha_{0}^{\vee}\rangle < 17$. Thus, let us define $\Gamma \subseteq X_{+}$ to be the following set of dominant weights:
\begin{displaymath}
\Gamma=\{ \nu \in X(T)_+ : \langle \nu, \alpha_{0}^{\vee}\rangle < 17 \},
\end{displaymath}
and let $\mathcal{G}_{\Gamma}(k)$ be the finite-dimensional truncated submodule of $\mathcal{G}(k)$ with composition factors with highest weights in $\Gamma$.

We obtain for $\lambda, \mu \in X_{\sigma}$,
\begin{equation} \label{Gamma}
\Ext_{G(\sigma)}^{1}(L(\lambda), L(\mu)) \cong  \Ext_{G}^{1}(L(\lambda), L(\mu) \otimes \mathcal{G}_{\Gamma}(k)).
\end{equation}

\subsection{Finite group extensions} \label{subsec42}
Next, we make use of $(\ref{Gamma})$ and Theorem \ref{sections} to deduce some information concerning $\Ext_{G(\sigma)}^{1}(L(\lambda), L(\mu))$, under some conditions on the size of the finite group, $\prescript{2}{}F_4(2^{2s+1})$ -- the conditions will therefore be imposed on the value of $s$ and hence $r=2s+1$. 

First, by \cite[(5.3.1)]{BNP06}, we have for $W$ a $G$-module with a filtration $0=W_0 \subset W_1 \subset W_2 \subset \ldots \subset W_l=W$, for all $G$-modules $V$,
\begin{equation} \label{dimensions}
    \mathrm{dim} \Ext^1_G(V,W) \leq \sum_{n=1}^l \mathrm{dim} \Ext^1_G(V, W_n/W_{n-1}).
\end{equation}
\begin{prop} \label{identif}
	 Let $s \geq 7,$ such that $r \geq 15.$ Let $\lambda, \mu \in X_{r/2}$ and $\Gamma^{\prime}=\Gamma-\{0\}$. Then, the following hold: 

	(a) We have
	\begin{displaymath}
	\mathrm{dim} \Ext_{G(\sigma)}^{1}(L(\lambda), L(\mu))  \leq \mathrm{dim} \Ext_{G}^{1}(L(\lambda), L(\mu)) + \mathrm{dim} R,
	\end{displaymath}
	where 
	\begin{displaymath}
	\begin{aligned}
	R & \cong \bigoplus_{\nu \in \Gamma^{\prime}} \Ext_{G}^{1}(L(\lambda) \otimes V(\nu)^{(r/2)}, L(\mu) \otimes \opH^{0}(\nu)) \\
	& \cong \bigoplus_{\nu \in \Gamma^{\prime}} \Hom_{G / G_{r/2}}(V(\nu)^{(r/2)}, \Ext_{G_{r/2}}^{1}(L(\lambda), L(\mu) \otimes \opH^{0}(\nu))).
	\end{aligned}
	\end{displaymath}
	(b) Let $4 \leq t \leq s-3$. Set $\lambda=\lambda_{0}+2^{t} \lambda_{1}$ and $\mu=\mu_{0}+2^{t} \mu_{1}$ with $\lambda_{0}, \mu_{0} \in X_{t}$ and $\lambda_{1}, \mu_{1} \in X_{r/2-t}$.
	Then we may reidentify $R$ as
	\[
	\begin{aligned}
	R &\cong \bigoplus_{\nu \in \Gamma^{\prime}} \Ext_{G}^{1}(L\left(\lambda_{1}\right) \otimes V(\nu)^{(r/2-t)}, L\left(\mu_{1}\right)) \otimes \Hom_{G}(L(\lambda_{0}), L(\mu_{0}) \otimes \opH^{0}(\nu)) \\
	& \cong \bigoplus_{\nu \in \Gamma^{\prime}} \Hom_{G}(V(\nu)^{(r/2-t)}, \Ext_{G_{r/2-t}}^{1}(L(\lambda_{1}), L(\mu_{1}))) \otimes \Hom_{G}(L(\lambda_{0}), L(\mu_{0}) \otimes \opH^{0}(\nu)).
	\end{aligned}
	\]
\end{prop}
\begin{proof}
(a)  Note that by the previous discussion and Theorem \ref{sections}, $\mathcal{G}_{\Gamma}(k)$ has a filtration with factors of the form $\opH^0(\nu) \otimes \opH^0(\nu^{\ast})^{(r/2)}$, exactly one for each $\nu \in \Gamma$. 

Now, by (\ref{Gamma}) and (\ref{dimensions}), we obtain
\begin{displaymath}
\begin{aligned}
\mathrm{dim} \Ext_{G(\sigma)}^{1}(L(\lambda), L(\mu)) & = \mathrm{dim} \Ext_{G}^{1}(L(\lambda), L(\mu) \otimes \mathcal{G}_{\Gamma}(k)) \\
& \leq \sum_{\nu \in \Gamma} \mathrm{dim} \Ext_{G}^{1}(L(\lambda) \otimes V(\nu)^{(r/2)}, L(\mu) \otimes \opH^0(\nu)) \\
& = \mathrm{dim} \Ext_{G}^{1}(L(\lambda), L(\mu)) + \\
& \qquad \sum_{\nu \in \Gamma^{\prime}} \mathrm{dim} \Ext_{G}^{1}(L(\lambda) \otimes V(\nu)^{(r/2)}, L(\mu) \otimes \opH^0(\nu)).
\end{aligned}
\end{displaymath}
The first isomorphism is an immediate consequence of (\ref{Gamma}) and the properties of $\mathcal{G}_{\Gamma^{\prime}}(k)$. 
For the other isomorphism, note that since $2^s \geq 2^5 >17$, we may apply Lemma \ref{l33} (a) to conclude that $\Hom_{G_{r/2}}(L(\lambda), L(\mu) \otimes \opH^0(\nu))$ has trivial $G$-structure. 

Now, let $M:=\Ext_{G}^{1}(L(\lambda) \otimes V(\nu)^{(r/2)}, L(\mu) \otimes \opH^{0}(\nu))$ and we run the LHS spectral sequence corresponding to $G_{r/2} \lhd G$. First, we investigate the $E_2^{i,0}$-term and we get
\begin{equation*}
\begin{aligned}
E_2^{i,0} & \cong \Ext^i_{G/G_{r/2}}(V(\nu)^{(r/2)}, \Hom_{G_{r/2}}(L(\lambda), L(\mu) \otimes \opH^0(\nu))) \\
& \cong \Ext^i_G(V(\nu), k)\otimes \Hom_{G_{r/2}}(L(\lambda), L(\mu) \otimes \opH^0(\nu)).
\end{aligned}
\end{equation*}
By \cite[II.4.13]{Jan03}, $\Ext^i_G(V(\nu), k)=0$ for $i >0$, so we conclude that the $E_2^{i,0}$-terms all vanish. Hence $M \cong E_2^{0,1}$, giving \[R \cong \bigoplus_{\nu \in \Gamma^{\prime}} \Hom_{G / G_{r/2}}(V(\nu)^{(r/2)}, \Ext_{G_{r/2}}^{1}(L(\lambda), L(\mu) \otimes \opH^{0}(\nu))),\] the desired result. 

For (b), let $\lambda$ and $\mu$ be expressed as suggested. We apply the LHS spectral sequence corresponding to $G_{t} \lhd G$ to the terms in the first expression for $R$ in part (a). The $E_2$-page is given by 
\begin{displaymath}
E^{i,j}_2:= \Ext^i_{G/G_{t}}(L(\lambda_1)^{(t)} \otimes V(\nu)^{(r/2)}, \Ext^j_{G_{t}}(L(\lambda_0), L(\mu_0) \otimes \opH^0(\nu)) \otimes L(\mu_1)^{(t)}).
\end{displaymath}
First we consider the $E_2^{0,1}$-term.
\begin{align*}
    E_2^{0,1} & \cong \Hom_{G/G_{t}}(L(\lambda_1)^{(t)} \otimes V(\nu)^{(r/2)}, \Ext^1_{G_{t}}(L(\lambda_0), L(\mu_0) \otimes \opH^0(\nu)) \otimes L(\mu_1)^{(t)}) \\
& \cong \Hom_{G}(L(\lambda_1) \otimes V(\nu)^{(r/2-t)}, \Ext^1_{G_{t}}(L(\lambda_0), L(\mu_0) \otimes \opH^0(\nu))^{(-t)} \otimes L(\mu_1)).
\end{align*}
By \cite[(5.2.4)]{BNP06}, any weight $\gamma$ of $\Ext^1_{G_{t}}(L(\lambda_0), L(\mu_0) \otimes \opH^0(\nu))^{(-t)}$ satisfies $\langle \gamma, \alpha_{0}^{\vee}\rangle \leq \frac{2^{t}-1}{2^{t}}(h-1)+ \frac{\langle \nu, \alpha_{0}^{\vee}\rangle}{2^{t}}+\frac{3}{4} < h = 12.$ Assume without loss of generality that $\left\langle  \mu_1,  \alpha_{0}^{\vee}\right\rangle \leq \left\langle  \lambda_1,  \alpha_{0}^{\vee}\right\rangle$. Therefore, $E^{0,1}_2$ vanishes unless $\langle \tau^{r-2t} \nu, \alpha_{0}^{\vee}\rangle\leq \langle \gamma, \alpha_{0}^{\vee}\rangle$. We obtain $\langle \tau^{r-2t} \nu, \alpha_{0}^{\vee}\rangle=2^{s-t}\langle \tau \nu, \alpha_{0}^{\vee}\rangle \leq \langle \gamma, \alpha_{0}^{\vee}\rangle <12$. Assuming $\nu \neq 0$, we have $\langle \tau \nu, \alpha_{0}^{\vee}\rangle \geq 2$, so $E^{0,1}_2=0$, since $s-t \geq 3$. Thus, we have $E_2 \cong E^{1,0}_2$.

It remains to compute the $E_2^{1,0}$-term. We have
\begin{displaymath}
E_2^{1,0}  \cong \Ext^1_{G/G_{t}}(L(\lambda_1)^{(t)} \otimes V(\nu)^{(r/2)}, \Hom_{G_{t}}(L(\lambda_0), L(\mu_0) \otimes \opH^0(\nu)) \otimes L(\mu_1)^{(t)}) 
\end{displaymath}
By Lemma \ref{l33} (b), any weight $\gamma$ of $\Hom_{G_{t}}(L(\lambda_0), L(\mu_0) \otimes \opH^0(\nu))^{(-t)}$ satisfies $\langle \gamma, \alpha_{0}^{\vee}\rangle \leq \frac{1}{p^t}\langle  \nu, \alpha_{0}^{\vee}\rangle$. Now, since $t \geq 4$ and $\langle \nu, \alpha_{0}^{\vee}\rangle \leq 16$ by Lemma \ref{bnp52}, it follows that $\langle \gamma, \alpha_{0}^{\vee}\rangle \leq 1$. However, this forces $\langle \gamma, \alpha_{0}^{\vee}\rangle =0$, so $\Hom_{G_{t}}(L(\lambda_0), L(\mu_0) \otimes \opH^0(\nu))$ must have trivial $G$-structure. Thus 
\begin{equation*}
\begin{aligned}
E_2^{1,0} & \cong \Ext^1_{G/G_{t}}(L(\lambda_1)^{(t)} \otimes V(\nu)^{(r/2)}, L(\mu_1)^{(t)}) \otimes \Hom_{G_{t}}(L(\lambda_0), L(\mu_0) \otimes \opH^0(\nu)) \\
& \cong \Ext^1_G(L(\lambda_1) \otimes V(\nu)^{(r/2-t)}, L(\mu_1))\otimes \Hom_{G}(L(\lambda_0), L(\mu_0) \otimes \opH^0(\nu)).
\end{aligned}
\end{equation*}
This is the first reidentification. Now, consider $\Ext^1_G(L(\lambda_1) \otimes V(\nu)^{(r/2-t)}, L(\mu_1))$ and we apply the LHS spectral sequence corresponding to $G_{r/2-t} \lhd G.$ First, consider the $E_2^{i,0}$-term, for $i>0$. We obtain 
\begin{displaymath}
E_2^{i,0} \cong \Ext^i_{G/G_{r/2-t}}(V(\nu)^{(r/2-t)}, \Hom_{G_{r/2-t}}(L(\lambda_1), L(\mu_1))).
\end{displaymath}
Then, there are two possibilities -- either $\lambda_1 = \mu_1$ or not. If they are not equal, it follows that $\Hom_{G_{r/2-t}}(L(\lambda_1), L(\mu_1))$ automatically vanishes, so $E_2^{1,0}=E_2^{2,0}=0.$ If they are equal, $\Hom_{G_{r/2-t}}(L(\lambda_1), L(\lambda_1))$ has trivial $G$-structure, and, once again $E_2^{1,0}$ and $E_2^{2,0}$ vanish, as $\Ext^i_G(V(\nu), k)=0$, $i>0$ (cf. \cite[II.4.13]{Jan03}). We may now conclude that 
\begin{displaymath}
\Ext^1_G(L(\lambda_1) \otimes V(\nu)^{(r/2-t)}, L(\mu_1)) \cong \Hom_{G/G_{r/2-t}}(V(\nu)^{(r/2-t)}, \Ext^1_{G_{r/2-t}}(L(\lambda_1), L(\mu_1))),
\end{displaymath}
and this completes the proof.
\end{proof}

\begin{cor}\label{corR}
    With the hypothesis of the previous proposition, there exists an isomorphism $\Ext^1_{G(\sigma)}(L(\lambda), L(\mu)) \cong \Ext^1_G(L(\lambda), L(\mu))$ if either of the following hold:
    \begin{itemize}
        \item[(i)] $\Ext^1_{G_{r/2-t}}(L(\lambda_1), L(\mu_1))=0$ 
        \item[(ii)] $\Hom_G(L(\lambda_0), L(\mu_0) \otimes \opH^0(\nu))=0$, for all $\nu \in \Gamma^{\prime}$.
    \end{itemize}
\end{cor}

Next, we provide an analogue of \cite[Theorem~5.4]{BNP06} showing that generically, for the Ree groups of type $F_4$, self-extensions between simple modules vanish.

\begin{theorem}\label{self}Let $r=2s+1$ be odd with $s\geq 7$. Then \[\Ext^1_{G(\sigma)}(L(\lambda),L(\lambda))=0,\] for all $\lambda\in X_{\sigma}$.\end{theorem}

\begin{proof}
We know that self-extensions for classical Frobenius kernels vanish, as $G$ is not of type $C_n$  (cf. \cite[II.12.9]{Jan03}); hence $\Ext^1_{G_s}(L(\lambda),L(\lambda))=0$ for any $\lambda\in X_s$.

We aim to extend this result by replacing $s$ with $r/2$. When $r=1$ the result follows from \cite[1.7(1)(2),4.5]{Sin94b}. Suppose $r \neq 1$ and let $\lambda=\lambda_0+\tau^{r-1}\lambda_1=\lambda_0+2^s\lambda_1$ with $\lambda_1\in X_{\tau}$. We apply the LHS spectral sequence corresponding to $G_s\triangleleft G_{r/2}$. The $E_2$-page is given by
\begin{displaymath}
E^{i,j}_2= \Ext^i_{G_{r/2}/G_s}(L(\lambda_1)^{(s)}, \Ext^j_{G_s}(L(\lambda_0), L(\lambda_0)) \otimes L(\lambda_1)^{(s)}).
\end{displaymath}
First consider the $E^{1,0}$-term:
\begin{align*}
    E^{1,0}_2 &= \Ext^1_{G_{\tau}}(L(\lambda_1), \Hom_{G_s}(L(\lambda_0), L(\lambda_0))^{(-s)} \otimes L(\lambda_1)) \\
    & \cong \Ext^1_{G_{\tau}}(L(\lambda_1), L(\lambda_1))^{(s)}=0,
\end{align*}
by the discussion above. Now, we turn our attention to the $E^{0,1}$-term, which is isomorphic to 
\[\Hom_{G_{\tau}}(L(\lambda_1),\Ext^1_{G_s}(L(\lambda_0),L(\lambda_0))^{(-s)}\otimes L(\lambda_1))^{(s)}=0,\] by \cite[II.12.9]{Jan03}. Therefore, $\Ext^1_{G_{r/2}}(L(\lambda),L(\lambda))=0$, for any $\lambda \in X_{r/2}$.

Having evaluated the self-extensions for $G_{r/2}$, we now express $\lambda$ as $\lambda=\lambda_0+2^{t}\lambda_1$, with $\lambda_0\in X_{t}$ and $\lambda_1\in X_{r/2-t}$. Then, since $s \geq 7$ and $\Ext^1_{G_{r/2-t}}(L(\lambda_1),L(\lambda_1))=0$, we may apply Corollary \ref{corR}(i) and the claim follows.\end{proof}

Now, we note that in order to relate extensions between simple modules for the finite group to extensions between simple modules for the algebraic group, one needs, among other things, to bound the weights of the composition factors of $\Ext^1_{G_{7/2}}(L(\lambda),L(\mu))^{(-7/2)}$, for $\lambda,\mu \in X_{7/2}$. For the reader's convenience, we first present the following auxiliary lemma.
\begin{lemma}\label{NEW}
    Let $\lambda,\mu \in X_{7/2}$. Then any weight $\theta$ of $\Ext^1_{G_{7/2}}(L(\lambda),L(\mu))^{(-7/2)}$ satisfies $\left\langle \theta, \alpha_{0}^{\vee} \right\rangle \leq 6$.
\end{lemma}

\begin{proof}
First, we apply the LHS spectral sequence corresponding to $G_{\tau} \lhd G_{7/2}$ and we obtain
\begin{displaymath}
E^{i,j}_2:=\Ext^i_{G_3}(L(\lambda_1),\Ext^j_{G_{\tau}}(L(\lambda_0),L(\mu_0))^{(-\tau)} \otimes L(\mu_1))^{(\tau)}.
\end{displaymath}
Then, notice that for $\lambda_0=\mu_0$, we have $\Ext^1_{G_{\tau}}(L(\lambda_0),L(\lambda_0))=0$, so it follows that $E^{0,1}_2=0$; thus $E^1\cong E^{1,0}_2=\Ext^1_{G_3}(L(\lambda_1),L(\mu_1))^{(\tau)}$.

Otherwise, the case $\lambda_0 \neq \mu_0$ implies $E^{1,0}_2=E^{2,0}_2=0$ and therefore 
\begin{align*}
    E^1 & \cong E^{0,1}_2=\Hom_{G_3}(L(\lambda_1),\Ext^1_{G_{\tau}}(L(\lambda_0),L(\mu_0))^{(-\tau)} \otimes L(\mu_1))^{(\tau)} \\
    & \cong \Hom_{G_3}(L(\lambda_1), W_1 \otimes L(\mu_1))^{(\tau)},
\end{align*}
where $W_1 \in \{k,L(\omega_4), k \oplus L(\omega_1)\}$ (cf. \cite[4.5(a), 4.6, 4.9]{Sin94b}). Then, by Lemma \ref{l33}, it follows that $E^1  \cong \Hom_{G}(L(\lambda_1), W_1 \otimes L(\mu_1))$. Thus, a weight $\tau^7 \gamma$ of $\Ext^1_{G_{7/2}}(L(\lambda),L(\mu))$ is either zero, or is among the weights of $\Ext^1_{G_3}(L(\lambda_1), L(\mu_1))^{(\tau)}$.

Hence, we proceed by considering $M_1:=\Ext^1_{G_3}(L(\lambda_1), L(\mu_1))$ and we apply the spectral sequence corresponding to $G_1 \lhd G_3$. We obtain 
\begin{displaymath}
E^{i,j}_2:=\Ext^i_{G_2}(L(\lambda_{11}),\Ext^j_{G_1}(L(\lambda_{10}),L(\mu_{10}))^{(-1)} \otimes L(\mu_{11}))^{(1)}.
\end{displaymath}
A similar argument shows that the case $\lambda_{10}=\mu_{10}$ implies $M_1 \cong E^{1,0}_2=\Ext^1_{G_2}(L(\lambda_{11}),L(\mu_{11}))^{(1)}$.

Otherwise, if $\lambda_{10} \neq \mu_{10}$, we obtain
\begin{align*}
    M_1 & \cong \Hom_{G_2}(L(\lambda_{11}),\Ext^1_{G_1}(L(\lambda_{10}),L(\mu_{10}))^{(-1)} \otimes L(\mu_{11}))^{(1)} \\
    & \cong \Hom_{G_2}(L(\lambda_{11}), W_2 \otimes L(\mu_{11}))^{(1)},
\end{align*}
where $W_2 \in \{k, k \oplus k, L(\omega_4), k \oplus L(\omega_1), k \oplus L(\omega_4)\}$ (cf. \cite[4.11]{Sin94b}). In either case, $W_2$ has weights that are $2$-small, and it follows that by Lemma \ref{l33}, $M_1$ has trivial $G$-structure. Thus, it remains to consider the weights of $\Ext^1_{G_2}(L(\lambda_{11}), L(\mu_{11}))^{(1)}$.

Finally, let $M_2:=\Ext^1_{G_2}(L(\lambda_{11}), L(\mu_{11}))$ and run the LHS spectral sequence corresponding to $G_1 \lhd G_2$. We have 
\begin{displaymath}
E^{i,j}_2:=\Ext^i_{G_1}(L(\lambda_{111}),\Ext^j_{G_1}(L(\lambda_{110}),L(\mu_{110}))^{(-1)} \otimes L(\mu_{111}))^{(1)}.
\end{displaymath}
Analogously, if $\lambda_{110}=\mu_{110}$, then $M_2 \cong E^{1,0}_2=\Ext^1_{G_1}(L(\lambda_{111}),L(\mu_{111}))^{(1)}$.

Then, if $\lambda_{110} \neq \mu_{110}$, we have
\begin{align*}
    M_2 & \cong \Hom_{G_1}(L(\lambda_{111}),\Ext^1_{G_1}(L(\lambda_{110}),L(\mu_{110}))^{(-1)} \otimes L(\mu_{111}))^{(1)} \\
    & \cong \Hom_{G_1}(L(\lambda_{111}), W_3 \otimes L(\mu_{111}))^{(1)},
\end{align*}
with $W_3 \in \{k, k \oplus k, L(\omega_4), k \oplus L(\omega_1), k \oplus L(\omega_4)\}$ (cf. \cite[4.11]{Sin94b}). In fact, in order to bound the weights, it suffices to consider $\Hom_{G_1}(L(\lambda_{111}), L(\xi) \otimes L(\mu_{111}))$, for $\xi \in \{\omega_1,\omega_4\}$. 

Let $\lambda_{111}=\lambda_{1110}+\tau\overline{\lambda_{111}}$, with $\lambda_{1110} \in X_{\tau}$ and consider a similar expression for $\mu_{111}$. First, suppose $\xi=\omega_1$ and we have
\begin{align*}
\Hom_{G_1}&(L(\lambda_{111}),L(\omega_1)\otimes L(\mu_{111}))\\
&\cong \Hom_{G_1/G_{\tau}}(L(\overline{\lambda_{111}})^{(\tau)},\Hom_{G_{\tau}}(L(\lambda_{1110}),L(\mu_{1110}))\otimes L(\omega_4)^{(\tau)} \otimes L(\overline{\mu_{111}})^{(\tau)})\\
&\cong \Hom_{G_{\tau}}(L(\overline{\lambda_{111}}),L(\omega_4) \otimes L(\overline{\mu_{111}}))^{(\tau)},\end{align*}
if and only if $\lambda_{1110}=\mu_{1110}$. Thus, for $\gamma$ a weight of $\Hom_{G_{\tau}}(L(\overline{\lambda_{111}}),L(\omega_4) \otimes L(\overline{\mu_{111}}))^{(-\tau)}$, it follows that $\overline{\lambda_{111}}+\tau\gamma$ is a weight of $L(\omega_4) \otimes L(\overline{\lambda_{111}})$. Then, using \cite[Table 2]{Sin94a}, we may deduce the possible values of $\gamma$: $\omega_1$, $\omega_2$, $\omega_3$, $\omega_4$, $\omega_1+\omega_4$, $2\omega_4$. In particular, note that we have $\left\langle \gamma, \alpha_{0}^{\vee} \right\rangle \leq 4$.

Then, suppose $\xi=\omega_4$ and we obtain
\begin{align*}
\Hom_{G_1}&(L(\lambda_{111}),L(\omega_4)\otimes L(\mu_{111}))\\
&\cong \Hom_{G_{\tau}}(L(\overline{\lambda_{111}}),  \Hom_{G_{\tau}}(L(\lambda_{1110}),L(\omega_4) \otimes L(\mu_{1110}))^{(-\tau)} \otimes L(\overline{\mu_{111}}))^{(\tau)}.\end{align*}
In order to bound the weights, we must all of the possible weights of $\Hom_{G_{\tau}}(L(\lambda_{1110}),L(\omega_4) \otimes L(\mu_{1110}))^{(-\tau)}$, mentioned in the previous case, in turn and perform a similar calculation. For instance, let $\omega_1+\omega_4$ be such a weight and we have
\begin{align*}
    \Hom_{G_{\tau}}(L&(\overline{\lambda_{111}}), L(\omega_1+\omega_4) \otimes L(\overline{\mu_{111}}))^{(\tau)} \\
    &\cong  [\Hom_{G_{\tau}}(L(\overline{\lambda_{111}}), L(\omega_4) \otimes L(\overline{\mu_{111}})) \otimes L(\omega_4)^{(\tau)}]^{(\tau)}.
\end{align*}
Thus, any weight $\gamma$ of $\Hom_{G_1}(L(\lambda_{111}),L(\omega_4)\otimes L(\mu_{111}))^{(-1)}$ must be a weight of \linebreak$\Hom_{G_{\tau}}(L(\overline{\lambda_{111}}), L(\omega_4) \otimes L(\overline{\mu_{111}}))^{(-\tau)} \otimes L(\omega_4)$; now, since any weight of \linebreak $\Hom_{G_{\tau}}(L(\overline{\lambda_{111}}), L(\omega_4) \otimes L(\overline{\mu_{111}}))^{(-\tau)}$ is $4$-small, as seen in the case $\xi=\omega_1$, it follows that $\gamma$ must satisfy $\left\langle \gamma, \alpha_{0}^{\vee} \right\rangle \leq 4+\left\langle \omega_4, \alpha_{0}^{\vee} \right\rangle=6$. The other cases follow in a similar fashion.

By the discussion above, we may conclude that any weight $\theta$ of $\Ext^1_{G_{7/2}}(L(\lambda),L(\mu))^{(-7/2)}$ is either zero, or is among the weights of $\Ext^1_{G_{1}}(L(\lambda_{111}),L(\mu_{111}))^{(-1)}$, (namely $k$, $k \oplus k$, $L(\omega_4)$, $k \oplus L(\omega_4)$, $k \oplus L(\omega_1)$ -- so $2$-small), or, lastly, that it must be $6$-small. Thus, our claim follows.
\end{proof}

Finally, the following theorem relates extensions between simple $kG(\sigma)$-modules and extensions between simple $G$-modules.

\begin{theorem}\label{ext} Assume $r=2s+1$ with $s\geq 7$. Given $\lambda,\mu\in X_{\sigma}$, let \begin{align*}\lambda&=\sum_{i=0}^{r-1}\tau^i\lambda_{i/2}\\&=\lambda_0+\tau\lambda_{1/2}+2\lambda_1+\tau^3\lambda_{3/2}+\dots+2^s\lambda_{(r-1)/2}\end{align*} be the $\tau$-adic expansion of $\lambda$, and take a similar expression for $\mu$. Then there exists an integer $0\leq n<r$ such that \[\Ext^1_{G(\sigma)}(L(\lambda),L(\mu))\cong \Ext^1_G(L(\tilde\lambda),L(\tilde\mu))\]
where
\[\tilde\lambda=\sum_{i=0}^{n-1}\tau^i\lambda_{\frac{i+r-n}{2}}+\sum_{i=n}^{r-1}\tau^i\lambda_{\frac{i-n}{2}}.\]\end{theorem}
\begin{proof}
We express $\lambda$ and $\tilde\lambda$ in this way, motivated by the fact that $V^{(r/2)}\cong_{G(\sigma)} V$ for any $G(\sigma)$-module $V$. Hence, applying Steinberg's Tensor Product Theorem leads to the isomorphism $L(\tilde\lambda)\cong_{G(\sigma)} L(\lambda)^{(n/2)}$. 

By \cite[2.1(c)]{Sin94a} there is an injection $\Ext^1_G(L(\tilde\lambda),L(\tilde\mu))\hookrightarrow \Ext^1_{G(\sigma)}(L(\tilde\lambda),L(\tilde\mu))$ and since $\tau$ is an automorphism of $G(\sigma)$, we have $\Ext^1_{G(\sigma)}(L(\tilde\lambda),L(\tilde\mu))\cong \Ext^1_{G(\sigma)}(L(\lambda),L(\mu))$. Thus it suffices to show (by dimensions) that there is also an injection $\Ext^1_{G(\sigma)}(L(\tilde\lambda),L(\tilde\mu))\hookrightarrow \Ext^1_G(L(\tilde\lambda),L(\tilde\mu))$. 

First, suppose $\lambda=\mu$ and the claim follows from Theorem \ref{self} with $n=0$. 
Now assume $\lambda \neq \mu$. Then there exists $0\leq i\leq r$ such that $\lambda_{i/2}\neq \mu_{i/2}$. Due to the discussion above, we may choose the integer $n$ such that the differing digits in the $\tau$-adic expansion of $\tilde\lambda$ and $\tilde\mu$ are in a certain position, namely $\tilde\lambda_{\frac{2s-7}{2}}\neq \tilde\mu_{\frac{2s-7}{2}}$. Thus, put $n=2s-7-i$ if $i \leq 2s-7$ and $n=r+2s-7-i$ if $i \geq 2s-7$. Therefore, we write $\tilde\lambda=\lambda'+\tau^{2s-7}\lambda''+\tau^{2s-6}\lambda'''$ with $\lambda'\in X_{\frac{2s-7}{2}}$, $\lambda''=\tilde\lambda_{\frac{2s-7}{2}}$ and $\lambda'''\in X_{7/2}$, and take a similar expression for $\mu$.

Then, we apply Proposition \ref{identif}(b) with $t=s-3$. Thus \begin{displaymath}
	\mathrm{dim} \Ext_{G(\sigma)}^{1}(L(\tilde\lambda), L(\tilde\mu))  \leq \mathrm{dim} \Ext_{G}^{1}(L(\tilde\lambda), L(\tilde\mu)) + \mathrm{dim} R,
	\end{displaymath} where $R$ is isomorphic to \[\bigoplus_{\nu\in \Gamma^{\prime}}\Ext^1_G(L(\lambda''')\otimes V(\nu)^{(7/2)},L(\mu''')) \otimes \Hom_G(L(\lambda'+\tau^{2s-7}\lambda''),L(\mu'+\tau^{2s-7}\mu'')\otimes \opH^0(\nu)).\]
We turn our attention to $\Hom_G(L(\lambda'+\tau^{2s-7}\lambda''),L(\mu'+\tau^{2s-7}\mu'')\otimes \opH^0(\nu))$. We have
\begin{align*}
\Hom_G&(L(\lambda')\otimes L(\lambda'')^{(\frac{2s-7}{2})},L(\mu')\otimes L(\mu'')^{(\frac{2s-7}{2})}\otimes \opH^0(\nu))\\
&\cong \Hom_{G/G_{\frac{2s-7}{2}}}(L(\lambda'')^{(\frac{2s-7}{2})},\Hom_{G_{\frac{2s-7}{2}}}(L(\lambda'),L(\mu')\otimes \opH^0(\nu))\otimes L(\mu'')^{(\frac{2s-7}{2})}). \end{align*}
Consider $\tau^{2s-7}\theta$ a weight of $\Hom_{G_{\frac{2s-7}{2}}}(L(\lambda'),L(\mu')\otimes \opH^0(\nu))$ and by Lemma \ref{l33} (b), it follows that $\left\langle \tau^{2s-7}\theta, \alpha_{0}^{\vee} \right\rangle \leq \left\langle \nu, \alpha_{0}^{\vee} \right\rangle$. Thus $\left\langle \tau\theta, \alpha_{0}^{\vee} \right\rangle \leq \frac{\left\langle \nu, \alpha_{0}^{\vee} \right\rangle}{2^{s-4}} \leq 2$. By Lemma \ref{bnp52}, $\left\langle \tau\theta, \alpha_{0}^{\vee} \right\rangle \leq 2$ only when $\nu=8\omega_4$, in which case $\theta=0$ or $\theta=\omega_4$. Otherwise, if $\nu \neq 8\omega_4$, it follows that $\left\langle \nu, \alpha_{0}^{\vee} \right\rangle <16$ and, therefore, that $\Hom_{G_{\frac{2s-7}{2}}}(L(\lambda'),L(\mu')\otimes \opH^0(\nu))$ must have trivial $G$-structure.

Now, suppose $\Hom_{G_{\frac{2s-7}{2}}}(L(\lambda'),L(\mu')\otimes \opH^0(\nu))$ has trivial $G$-structure, and we may write
\begin{align*}
\Hom_G&(L(\lambda')\otimes L(\lambda'')^{(\frac{2s-7}{2})},L(\mu')\otimes L(\mu'')^{(\frac{2s-7}{2})}\otimes \opH^0(\nu))\\
&\cong \Hom_{G/G_{\frac{2s-7}{2}}}(L(\lambda'')^{(\frac{2s-7}{2})},\Hom_{G_{\frac{2s-7}{2}}}(L(\lambda'),L(\mu')\otimes \opH^0(\nu))\otimes L(\mu'')^{(\frac{2s-7}{2})})\\
&\cong \Hom_{G}(L(\lambda''),L(\mu''))\otimes \Hom_{G}(L(\lambda'),L(\mu')\otimes \opH^0(\nu)).\end{align*}
Since $\lambda''=\tilde\lambda_{\frac{2s-7}{2}}\neq \tilde\mu_{\frac{2s-7}{2}}=\mu''$, all of the corresponding summands of $R$ vanish.

It remains to consider the case in which $\tau^{2s-7}\omega_4$ is a potential weight of \linebreak $\Hom_{G_{\frac{2s-7}{2}}}(L(\lambda'),L(\mu')\otimes \opH^0(\nu))$, so we may assume $\nu=8\omega_4$. We obtain 
\begin{align*}
\Hom_G&(L(\lambda')\otimes L(\lambda'')^{(\frac{2s-7}{2})},L(\mu')\otimes L(\mu'')^{(\frac{2s-7}{2})}\otimes \opH^0(\nu))\\
&\cong \Hom_{G/G_{\frac{2s-7}{2}}}(L(\lambda'')^{(\frac{2s-7}{2})},\Hom_{G_{\frac{2s-7}{2}}}(L(\lambda'),L(\mu')\otimes \opH^0(\nu))\otimes L(\mu'')^{(\frac{2s-7}{2})})\\
&\cong \Hom_{G}(L(\lambda''),L(\omega_4) \otimes L(\mu'')).\end{align*}
Moreover, we know that $\lambda''=\tilde\lambda_{\frac{2s-7}{2}}\neq \tilde\mu_{\frac{2s-7}{2}}=\mu'' \in X_{\tau}$. Then, careful consideration using \cite[Table V]{Sin94b} shows that whenever $\lambda''-\mu''=\pm \omega_4$, we have $\Hom_{G}(L(\lambda''),L(\omega_4) \otimes L(\mu''))\neq 0$. 

Thus, since we cannot yet conclude that the summand of $R$ corresponding to $\nu=8\omega_4$ vanishes, we must turn our attention to $\Ext^1_G(L(\lambda''') \otimes V(8\omega_4)^{(7/2)}, L(\mu'''))$, for $\lambda''', \mu''' \in X_{7/2}$. We run the LHS spectral sequence corresponding to $G_{7/2} \lhd G$. First, consider the $E^{i,0}_2$-term for $i>0$:
\begin{displaymath}
E^{i,0}_2:=\Ext^1_{G/G_{7/2}}(V(8\omega_4)^{(7/2)}, \Hom_{G_{7/2}}(L(\lambda'''), L(\mu'''))). 
\end{displaymath}
Since $\Hom_{G_{7/2}}(L(\lambda'''), L(\mu'''))$ is either zero or has trivial $G$-structure, it follows that $E^{1,0}_2=E^{2,0}_2=0$, so we have $E^1_2 \cong E^{0,1}_2$. Thus 
\begin{align*}
\Ext^1_G&(L(\lambda''') \otimes V(8\omega_4)^{(7/2)}, L(\mu''')) \\
&\cong \Hom_{G}(V(8\omega_4), \Ext^1_{G_{7/2}}(L(\lambda'''), L(\mu'''))^{(-7/2)}).
\end{align*}
Then, notice that by Lemma \ref{NEW}, any weight $\theta$ of $ \Ext^1_{G_{7/2}}(L(\lambda'''), L(\mu'''))^{(-7/2)}$ must satisfy $\left\langle \theta, \alpha_{0}^{\vee} \right\rangle \leq 6$. Hence, $\Hom_{G}(V(8\omega_4), \Ext^1_{G_{7/2}}(L(\lambda'''), L(\mu'''))^{(-7/2)})=0$, so $\Ext^1_G(L(\lambda''') \otimes V(8\omega_4)^{(7/2)}, L(\mu'''))=0$. Thus, all of the summands vanish, giving $R=0$ and the claim follows.
\end{proof}

\begin{flushleft}

\end{flushleft}
\end{document}